\documentclass[times]{nmeauth}

\usepackage{moreverb}

\usepackage[pdftex,colorlinks,bookmarksopen,bookmarksnumbered,citecolor=red,urlcolor=red]{hyperref}

\usepackage{graphicx,xcolor}
\usepackage{amsmath,amssymb,amsfonts}
\usepackage{nccmath}
\usepackage{fancybox}
\usepackage{subfigure}
\usepackage{float} 
\usepackage{multirow}
\usepackage{bbdd} % insertion de lettres avec double barre
\usepackage{xspace}
\usepackage{xfrac}
\usepackage{upgreek}
\usepackage{pgfplots}
\pgfplotsset{compat=newest}
\usepackage{mathtools}
\usepackage{mynotations}

\theoremstyle{plain}% default

\newtheorem{corollary}{\bfseries \rmfamily Corollary}

\newtheorem{proposition}{\bfseries \rmfamily Proposition}

\newtheorem{definition}{\bfseries \rmfamily Definition}

\newtheorem{remark}{\itshape \rmfamily Remark}

\newcommand{\Figure}{Figure}

\newcommand{\Table}{Table}

\newcommand{\Section}{Section}
\newcommand{\Sections}{Sections}
\newcommand{\Appendix}{Appendix}

\newcommand{\Proposition}{Proposition}

\newcommand{\Corollary}{Corollary}

\newcommand{\Remark}{Remark}

\newcommand{\Fig}[1]{\Figure~\ref{#1}}

\newcommand{\Tab}[1]{\Table~\ref{#1}}

\newcommand{\Sect}[1]{\Section~\ref{#1}}
\newcommand{\Sects}[1]{\Sections~\ref{#1}}

\newcommand{\Prop}[1]{\Proposition~\ref{#1}}

\newcommand{\Cor}[1]{\Corollary~\ref{#1}}

\newcommand{\Rem}[1]{\Remark~\ref{#1}}

\newcommand{\ie}{i.e.\xspace}

\newcommand\BibTeX{{\rmfamily B\kern-.05em \textsc{i\kern-.025em b}\kern-.08em
T\kern-.1667em\lower.7ex\hbox{E}\kern-.125emX}}

\def\volumeyear{2012}

\begin{document}

\runningheads{P.\ Ladev\`eze, F.\ Pled and L.\ Chamoin}{New bounding techniques for goal-oriented error estimation}

\title{New bounding techniques for goal-oriented error estimation applied to linear problems}
%\footnotemark[2]
\author{P.~Ladev\`eze\affil{1}\comma\corrauth, F.~Pled\affil{1} and L.~Chamoin\affil{1}}

\address{\affilnum{1}LMT-Cachan (ENS-Cachan/CNRS/Paris 6 University), 61 Avenue du Pr\'esident Wilson, 94235 CACHAN Cedex, FRANCE}

\corraddr{P. Ladev\`eze, LMT-Cachan (ENS-Cachan/CNRS/Paris 6 University), 61 Avenue du Pr\'esident Wilson, 94235 CACHAN Cedex}

%\cgs{<Contract/grant sponsor name (no number)>}
%\cgsn{<Contract/grant sponsor name>}{<number>}

%: abstract
\begin{abstract}

The paper deals with the accuracy of guaranteed error bounds on outputs of interest computed from approximate methods such as the finite element method. A considerable improvement is introduced for linear problems thanks to new bounding techniques based on Saint-Venant's principle. The main breakthrough of these optimized bounding techniques is the use of properties of homothetic domains which enables to cleverly derive guaranteed and accurate boundings of contributions to the global error estimate over a local region of the domain. Performances of these techniques are illustrated through several numerical experiments.

\end{abstract}


\maketitle

%\footnotetext[2]{Please ensure that you use the most up to date
%class file,
%available from the NME Home Page at\\
%\href{http://www3.interscience.wiley.com/journal/1430/home}{\texttt{http://www3.interscience.wiley.com/journal/1430/home}}}

%: 1. Introduction
\section{Introduction}\label{1}

In the context of finite element model verification, research works currently focus on the evaluation of the numerical quality of specific quantities of practical interest. Such worthwhile methods, known as \textit{goal-oriented error estimation methods}, have been emerging for about twenty years \cite{Bab84,Par97,Ran97,Cir98,Pru99,Lad99a,Str00,Bec01} and have been recently extended to a wide range of mechanical problems \cite{Par06bis,Gal06,Cha07,Cha08,Lad08,Lad09,Pan10,Lad10}. Nevertheless, among all of them, only a few actually lead to robust and relevant estimates ensuring the recovery of strict and high-quality error bounds.

In order to achieve robust goal-oriented error estimation, a general method, initially introduced in \cite{Lad08}, has been prone to considerable developments. It is based on classical and powerful tools, such as the concept of constitutive relation error (CRE), and more recently handbook techniques \cite{Cha08,Cha09} and projection procedures \cite{Lad10}. First,
methods based on the CRE enable to set up a guaranteed global error estimate through the construction of an admissible pair, which constitutes the key technical point. Various techniques enable to construct such an admissible solution from the prescribed data of the reference problem and the approximate finite element (FE) solution (see \cite{Lad04,Par06,Moi09,Cot09,Gal09,Lad10bis,Ple11,Ple12} for further information). Then, a measure of the non-verification of the constitutive relation by this admissible pair (with respect to an energy norm) leads to an upper bound of the global discretization error. The goal-oriented error estimator considered in this work is based on extraction (or adjoint-based) techniques and involves the solution of an auxiliary problem, often referred to as dual (or adjoint in the linear case, or mirror in the non-linear case \cite{Lad08}) problem. An accurate solution of this auxiliary problem is required to achieve sharp local error bounds. A natural but intrusive way to properly solve this problem merely consists of performing a local space(-time) refinement of the mesh being considered. Alternatively, handbook techniques, initially proposed in \cite{Str00bis}, rely on a local enrichment of the solution of the auxiliary problem introduced through a partition of unity method (PUM). The enrichment functions correspond to (quasi-)exact local solutions of the auxiliary problem calculated either analytically or numerically in (semi-)infinite domains. This enrichment is particularly well-suited to handle truly pointwise quantities of interest without resorting to any regularization of the quantity of interest being studied. In practice, this technique preserves the non-intrusive nature offered by the FE framework as it circumvents the need to perform local refinement of the auxiliary problem to provide precise error bounds. Lastly, projection procedures allow the treatment of non-linear quantities of interest without having recourse to any classical linearization techniques and keeping intact the strict nature of the resulting bounds.

However, among those various tools, only the use of handbook techniques leads to the derivation of outstandingly accurate local error bounds, provided that appropriate enrichment functions are available in a library of pre-computed handbooks functions. Indeed, without additional techniques, the classical bounding technique may achieve low-quality guaranteed error bounds even with an extremely refined mesh. The reason why such error estimates are crude comes from the basic bounding technique, \ie the Cauchy-Schwarz inequality, which is only efficient if terms involved in this inequality are close to be collinear, \ie when the main contributions to the discretization errors related to both reference and auxiliary problems are located in the same regions. In other words, when the zone of interest does not coincide with the most concentrated error regions associated to the reference problem, the classical bounding technique yields inaccurate and sometimes useless bounds.

The present work sheds some light on this crucial question. In this paper, we revisit the classical bounding technique based on global error estimation methods applied to both reference and auxiliary problems, and introduce two new bounding techniques to alleviate problems related to the classical bounding technique. These improved techniques are based on both classical and innovative tools; they lean on Saint-Venant's principle and are therefore restricted to linear problems. Basic extraction techniques and specific homotheticity properties are employed to get guaranteed and relevant bounds of better quality than the classical bounding technique. More precisely, the main idea consists of considering separately the zone of interest and the remainder of the structure. Both enhanced techniques use homothetic domains properties to cleverly derive accurate bounds over a local region surrounding the zone of interest. In this work, we examine only the case of linear quantities of interest with respect to the displacement field associated to linear elasticity problems. Finally, these new basic bounding techniques are combined with handbook techniques to solve in a very accurate manner the adjoint problem, thus resulting in efficient error bounds. 2D applications are performed and show the important resulting gain on the quality of the computed bounds using the two new basic bounding techniques.

The paper is organized as follows: after this introduction, \Sect{2} presents both reference and adjoint problems and defines the discretization error; \Sect{3} recalls basics on goal-oriented error estimation using the concept of admissible solutions and associated constitutive relation error; \Sects{4} and \ref{5} provide a detailed description of the first and second improved bounding techniques, respectively, both based on homothetic domains properties; several numerical experiments are presented in \Sect{6} to demonstrate the effectivity and relevance of the two proposed techniques; eventually, \Sect{7} draws some conclusions and may provide road maps to future developments.

%: 2.
\section{Reference and adjoint problems}\label{2}

%: 2.1
\subsection{Reference problem and discretization error}\label{2.1}

Let us consider a mechanical structure occupying an open bounded domain $\Om$, with Lipschitz boundary $\dOm$. The prescribed loading acting on $\Om$ consists of: a displacement field $\Uu_d$ on part $\partial_1 \Om \subset \dOm$ $(\partial_1 \Om \neq \emptyset)$; a traction force density $\und{F}_d$ on the complementary part $\partial_2 \Om$ of $\dOm$ such that $\overline{\partial_1 \Om \cup \partial_2 \Om} = \dOm$, $\partial_1 \Om \cap \partial_2 \Om = \emptyset$; a body force field $\und{f}_d$ within $\Om$. Structure $\Om$ is assumed to be made of a material with isotropic, homogeneous, linear and elastic behavior characterized by Hooke's tensor $\K$. Assuming a quasi-static loading as well as isothermal and small perturbations state, the reference problem which describes the behavior of the structure consists of finding a displacement/stress pair $(\uu, \cont)$ in the space domain $\Om$, which verifies:
\begin{itemize}
\item 
the kinematic conditions:
\begin{multline}\label{eq1:CAref}
\shoveright{\uu \in \Ucb; \quad \uu_{\restrictto{\partial_1 \Om}} = \Uu_d; \quad \defo(\uu) = \frac{1}{2}\big(\nabla \uu + \nabla^T \uu \big);}
\end{multline}
\item
the weak form of equilibrium equations:
\begin{multline}\label{eq1:SAref}
\shoveright{\cont \in \Scb; \quad \forall \: \uu^{\ast} \in \Ucb_0, \quad \intO \Tr\big[\cont \: \defo(\uu^{\ast})\big] \dO = \intO \und{f}_d \cdot \uu^{\ast} \dO + \int_{\partial_2 \Om}\und{F}_d \cdot \uu^{\ast} \dS;}
\end{multline}
\item
the constitutive relation:
\begin{multline}\label{eq1:RDCref}
\shoveright{\cont(M) = \K \: \defo\big(\uu(M)\big) \quad \forall \: M \in \Om,}
\end{multline}
\end{itemize}
where $\defo(\uu)$ represents the classical linearized strain tensor corresponding to the symmetric part of the gradient of displacement field $\uu$. Affine spaces $\Ucb = \left\{ \und{u} \in [\Hc^1(\Om)]^3 \right\}$ and $\Scb = \left\{ \cont \in \Mc_s(3) \cap [\Lc^2(\Om)]^6 \right\}$ guarantee the existence of finite-energy solutions, $\Mc_s(n)$ representing the space of symmetric square matrices of order $n$. Lastly, $\Ucb_0 \subset \Ucb$ denotes the vectorial space associated to $\Ucb$, \ie containing the functions subjected to homogeneous kinematic boundary conditions over $\partial_1 \Om$.

In practical applications, the exact solution of the reference problem, hereafter denoted $(\uu, \cont)$, remains usually out of reach and only an approximate solution, referred to as $(\uu_h, \cont_h)$, can be obtained through numerical approximation methods (such as the finite element method (FEM) associated with a space mesh $\Mc_h$ mapping $\Om$). Such a numerical approximation is searched in a discretized space $\Ucb_{h} \times \Scb_{h} \subset \Ucb \times \Scb$. A displacement-type FEM leads to a displacement field $\uu_h$ verifying kinematic constraints (\ref{eq1:CAref}) and a stress field $\cont_h$ computed \textit{a posteriori} from constitutive relation (\ref{eq1:RDCref}).

The resulting discretization error, denoted $\und{e}_h = \uu - \uu_h$, can be evaluated by means of:
\begin{itemize}
\item a global measure, such as the classical energy norm $\lnorm{\bullet}_{u, \Om} = \left( \intO \Tr\big[\K \: \defo(\bullet) \: \defo(\bullet)\big] \dO \right)^{1/2}$, providing a global discretization error $e^{\glob}_h = \lnorm{\und{e}_h}_{u, \Om}$.
\item a local measure defined with respect to a specified output of interest $I(\uu)$ of the problem, providing a local error $e^{\loc}_h = I(\uu) - I(\uu_h)$. Under the assumption of a linear quantity of interest with respect to displacement $\uu$, it merely reads: $e^{\loc}_h = I(\und{e}_h)$.
\end{itemize}

%: 2.2
\subsection{Adjoint problem}\label{2.2}

The quantity of interest, hereafter denoted $I$, is a goal-oriented output, such as the mean value of a stress component over a local region or the displacement value at a specific point, for instance. These meaningful quantities of engineering interest are usually defined by means of extraction techniques \cite{Par97,Str00,Ohn01}, \ie by expressing the local quantity $I$ being considered in the global form involving global extraction operators, also called extractors. In this work, for the sake of simplicity, the quantity of interest is represented as a linear functional $\Lc$ of displacement field $\uu$ on a finite support under the following global form:
\begin{align}\label{eq1:interestquantity}
I = \Lc(\uu) = \intO \left( \Tr\big[\tilde{\cont}_{\Sigma} \: \defo(\uu)\big] + \tilde{\und{f}}_{\Sigma} \cdot \uu \right) \dO,
\end{align}
where so-called extractors $\tilde{\cont}_{\Sigma}$ and $\tilde{\und{f}}_{\Sigma}$, known analytically, can be mechanically viewed as a prestress field and a body force field, respectively. These extractors can be defined explicitly or implicitly depending on the selected output of interest. Let us note that the extraction technique provides a natural framework to handle a wide range of local quantities. In the following, let $\Iex = \Lc(\uu)$ and $\Ih = \Lc(\uu_h)$ be the unknown exact value of the quantity of interest $I$ being studied and its approximate value obtained through the FEM, respectively.

Once the quantity of interest has been put into such a global form, the classical approach then consists of introducing an auxiliary problem, also called adjoint problem, which is similar to the reference problem, except that the external mechanical loading $(\und{F}_d, \und{f}_d)$ is replaced by the extractors on the one hand, and the non-homogeneous Dirichlet boundary conditions are changed to homogeneous kinematic constraints on the other hand. The adjoint problem consists of finding a displacement/stress pair $(\tilde{\uu} ,\tilde{\cont})$, in the space domain $\Om$, which verifies:
\begin{itemize}
\item 
the kinematic conditions:
\begin{multline}\label{eq1:CAadjoint}
\shoveright{\tilde{\uu} \in \Ucb_0;}
\end{multline}
\item
the weak form of equilibrium equations:
\begin{multline}\label{eq1:SAadjoint}
\shoveright{\tilde{\cont} \in \Scb; \quad \forall \: \uu^{\ast} \in \Ucb_0, \quad \intO \Tr\big[\tilde{\cont} \: \defo(\uu^{\ast})\big] \dO = \Lc(\uu^{\ast}) = \intO \left( \Tr\big[\tilde{\cont}_{\Sigma} \: \defo(\uu^{\ast})\big] + \tilde{\und{f}}_{\Sigma} \cdot \uu^{\ast} \right) \dO;}
\end{multline}
\item
the constitutive relation:
\begin{multline}\label{eq1:RDCadjoint}
\shoveright{\tilde{\cont}(M) = \K \: \defo\big(\tilde{\uu}(M)\big) \quad \forall \: M \in \Om.}
\end{multline}
\end{itemize}

Under the assumption of a linear functional $\Lc$, the following equality holds:
\begin{align}\label{eq1:influencefunction}
\Iex - \Ih = \Lc(\uu) - \Lc(\uu_h) = \Lc(\und{e}_h) = \intO \Tr\big[\K \: \defo(\tilde{\uu}) \: \defo(\und{e}_h)\big] \dO = \Rc_h(\tilde{\uu}).
\end{align}
The solution $\tilde{\uu}$ of the adjoint problem can thus be viewed as an influence function \cite{Pru99} indicating how the weak residual functional $\Rc_h$ (with respect to the discretization error $\und{e}_h$) affects the local discretization error $\Lc(\und{e}_h)$ (with respect to the specific measure $\Lc$). Let us note that solutions $\uu$ and $\tilde{\uu}$ of reference and adjoint problems, respectively, are mutually adjoint to each other \cite{Bec01} insofar as 
\begin{align}\label{eq2:influencefunction}
\Lc(\uu) = \intO \Tr\big[\defo(\uu) \: \K \: \defo(\tilde{\uu})\big] \dO = \Fc(\tilde{\uu}),
\end{align}
where
\begin{align}\label{eq3:influencefunction}
\Fc(\tilde{\uu}) = \intO \und{f}_d \cdot \tilde{\uu} \dO + \int_{\partial_2 \Om}\und{F}_d \cdot \tilde{\uu} \dS.
\end{align}

As for the reference problem, the exact solution $(\tilde{\uu} ,\tilde{\cont})$ of the adjoint problem remains out of reach in most practical applications, and one can only obtain an approximate solution, denoted $(\tilde{\uu}_h ,\tilde{\cont}_h)$. This last solution lies in discretized FE spaces associated with a space mesh $\tilde{\Mc}_h$, mapping the physical domain $\Om$, regardless of the FE mesh $\Mc_h$ used to solve the reference problem.

%: 3.
\section{Basics on goal-oriented error estimation based on constitutive relation error}\label{3}

We review here the classical procedure to get guaranteed local error bounds on functional outputs which constitute valuable and relevant information in standard engineering practice.

%: 3.1
\subsection{Constitutive relation error}\label{3.1}

In verification research activities, setting up robust error estimation methods has become an overriding concern. The construction of what is called an \textit{admissible pair} is currently an essential and crucial step in order to obtain guaranteed error bounds. An admissible pair $(\hat{\uu}_h, \hat{\cont}_h)$ verifies all the equations of the reference problem, apart from constitutive relation (\ref{eq1:RDCref}). On the one hand, a kinematically admissible displacement field is generally obtained by merely taking $\hat{\uu}_h$ equal to $\uu_h$ (apart from the case of incompressible materials, see \cite{Lad92}). On the other hand, the derivation of a statically admissible stress field can be achieved by using various balance techniques suitable to error estimation \cite{Lad04,Par06,Moi09,Cot09,Gal09,Lad10bis,Ple11,Ple12}. Such an admissible stress field $\hat{\cont}_h$ can be recovered from the data and the FE stress field $\cont_h$ alone. Starting from an admissible solution $(\hat{\uu}_h, \hat{\cont}_h)$ provided by one of the existing techniques, one can measure the global residual on constitutive relation (\ref{eq1:RDCref}), called the constitutive relation error (CRE) and hereafter referred to as $e_{\cre} \equiv e_{\cre}(\hat{\uu}_h, \hat{\cont}_h) = \lnorm{\hat{\cont}_h - \K \: \defo(\hat{\uu}_h)}_{\cont,\Om}$, with $\lnorm{\bullet}_{\cont, \Om} = \left( \intO \Tr\big[ \bullet \: \K^{-1} \: \bullet\big] \: \dO \right)^{1/2}$. Computing the CRE $e_{\cre}(\hat{\uu}_h, \hat{\cont}_h)$ provides a guaranteed upper bound of the global discretization error $\lnorm{\und{e}_h}_{u, \Om}$, as the well-known Prager-Synge hypercircle theorem \cite{Pra47} leads to the following bounding inequality:
\begin{align}\label{eq1:PragerSynge}
\lnorm{\und{e}_h}^2_{u, \Om} = \lnorm{\uu - \hat{\uu}_h}^2_{u, \Om} \leqslant \lnorm{\uu - \hat{\uu}_h}^2_{u, \Om} + \lnorm{\cont - \hat{\cont}_h}^2_{\cont,\Om} = e^2_{\cre},
\end{align}
which conveys the guaranteed nature of the CRE $e_{\cre}$.

Introducing the average admissible field:
\begin{align}\label{eq1:sigmahathm}
\hat{\cont}^m_h = \frac{1}{2} \left(\hat{\cont}_h + \K \: \defo(\hat{\uu}_h)\right),
\end{align}
one can directly deduce another fundamental relation, called the Prager-Synge's equality:
 \begin{align}\label{eq1:sigmahathmecre}
\lnorm{\cont - \hat{\cont}^m_h}_{\cont, \Om} = \frac{1}{2} e_{\cre}.
\end{align}
Equations (\ref{eq1:PragerSynge}) and (\ref{eq1:sigmahathmecre}) are key relations to derive guaranteed error bounds in both global and local robust error estimation methods.

In the same way as for the reference problem, an admissible solution of the adjoint problem, hereafter referred to as $(\hat{\tilde{\uu}}_h, \hat{\tilde{\cont}}_h)$, can be derived from one of the existing equilibration techniques. Then, the associated CRE $\tilde{e}_{\cre} \equiv e_{\cre}(\hat{\tilde{\uu}}_h, \hat{\tilde{\cont}}_h)$ of the adjoint problem can be computed leading to a global estimate of the discretization error $\tilde{\und{e}}_h = \tilde{\uu} - \tilde{\uu}_h$ of the adjoint problem.

Now, let us focus on the main principles of the classical bounding technique involved in goal-oriented error estimation method based on extraction techniques and CRE.

%: 3.2
\subsection{Basic identity and classical bounding technique}\label{3.2}

The definition of the quantity of interest $I$ recast in the global form (\ref{eq1:interestquantity}) and properties of both admissible solutions $(\hat{\uu}_h, \hat{\cont}_h)$ and $(\hat{\tilde{\uu}}_h, \hat{\tilde{\cont}}_h)$ lead to the following basic identity:
\begin{align}
\Iex - \Ih - \Ihh & = \intO \Tr\big[(\cont - \hat{\cont}_h^m) \: \K^{-1} \: (\hat{\tilde{\cont}}_h - \K \: \defo(\hat{\tilde{\uu}}_h)) \big] \dO \nonumber\\
& = \lscalproda{\cont - \hat{\cont}_h^m}{\hat{\tilde{\cont}}_h - \K \: \defo(\hat{\tilde{\uu}}_h)}_{\cont, \Om},\label{eq1:basicidentity}
\end{align}
with $\displaystyle\hat{\cont}_h^m = \frac{1}{2} (\hat{\cont}_h + \K \: \defo(\hat{\uu}_h))$; $\lscalproda{\bullet}{\circ}_{\cont, \Om}$ denotes an energetic inner product defined on the stress field space over $\Om$. $\Ihh$ can be viewed as a computable correction term involving known quantities of both reference and adjoint problems:
\begin{align}
\Ihh & = \displaystyle \intO \Tr\big[\hat{\tilde{\cont}}^m_h \: \K^{-1} \: (\hat{\cont}_h - \K \: \defo(\hat{\uu}_h)) \big] \dO + \Lc(\hat{\uu}_h - \uu_h) \nonumber\\
& = \lscalproda{\hat{\tilde{\cont}}^m_h}{\hat{\cont}_h - \K \: \defo(\hat{\uu}_h)}_{\cont, \Om} + \Lc(\hat{\uu}_h - \uu_h), \label{eq1:Ihh}
\end{align}
with $\displaystyle\hat{\tilde{\cont}}^m_h = \frac{1}{2} (\hat{\tilde{\cont}}_h + \K \: \defo(\hat{\tilde{\uu}}_h))$, and leading to a new approximate solution $\Ih + \Ihh$ of the exact value $\Iex$ of the quantity of interest. A complete and detailed proof of this basic identity can be found in \cite{Lad08,Lad10}.

The fundamental equality (\ref{eq1:basicidentity}), which does not require any orthogonality property of the FE solutions (contrary to pioneering techniques \cite{Par97,Pru99,Bec01}) and allows to decouple discretizations of reference and adjoint problems, is the cornerstone of the classical bounding technique as well as the improved ones described in section \ref{4}. Besides, this bounding technique could conceivably be extended to problems solved using numerical approximation methods different from the FEM.

Subsequently, the classical bounding procedure merely consists of applying the Cauchy-Schwarz inequality to (\ref{eq1:basicidentity}) with respect to inner product $\lscalproda{\bullet}{\circ}_{\cont, \Om}$ and then using Prager-Synge's equality (\ref{eq1:sigmahathmecre}). This yields:
\begin{align}\label{eq1:localerrorbounding}
\labs{\Iex - \Ih - \Ihh} \leqslant \frac{1}{2} \: e_{\cre} \: \tilde{e}_{\cre}.
\end{align}
Subsequently, the derivation of strict lower and upper bounds $(\xiinf,\xisup)$ of $\Iex$ (or, equivalently, of the local error $\Iex - \Ih$) can be achieved straightforwardly, just having a global error estimation procedure at hand:
\begin{align}\label{eq1:classicallocalerrorbounding}
\xiinf \leqslant \Iex \leqslant \xisup,
\end{align}
with
\begin{align}\label{eq1:classicallocalerrorbounds}
& \xiinf = \Ih + \Ihh - \frac{1}{2} \: e_{\cre} \: \tilde{e}_{\cre};\\
& \xisup = \Ih + \Ihh + \frac{1}{2} \: e_{\cre} \: \tilde{e}_{\cre}.
\end{align}

Besides, owing to the independent natures of spatial discretizations associated to reference and adjoint problems, a convenient way to achieve accurate and sharp bounds of $\Iex$ is to perform a local space refinement of the adjoint mesh $\tilde{\Mc}_h$ alone around the zone of interest in order to properly solve the adjoint problem while keeping a reasonable computational cost. In most practical applications, the discretization error related to the adjoint problem is concentrated in the vicinity of the zone of interest, whereas that related to the reference problem may be scattered around zones which present some singularities or other error sources. However, when the error related to the reference problem is mostly located outside and far from the zone of interest, the classical bounding technique may yield large and low-quality local error bounds and thus makes useless bounding result (\ref{eq1:classicallocalerrorbounding}). This is the point that we are revisiting here.

The proposed bounding techniques we present in the two following sections are intended to circumvent this serious drawback inherent to the classical technique in order to get sharp local error bounds.

%: 4.
\section{First improved bounding technique}\label{4}

%4.1
\subsection{Definitions and preliminaries}\label{4.1}

Let us consider a reference subdomain, denoted $\om_1$ and included in $\Om$, defined by a point $O$ and a geometric shape. The set of homothetic domains $\omla$ is defined as:
\begin{align}
\omla = \Hc[O;\la](\om_1)
\end{align}
where $\Hc[O;\la]$ stands for the homothetic transformation centered in point $O$, called homothetic center, and parameterized by a nonzero positive number $\la \in \intervalcc{0}{\la_{\max}}$, also called magnification ratio, scale factor or similitude ratio, such that $\omla \subset \Om$, \ie $\omla$ is a subset of $\Om$ (see \Fig{fig1:homothetic_domains}). The geometric shape defining the set of homothetic domains $\omla$ is arbitrary, but in practice these physical domains are assumed to be basic, such as a circle or a rectangle in 2D, and a sphere or a rectangular cuboid (also called rectangular parallelepiped or right rectangular prism) in 3D, for instance.

By considering the parameterization $(\la, s)$ (resp. $(\la, s_1, s_2)$) of a given domain $\omla \subset \Om$ in 2D (resp. 3D), where $s$ denotes the curvilinear abscissa along boundary $\domla$ (see \Fig{fig1:homothetic_domains}), the following equalities hold:
\begin{align}
\intola \bullet \dO & = \int_{\la'=0}^{\la} \left[ \intdolaprime \bullet \dS \right] \dla'; \label{eq1:linkintolaintolaprime}\\
\frac{\diff}{\dla} \left[ \intola \bullet \dO \right] & = \intdola \bullet \dS. \label{eq1:diffintola}
\end{align}
where $\dS$ is defined in a generic manner as:
\begin{equation}\label{de1:dS}
\dS = \begin{cases}
\ds & \text{in 2D}; \\
a(s_1,s_2) \ds_1 \ds_2 & \text{in 3D},
\end{cases}
\end{equation}
and depends on the geometric shape defining the set of homothetic domains.

\begin{figure}
\centering\includegraphics[scale = 0.4]{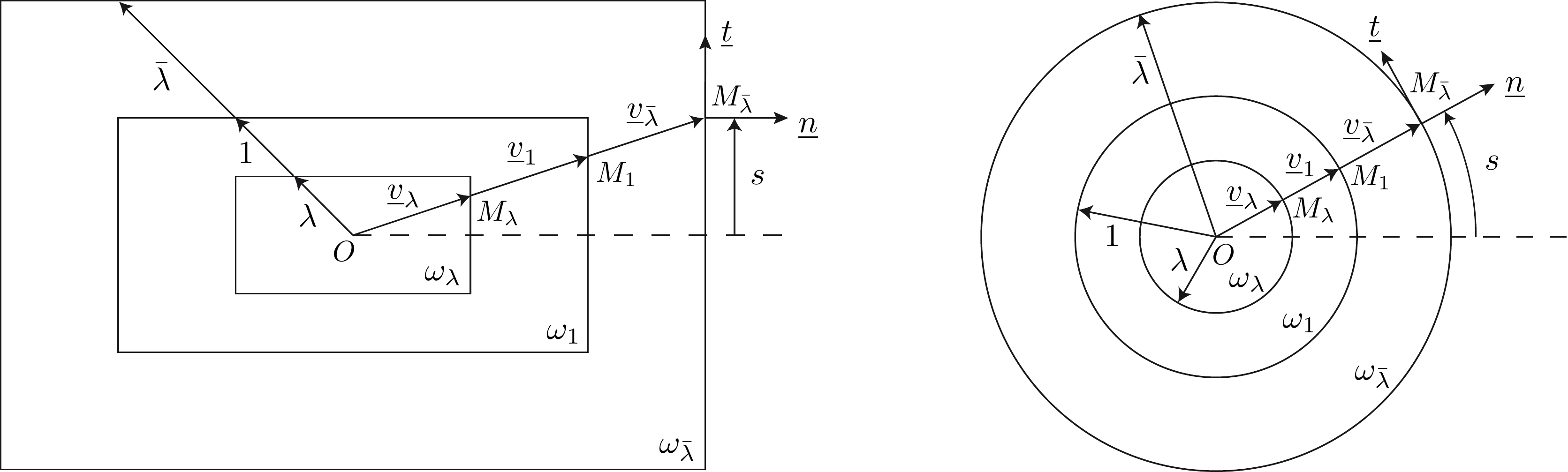}
\caption{Rectangular (left) and circular (right) homothetic domains in two dimensions.}\label{fig1:homothetic_domains}
\end{figure}
For a given pair $(\omla,\omlabar)$ of homothetic domains included in $\Om$, represented in \Fig{fig1:homothetic_domains} and parameterized by $(\la, \labar)$, such that $\omla \subset \omlabar \subset \Om$, \ie $\la \in \intervaloc{0}{\labar}$, the position $\vu_{\la}$ of a point $M_{\la}$ along $\domla$ can be defined from the position $\vu_{\labar}$ of the corresponding point $M_{\labar}$ along $\domlabar$ by the following relation:
\begin{align}
\vu_{\la} = \begin{dcases}
\displaystyle\frac{\la}{\labar} \: \vu_{\labar}(s) \quad \text{parameterized by} \ (\la,s) & \text{in 2D}; \\
\displaystyle\frac{\la}{\labar} \: \vu_{\labar}(s_1,s_2) \quad \text{parameterized by} \ (\la,s_1,s_2) & \text{in 3D},
\end{dcases}
\end{align}
where $s$ (resp. $s_1$ and $s_2$) represent the curvilinear abscissa along boundary $\domlabar$ in 2D (resp. 3D).

Such a parameterization leads to the following relations:
\begin{align}
\intola \bullet \dO & = \int_{\la'=0}^{\la} \left[ \intdolabar \bullet \: \vu_{\labar} \cdot \und{n} \dS \right] {\left(\frac{\la'}{\displaystyle\labar}\right)}^n \frac{1}{\displaystyle\labar} \dla'; \label{eq1:linkintolaintdolabar}\\
\intdola \bullet \dS & = \intdolabar \bullet \dS \: {\left(\frac{\la}{\displaystyle\labar}\right)}^n, \label{eq1:linkintdolaintdolabar}
\end{align}
where $n$ denotes a positive integer defined as:
\begin{equation}\label{def1:n}
n = \begin{cases}
1 & \text{in 2D}; \\
2 & \text{in 3D}.
\end{cases}
\end{equation}

Let us define inner products and associated norms over a given homothetic domain $\omla \subset \Om$:
\begin{align}
\lscalproda{\bullet}{\circ}_{u, \omla} & = \intola \Tr\big[\defo(\bullet) \: \K \: \defo(\circ)\big] \dO \quad \text{and} \quad \lnorm{\bullet}_{u, \omla} = \left( \intola \Tr\big[\defo(\bullet) \: \K \: \defo(\bullet)\big] \dO \right)^{1/2}; \label{normuOm} \\
\lscalproda{\bullet}{\circ}_{\cont, \omla} & = \intola \Tr\big[ \bullet \: \K^{-1} \: \circ\big] \dO \quad \text{and} \quad \lnorm{\bullet}_{\cont, \omla} = \left( \intola \Tr\big[ \bullet \: \K^{-1} \: \bullet\big] \dO \right)^{1/2}. \label{normcontOm}
\end{align}

Similarly, let us also define inner products and associated norms over boundary $\domla$ of a given homothetic domain $\omla$, such that $\omla \subset \omlabar \subset \Om$:
\begin{align}
%\lscalproda{\bullet}{\circ}_{u, \domla} & = \intdola \Tr\big[\defo(\bullet) \: \K \: \defo(\circ)\big] \dS \quad \text{and} \quad \lnorm{\bullet}_{u, \domla} = \left( \intdola \Tr\big[\defo(\bullet) \: \K \: \defo(\bullet)\big] \dS \right)^{1/2}; \label{normudOm} \\
\lscalprodp{\bullet}{\circ}_{u, \domla} & = \intdola \Tr\big[\defo(\bullet) \: \K \: \defo(\circ)\big] \: \vu_{\labar} \cdot \und{n} \dS \quad \text{and} \quad \labs{\bullet}_{u, \domla} = \left( \intdola \Tr\big[\defo(\bullet) \: \K \: \defo(\bullet)\big] \: \vu_{\labar} \cdot \und{n} \dS \right)^{1/2}; \label{normudOmbis} \\
%\lscalproda{\bullet}{\circ}_{\cont, \domla} & = \intdola \Tr\big[ \bullet \: \K^{-1} \: \circ\big] \dS \quad \text{and} \quad \lnorm{\bullet}_{\cont, \domla} = \left( \intdola \Tr\big[ \bullet \: \K^{-1} \: \bullet\big] \dS \right)^{1/2}; \label{normcontdOm} \\
\lscalprodp{\bullet}{\circ}_{\cont, \domla} & = \intdola \Tr\big[ \bullet \: \K^{-1} \: \circ\big] \: \vu_{\labar} \cdot \und{n} \dS \quad \text{and} \quad \labs{\bullet}_{\cont, \domla} = \left( \intdola \Tr\big[ \bullet \: \K^{-1} \: \bullet\big] \: \vu_{\labar} \cdot \und{n} \dS \right)^{1/2}. \label{normcontdOmbis}
\end{align}

One can easily prove that the following equalities hold:
\begin{align}
\frac{\diff}{\dla} \left[ \lnorm{\bullet}^2_{u, \omla} \right] & = {\left(\frac{\la}{\displaystyle\labar}\right)}^n \frac{1}{\displaystyle\labar} \labs{\bullet}^2_{u, \domlabar}; \label{eq1:diff} \\
\frac{\diff}{\dla} \left[ \lnorm{\bullet}^2_{\cont, \omla} \right] & = {\left(\frac{\la}{\displaystyle\labar}\right)}^n \frac{1}{\displaystyle\labar} \labs{\bullet}^2_{\cont, \domlabar}, \label{eq2:diff}
\end{align}
where $n$ is defined by \eqref{def1:n}.

\begin{proof}
Starting from definition (\ref{normuOm}) of $\lnorm{\bullet}^2_{u, \omla}$, then using relation (\ref{eq1:linkintolaintdolabar}) and definition (\ref{normudOmbis}) of $\labs{\bullet}^2_{u, \domlabar}$, one gets directly:
\begin{align}
\lnorm{\bullet}^2_{u, \omla} & = \intola \Tr\big[\defo(\bullet) \: \K \: \defo(\bullet)\big] \dO \nonumber\\
& = \int_{\la' = 0}^{\la} \left[ {\left(\frac{\la'}{\displaystyle\labar}\right)}^n \frac{1}{\displaystyle\labar} \intdolabar \Tr\big[ \defo(\bullet) \: \K \: \defo(\bullet)\big] \: \vu_{\labar} \cdot \und{n} \dS \right] \dla' \nonumber\\
& = \int_{\la' = 0}^{\la} \left[ {\left(\frac{\la'}{\displaystyle\labar}\right)}^n \frac{1}{\displaystyle\labar} \labs{\bullet}^2_{u, \domlabar} \right] \dla' \label{eq1:proofdiff}
\end{align}
Eventually, differentiation of (\ref{eq1:proofdiff}) with respect to variable $\la$ completes the proof of relation (\ref{eq1:diff}). Similarly, the derivation of relation (\ref{eq2:diff}) can be proved in a straightforward manner.
\end{proof}

\begin{definition}
Measures $e_{\cre,\la}(\hat{\uu}_h, \hat{\cont}_h)$ and $e_{\cre,\la}(\hat{\tilde{\uu}}_h, \hat{\tilde{\cont}}_h)$ of the non-verification of the constitutive relations related to both reference and adjoint problems in part $\omla \subset \Om$ are defined by:
\begin{equation}\label{eq1:ecrela}
e_{\cre,\la} \equiv e_{\cre,\la}(\hat{\uu}_h, \hat{\cont}_h) = \lnorm{\hat{\cont}_h - \K \: \defo(\hat{\uu}_h)}_{\cont, \omla}
\end{equation}
and
\begin{equation}\label{eq1:ecrelatilde}
\tilde{e}_{\cre,\la} \equiv e_{\cre,\la}(\hat{\tilde{\uu}}_h, \hat{\tilde{\cont}}_h) = \lnorm{\hat{\tilde{\cont}}_h - \K \: \defo(\hat{\tilde{\uu}}_h)}_{\cont, \omla},
\end{equation}
respectively.

Similarly, measures $e_{\cre,\setminus \la}(\hat{\uu}_h, \hat{\cont}_h)$ and $e_{\cre,\setminus \la}(\hat{\tilde{\uu}}_h, \hat{\tilde{\cont}}_h)$ of the non-verification of the constitutive relations related to both reference and adjoint problems in the complementary part $\omminusla$ are defined by:
\begin{equation}\label{eq1:ecresetminusla}
e_{\cre,\setminus \la} \equiv e_{\cre,\setminus \la}(\hat{\uu}_h, \hat{\cont}_h) = \lnorm{\hat{\cont}_h - \K \: \defo(\hat{\uu}_h)}_{\cont, \omminusla}
\end{equation}
and
\begin{equation}\label{eq1:ecresetminuslatilde}
\tilde{e}_{\cre,\setminus \la} \equiv e_{\cre,\setminus \la}(\hat{\tilde{\uu}}_h, \hat{\tilde{\cont}}_h) = \lnorm{\hat{\tilde{\cont}}_h - \K \: \defo(\hat{\tilde{\uu}}_h)}_{\cont, \omminusla},
\end{equation}
respectively.
\end{definition}

%: 4.2
\subsection{General idea}\label{4.2}

First, let us recall that the quantity $q$ to bound for building an upper error bound is (see (\ref{eq1:basicidentity})):
\begin{align}
%q & = \intO \Tr\big[(\cont - \hat{\cont}_h^m) \: \K^{-1} \: (\hat{\tilde{\cont}}_h - \K \: \defo(\hat{\tilde{\uu}}_h)) \big] \dO \nonumber\\
q & = \lscalproda{\cont - \hat{\cont}_h^m}{\hat{\tilde{\cont}}_h - \K \: \defo(\hat{\tilde{\uu}}_h)}_{\cont, \Om}, \label{eq1:q}
\end{align}
where $\hat{\cont}_h^m$ and $\hat{\tilde{\cont}}_h - \K \: \defo(\hat{\tilde{\uu}}_h)$ are given quantities and $\cont$ is the unknown exact stress solution of the reference problem.

Let us consider a subdomain $\omla$ of domain $\Om$, whose complementary part is denoted by $\omminusla$ in the following. Then, $q$ can be split into two parts:
\begin{align}
q = q_{\la} + q_{\setminus \la}, \label{eq2:q}
\end{align}
where
\begin{align}
%q_{\la} & = \intola \Tr\big[(\cont - \hat{\cont}_h^m) \: \K^{-1} \: (\hat{\tilde{\cont}}_h - \K \: \defo(\hat{\tilde{\uu}}_h)) \big] \dO \nonumber\\
q_{\la} & = \lscalproda{\cont - \hat{\cont}_h^m}{\hat{\tilde{\cont}}_h - \K \: \defo(\hat{\tilde{\uu}}_h)}_{\cont, \omla} \label{eq1:qomla}
\end{align}
and
\begin{align}
%q_{\setminus \la} & = \intominusla \Tr\big[(\cont - \hat{\cont}_h^m) \: \K^{-1} \: (\hat{\tilde{\cont}}_h - \K \: \defo(\hat{\tilde{\uu}}_h)) \big] \dO \nonumber\\
q_{\setminus \la} & = \lscalproda{\cont - \hat{\cont}_h^m}{\hat{\tilde{\cont}}_h - \K \: \defo(\hat{\tilde{\uu}}_h)}_{\cont, \omminusla}. \label{eq1:qomminusla}
\end{align}

If quantity $\hat{\tilde{\cont}}_h - \K \: \defo(\hat{\tilde{\uu}}_h)$ is concentrated over $\omla$, \ie if subdomain $\omla$ surrounds the zone of interest $\om$, part $q_{\setminus \la}$ can be merely bounded as follows:
\begin{align}
\labs{q_{\setminus \la}} %& \leqslant \lnorm{\cont - \hat{\cont}_h^m}_{\cont, \omminusla} \lnorm{\hat{\tilde{\cont}}_h - \K \: \defo(\hat{\tilde{\uu}}_h)}_{\cont, \omminusla}, \nonumber\\
& \leqslant \lnorm{\cont - \hat{\cont}_h^m}_{\cont, \omminusla} \tilde{e}_{\cre,\setminus \la}, \label{ineq1:qomminusla}
\end{align}
$\tilde{e}_{\cre,\setminus \la}$ being a relatively small computable term. It follows that the main contribution to the error comes from $q_{\la}$. Consequently, quantity $q_{\la}$ has to be correctly bounded.

In order to derive accurate bounds for part $q_{\la}$, the discretization error $\und{e}_h = \uu - \hat{\uu}_h$ on $\omla$ is split into:
\begin{itemize}
\item a local error, denoted $\uu_1$;
\item a pollution error, denoted $\uu_2$,
\end{itemize}
which are solutions of the following two local problems defined on $\omla$, referred to as $(\Pone)$ and $(\Ptwo)$, respectively:

\begin{enumerate}
\item[$\bullet$]
Problem $(\Pone)$ consists of searching $(\uu_1, \cont_1) \in \Ucb \times \Scb$ such that:
\begin{itemize}
\item[$\circ$] $\uu_1 = \und{0} \quad \text{on} \ \domla$ \\
\item[$\circ$] $\und{\diver}(\cont_1) = \und{\diver}(\K \: \defo(\uu - \hat{\uu}_h))$ \\
\item[$\circ$] $\cont_1 = \K \: \defo(\uu_1)$
\end{itemize}

\item[$\bullet$]
Problem $(\Ptwo)$ consists of searching $(\uu_2, \cont_2) \in \Ucb \times \Scb$ such that:
\begin{itemize}
\item[$\circ$] $\uu_2 = \uu - \hat{\uu}_h \quad \text{on} \ \domla$ \\
\item[$\circ$] $\und{\diver}(\cont_2) = \und{0}$ \\
\item[$\circ$] $\cont_2 = \K \: \defo(\uu_2)$
\end{itemize}
\end{enumerate}

Decomposition $\uu - \hat{\uu}_h = \uu_1 + \uu_2$ is the starting point for deriving the main technical result presented in the next section. It is worthy noticing that there is no need to perform the extra-resolutions of local problems $(\Pone)$ and $(\Ptwo)$ to get the main technical and final bounding results (\ref{eq3:fundrelation}) and (\ref{eq2:localerrorbounding}) introduced in \Sects{4.3} and \ref{4.4}, respectively. This point is discussed and proved in \Appendix{} A.

%: 4.3
\subsection{Main technical result}\label{4.3}

Let us consider the space $\Vcb$ of functions satisfying equilibrium conditions:
\begin{align}
\Vcb = \left\{ \vu \in \Ucb \slash \und{\diver}(\K \: \defo(\vu)) = \und{0} \right\} \label{eq1:V},
\end{align}
and let us introduce the Steklov constant, or Steklov eigenvalue, $h$ defined in \cite{Ste02} as:
\begin{align}
h = \max_{\substack{ \vu \in \Vcb }} \Scb_1(\vu) \label{eq1:constanth}
\end{align}
with
\begin{align}
\Scb_1(\vu) = \frac{\lnorm{\K \: {(\vu \otimes \und{n})}_{\sym}}^2_{\cont, \dom_1}}{\lnorm{\vu}^2_{u, \om_1}}, \label{eq1:S1}
\end{align}
where subdomain $\om_1$ denotes the homothetic domain $\omla$ associated to a constant parameter $\la = 1$. Then, for any homothetic domain $\omla \subset \Om$ parametrized by $\la > 0$, one can derive a relation involving the product of constant $h$ and parameter $\la$:
\begin{align}
h \la = \max_{\substack{ \vu \in \Vcb }} \Scb_{\la}(\vu) \label{eq1:constanthla}
\end{align}
with
\begin{align}
\Scb_{\la}(\vu) = \frac{\lnorm{\K \: {(\vu \otimes \und{n})}_{\sym}}^2_{\cont, \domla}}{\lnorm{\vu}^2_{u, \omla}}. \label{eq1:Sla}
\end{align}

\begin{proposition}\label{prop1}
Let $(\omla,\omlabar)$ be a pair of homothetic domains such that $\la \in \intervaloc{0}{\labar}$, \ie $\omla \subset \omlabar$. The following key inequality holds:
\begin{align}
\lnorm{\cont - \hat{\cont}_h}^2_{\cont, \omla} \leqslant \left( \frac{\la}{\labar} \right)^{1/h} \lnorm{\cont - \hat{\cont}_h}^2_{\cont, \omlabar} + \gamma_{\la, \labar}, \label{eq3:fundrelation}
\end{align}
where
\begin{equation}\label{eq1:gamma}
\displaystyle\gamma_{\la, \labar} \equiv \gamma_{\la, \labar}(\hat{\uu}_h, \hat{\cont}_h) = \int_{\la' = \la}^{\labar} \left[ \left( \frac{\la'}{\la} \right)^{-1/h} \frac{1}{h \la'} e^2_{\cre,\la'} \right] \dla'.
\end{equation}
\end{proposition}

The proof of \Prop{prop1} is given in \Appendix{} A.

Let us note that, using Prager-Synge's equality (\ref{eq1:sigmahathmecre}), unknown term $\lnorm{\cont - \hat{\cont}_h}_{\cont, \omlabar}$ involved in the right-hand side term of fundamental inequality (\ref{eq3:fundrelation}) in \Prop{prop1} is readily bounded as:
\begin{equation}\label{ineq1:sigmaminussigmahath}
\lnorm{\cont - \hat{\cont}_h}^2_{\cont, \omlabar} \leqslant \left( \lnorm{\cont - \hat{\cont}^m_h}_{\cont, \omlabar} + \lnorm{\hat{\cont}^m_h - \hat{\cont}_h}_{\cont, \omlabar} \right)^2 \leqslant \frac{1}{4} \left( e_{\cre}+ e_{\cre,\labar} \right)^2.
\end{equation}
It follows that fundamental result (\ref{eq3:fundrelation}) can be rewritten in terms of perfectly known quantities as:
\begin{equation}\label{eq4:fundrelation}
\lnorm{\cont - \hat{\cont}_h}^2_{\cont, \omla} \leqslant \left( \frac{\la}{\labar} \right)^{1/h} \frac{1}{4} \left( e_{\cre} + e_{\cre,\labar} \right)^2 + \gamma_{\la, \labar}.
\end{equation}

\begin{remark}
Quantity $\lnorm{\cont - \hat{\cont}_h}_{\cont, \omlabar}$ can be straightforwardly bounded without introducing $\hat{\cont}^m_h$ by simply using $\lnorm{\cont - \hat{\cont}_h}^2_{\cont, \omlabar} \leqslant \lnorm{\cont - \hat{\cont}_h}^2_{\cont, \Om} \leqslant e^2_{\cre}$. Nevertheless, the corresponding bound, namely $e^2_{\cre}$, is less accurate than the one given by (\ref{ineq1:sigmaminussigmahath}), namely $1/4 \left( e_{\cre} + e_{\cre,\labar} \right)^2$.
\end{remark}

%: 4.4
\subsection{Final bounding result}\label{4.4}

\begin{proposition}\label{prop2}
The final improved bounding result reads:
\begin{align}\label{eq2:localerrorbounding}
\labs{\Iex - \Ih - \Ihh - \Ihhh} \leqslant \tilde{e}_{\cre,\la} \: \delta_{\la, \labar} + \frac{1}{2} \: e_{\cre} \: \tilde{e}_{\cre,\setminus \la},
\end{align}
where
\begin{align}
\delta_{\la, \labar} \equiv \delta_{\la, \labar}(\hat{\uu}_h,\hat{\cont}_h) = \left[ \left( \frac{\la}{\labar} \right)^{1/h} \frac{1}{4} \left( e_{\cre} + e_{\cre,\labar} \right)^2 + \gamma_{\la, \labar} \right]^{1/2}
\end{align}
and
\begin{align}
\Ihhh & = \frac{1}{2} \intola \Tr\big[(\hat{\cont}_h - \K \: \defo(\hat{\uu}_h)) \: \K^{-1} \: (\hat{\tilde{\cont}}_h - \K \: \defo(\hat{\tilde{\uu}}_h)) \big] \dO \nonumber\\
& = \frac{1}{2} \lscalproda{\hat{\cont}_h - \K \: \defo(\hat{\uu}_h)}{\hat{\tilde{\cont}}_h - \K \: \defo(\hat{\tilde{\uu}}_h)}_{\cont, \omla} \label{eq1:Ihhh}
\end{align}
are all computable from the calculated approximate solutions of both reference and adjoint problems. $\Ih + \Ihh + \Ihhh$ can be viewed as a new approximate solution of the exact value $\Iex$ of the quantity of interest.
\end{proposition}

\begin{proof}
First, let us address the question of bounding of part $q_{\setminus \la}$. Starting from inequation (\ref{ineq1:qomminusla}) and using Prager-Synge's equality (\ref{eq1:sigmahathmecre}), quantity $q_{\setminus \la}$ can be bounded as follows:
\begin{align}
\labs{q_{\setminus \la}} %& \leqslant \lnorm{\cont - \hat{\cont}_h^m}_{\cont, \Om} \lnorm{\hat{\tilde{\cont}}_h - \K \: \defo(\hat{\tilde{\uu}}_h)}_{\cont, \omminusla} \nonumber\\
& \leqslant \frac{1}{2} \: e_{\cre} \: \tilde{e}_{\cre,\setminus \la}. \label{ineq2:qomminusla}
\end{align}

Second, let us now handle the question of bounding of part $q_{\la}$. Equation (\ref{eq1:qomla}) can be rewritten as follows:
\begin{align}
q_{\la} %& = \intola \Tr\big[(\cont - \hat{\cont}_h) \: \K^{-1} \: (\hat{\tilde{\cont}}_h - \K \: \defo(\hat{\tilde{\uu}}_h)) \big] \dO + \intola \Tr\big[(\hat{\cont}_h - \hat{\cont}_h^m) \: \K^{-1} \: (\hat{\tilde{\cont}}_h - \K \: \defo(\hat{\tilde{\uu}}_h)) \big] \dO \\
& = \lscalproda{\cont - \hat{\cont}_h}{\hat{\tilde{\cont}}_h - \K \: \defo(\hat{\tilde{\uu}}_h)}_{\cont, \omla} + \lscalproda{\hat{\cont}_h - \hat{\cont}_h^m}{\hat{\tilde{\cont}}_h - \K \: \defo(\hat{\tilde{\uu}}_h)}_{\cont, \omla} \\
%& = \intola \Tr\big[(\cont - \hat{\cont}_h) \: \K^{-1} \: (\hat{\tilde{\cont}}_h - \K \: \defo(\hat{\tilde{\uu}}_h)) \big] \dO + \frac{1}{2} \intola \Tr\big[(\hat{\cont}_h - \K \: \defo(\hat{\uu}_h)) \: \K^{-1} \: (\hat{\tilde{\cont}}_h - \K \: \defo(\hat{\tilde{\uu}}_h)) \big] \dO. \\
& = \lscalproda{\cont - \hat{\cont}_h}{\hat{\tilde{\cont}}_h - \K \: \defo(\hat{\tilde{\uu}}_h)}_{\cont, \omla} + \Ihhh, \label{eq2:qomla}
\end{align}
where $\Ihhh$ is a calculable known term defined in (\ref{eq1:Ihhh}).

Applying the Cauchy-Schwarz inequality to (\ref{eq2:qomla}) with respect to scalar product $\lscalproda{\bullet}{\circ}_{\cont, \omla}$, one has:
\begin{align}
\labs{q_{\la} - \Ihhh} & \leqslant \lnorm{\cont - \hat{\cont}_h}_{\cont, \omla} \tilde{e}_{\cre,\la}. \label{ineq1:qomla}
\end{align}

Then, introducing the key inequality (\ref{eq4:fundrelation}) coming from \Prop{prop1} into (\ref{ineq1:qomla}) leads to the following bounding result:
\begin{align}
\labs{q_{\la} - \Ihhh} \leqslant \left[ \left( \frac{\la}{\labar} \right)^{1/h} \frac{1}{4} \left( e_{\cre} + e_{\cre,\labar} \right)^2 + \gamma_{\la, \labar}(\hat{\uu}_h, \hat{\cont}_h) \right]^{1/2} \tilde{e}_{\cre,\la}.\label{ineq2:qomla}
\end{align}

Finally, using both inequalities (\ref{ineq2:qomminusla}) and (\ref{ineq2:qomla}), one gets:
\begin{align}
\labs{q - \Ihhh} \leqslant & \left[ \left( \frac{\la}{\labar} \right)^{1/h} \frac{1}{4} \left( e_{\cre} + e_{\cre,\labar} \right)^2 + \gamma_{\la, \labar} \right]^{1/2} \tilde{e}_{\cre,\la} + \frac{1}{2} \: e_{\cre} \: \tilde{e}_{\cre,\setminus \la}, \label{ineq1:q}
\end{align}
which completes the proof of \Prop{2}.
\end{proof}

Thus, this improved technique provides the following guaranteed lower and upper bounds $(\chiinf,\chisup)$ of $\Iex$:
\begin{align}\label{eq1:improvedlocalerrorbounding}
\chiinf \leqslant \Iex \leqslant \chisup,
\end{align}
with
\begin{align}
\chiinf = & \Ih + \Ihh + \Ihhh - \labs{ \tilde{e}_{\cre,\la} \: \delta_{\la, \labar} + \frac{1}{2} \: e_{\cre} \: \tilde{e}_{\cre,\setminus \la} }; \label{eq1:improvedlocalerrorboundschiinf} \\
\chisup = & \Ih + \Ihh + \Ihhh + \labs{ \tilde{e}_{\cre,\la} \: \delta_{\la, \labar} + \frac{1}{2} \: e_{\cre} \: \tilde{e}_{\cre,\setminus \la} }. \label{eq1:improvedlocalerrorboundschisup}
\end{align}

\begin{remark}
These bounds depend on both parameters $\la$ and $\labar$. In order to get a practical minimizer, one seeks to reduce ratio $\displaystyle\frac{\la}{\labar}$ as much as possible by choosing:
\begin{itemize}
\item the smallest parameter $\la$ such that domain $\omla$ surrounds the zone of interest $\om$;
\item the largest parameter $\labar$ such that domain $\omlabar$ remains a homothetic mapping of $\omla$ (preserving its geometric shape) contained in $\Om$,
\end{itemize}
and leading to sharp error bounds.
\end{remark}

\begin{remark}
The Steklov constant $h$ could be easily computed, $\Scb_1(\vu)$ being a Rayleigh quotient associated with a symmetric eigenproblem. By considering a material with isotropic, homogeneous, linear and elastic behavior, in the two-dimensional (three-dimensional, respectively) case of a unit circle\footnotemark[1] (unit sphere\footnotemark[1], respectively) $\om_1$, it has been shown numerically that the maximum of $\Scb_1$ is reached for $\vu = \OM$, where $O$ is the homothetic center and $M \in \dom_1$; it follows that $\defo(\vu) = \Id$. Let us note that in the two-dimensional case of a unit cracked circle and of a double unit square\footnotemark[2] as well as in the three-dimensional case of a double unit parallelepiped\footnotemark[2], the same maximum eigenfunction for $\Scb_1$ has been obtained numerically.
Analytical expressions of constant $h$ for various shape domains are reported in \Tab{table1:values_constant_h} (see \Appendix{} C). Note that, in the particular shape domains we considered, constant $h$ only depends on Poisson's ratio $\nu$. Besides, lower bounds of this constant for a circular shape domain are given in \cite{Lad78}.
\footnotetext[1]{a unit circle (sphere, respectively) corresponds to a circular (spherical, respectively) domain of radius one.}
\footnotetext[2]{a double unit square (parallelepiped, respectively) corresponds to a squared (parallelepiped, respectively) domain of side length two.}
\end{remark}

%: 5.
\section{Second improved bounding technique}\label{5}

%: 5.1
\subsection{General idea}\label{5.1}

As for the first improved bounding method presented in \Sect{4}, let us consider a subdomain $\omlabar \subset \Om$, with complementary part $\omminuslabar$. Quantity $q$ previously defined in (\ref{eq1:q}) can be decomposed as follows:
\begin{equation}\label{eq3:q}
q = q_{\labar} + q_{\setminus \labar},
\end{equation}
where
\begin{align}
q_{\labar} %& = \intolabar \Tr\big[(\cont - \hat{\cont}_h^m) \: \K^{-1} \: (\hat{\tilde{\cont}}_h - \K \: \defo(\hat{\tilde{\uu}}_h)) \big] \dO \nonumber\\
& = \lscalproda{\cont - \hat{\cont}_h^m}{\hat{\tilde{\cont}}_h - \K \: \defo(\hat{\tilde{\uu}}_h)}_{\cont, \omlabar} \label{eq1:qomlabar}
\end{align}
and
\begin{align}
q_{\setminus \labar} %& = \intominuslabar \Tr\big[(\cont - \hat{\cont}_h^m) \: \K^{-1} \: (\hat{\tilde{\cont}}_h - \K \: \defo(\hat{\tilde{\uu}}_h)) \big] \dO \nonumber\\
& = \lscalproda{\cont - \hat{\cont}_h^m}{\hat{\tilde{\cont}}_h - \K \: \defo(\hat{\tilde{\uu}}_h)}_{\cont, \omminuslabar}. \label{eq1:qomminuslabar}
\end{align}

Similarly to the previous improved technique, quantity $q_{\setminus \labar}$ can be easily bounded as follows:
\begin{align}
\labs{q_{\setminus \labar}} %\leqslant \lnorm{\cont - \hat{\cont}_h^m}_{\cont, \omminuslabar} \lnorm{\hat{\tilde{\cont}}_h - \K \: \defo(\hat{\tilde{\uu}}_h)}_{\cont, \omminuslabar}; \nonumber\\
\leqslant \lnorm{\cont - \hat{\cont}_h^m}_{\cont, \omminuslabar} \tilde{e}_{\cre,\setminus \labar}; \label{ineq1:qomminuslabar}
\end{align}
whereas the bounding of $q_{\labar}$ differs widely. As previously mentioned in \Sect{4.2}, the main contribution to the error derives from $q_{\labar}$.

In order to build sharp bounds for part $q_{\labar}$, let us introduce the following local problem defined on $\omlabar \subset \Om$, referred to as $(\Pexh)$, which consists of searching $(\uu^h_{ex}, \cont^h_{ex}) \in \Ucb \times \Scb$ such that:
\begin{itemize}
\item[$\circ$] $\uu^h_{ex} = \hat{\uu}_h \quad \text{on} \ \domlabar$ \\
\item[$\circ$] $\und{\diver}(\cont^h_{ex}) + \und{f}_d = \und{0}$ \\
\item[$\circ$] $\cont^h_{ex} = \K \: \defo(\uu^h_{ex})$
\end{itemize}
Similarly to local problems $(\Pone)$ and $(\Ptwo)$ previously defined in \Sect{4.2} for the first improved bounding technique, one can mention that the final bounding result (\ref{eq3:localerrorbounding}) introduced in \Sect{5.4} does not require the solution of local problem $(\Pexh)$, as the present improved technique circumvents the need to perform any additional resolution of $(\Pexh)$.

It follows that the discretization error $\und{e}_h = \uu - \hat{\uu}_h$ and $\cont - \hat{\cont}_h$ on $\omlabar$ could be rewritten:
\begin{align}\label{eq2:decomposition}
\uu - \hat{\uu}_h = (\uu - \uu^h_{ex}) + (\uu^h_{ex} - \hat{\uu}_h)
\end{align}
and
\begin{align}\label{eq3:decomposition}
\cont - \hat{\cont}_h = (\cont - \cont^h_{ex}) + (\cont^h_{ex} - \hat{\cont}_h),
\end{align}
respectively.

Then, quantity $q_{\labar}$ can be decomposed into two parts:
\begin{equation}
q_{\labar} = q_{\labar,1} + q_{\labar,2},
\end{equation}
where
\begin{align}
q_{\labar,1} %& = \intolabar \Tr\big[(\cont - \cont^h_{ex}) \: \K^{-1} \: (\hat{\tilde{\cont}}_h - \K \: \defo(\hat{\tilde{\uu}}_h)) \big] \dO \nonumber\\
& = \lscalproda{\cont - \cont^h_{ex}}{\hat{\tilde{\cont}}_h - \K \: \defo(\hat{\tilde{\uu}}_h)}_{\cont, \omlabar} \label{eq1:qomlabar1}
\end{align}
and
\begin{align}
q_{\labar,2} %& = \intolabar \Tr\big[(\cont^h_{ex} - \hat{\cont}_h^m) \: \K^{-1} \: (\hat{\tilde{\cont}}_h - \K \: \defo(\hat{\tilde{\uu}}_h)) \big] \dO \nonumber\\
& = \lscalproda{\cont^h_{ex} - \hat{\cont}_h^m}{\hat{\tilde{\cont}}_h - \K \: \defo(\hat{\tilde{\uu}}_h)}_{\cont, \omlabar}. \label{eq1:qomlabar2}
\end{align}

It will be shown that second part $q_{\labar,2}$ could be easily bounded, while a sharp bound of first part $q_{\labar,1}$ will be derived thanks to Saint-Venant's principle.

%: 5.2
\subsection{Bounding of part $q_{\labar,2}$}\label{5.2}

First, let us address the question of bounding of part $q_{\labar,2}$. Applying the Cauchy-Schwarz inequality to (\ref{eq1:qomlabar2}) with respect to scalar product $\lscalproda{\bullet}{\circ}_{\cont, \omlabar}$ leads to:
\begin{align}
\labs{q_{\labar,2}} %& \leqslant \lnorm{\cont^h_{ex} - \hat{\cont}_h^m}_{\cont, \omlabar} \lnorm{\hat{\tilde{\cont}}_h - \K \: \defo(\hat{\tilde{\uu}}_h)}_{\cont, \omlabar}. \label{ineq1:qomlabar2}
& \leqslant \lnorm{\cont^h_{ex} - \hat{\cont}_h^m}_{\cont, \omlabar} \tilde{e}_{\cre,\labar}. \label{ineq1:qomlabar2}
\end{align}

\begin{proposition}\label{prop3}
Let $(\uu^h_{ex}, \cont^h_{ex}) \in \Ucb \times \Scb$ be the exact solution of problem $(\Pexh)$ defined over $\omlabar$. Given an admissible approximate solution $(\hat{\uu}_h, \hat{\cont}_h)$ of the reference problem, the following equality holds:
\begin{equation}\label{eq1:prop3}
\lnorm{\cont^h_{ex} - \hat{\cont}_h^m}_{\cont, \omlabar} %= \frac{1}{2} \lnorm{\hat{\cont}_h - \K \: \defo(\hat{\uu}_h)}_{\cont, \omlabar} 
= \frac{1}{2} \: e_{\cre,\labar}
\end{equation}
\end{proposition}

Let us note that replacing $\cont^h_{ex}$ by the exact solution $\cont$ of the reference problem in \Prop{prop3} leads generally to a incorrect result except over the whole domain $\Om$, where Prager-Synge's equality (\ref{eq1:sigmahathmecre}) holds.

\begin{proof}
Noticing that the restriction of $(\hat{\uu}_h, \hat{\cont}_h)$ to $\omlabar$ is also an admissible solution of problem $(\Pexh)$ defined over $\omlabar$, the well-known Prager-Synge hypercircle theorem \cite{Pra47} leads to the following bounding inequality:
\begin{align}\label{eq2:PragerSynge}
\lnorm{\uu^h_{ex} - \hat{\uu}_h}^2_{u, \omlabar} \leqslant \lnorm{\uu^h_{ex} - \hat{\uu}_h}^2_{u, \omlabar} + \lnorm{\cont^h_{ex} - \hat{\cont}_h}^2_{\cont,\omlabar} = e^2_{\cre,\labar},
\end{align}
which conveys the guaranteed nature of the CRE $e_{\cre,\labar}$ on $\omlabar$ for problem $(\Pexh)$.

Introducing the average admissible field $\hat{\cont}^m_h$ defined by (\ref{eq1:sigmahathm}) completes the proof of \Prop{prop3}.
\end{proof}

Finally, incorporating result (\ref{eq1:prop3}) of \Prop{prop3} into (\ref{ineq1:qomlabar2}), one gets:
\begin{align}
\labs{q_{\labar,2}} & \leqslant \frac{1}{2} \: e_{\cre,\labar} \: \tilde{e}_{\cre,\labar}. \label{ineq2:qomlabar2}
\end{align}

From now on, in order to get accurate bounds for part $q_{\labar,1}$, let us introduce the main technical result.

%: 5.3
\subsection{Main technical result}\label{5.3}

Let us consider the space $\Vcb$ introduced in (\ref{eq1:V}) and let us define the following dimensionless constant:
\begin{align}
k = \min_{\substack{ \vu \in \Vcb }} \Rcb_{\labar}(\vu) \label{eq1:constantk}
\end{align}
with
\begin{align}
\Rcb_{\labar}(\vu) = \frac{\labs{\vu}^2_{u, \domlabar}}{\lnorm{\vu}^2_{u, \omlabar}}, \label{eq1:Rlabar}
\end{align}
for a given homothetic domain $\omlabar \subset \Om$ associated to parameter $\labar >0$. Then, for any domain $\omla$ homothetic to $\omlabar$ such that $\omla \subset \omlabar \subset \Om$, \ie for any $\la \in \intervaloc{0}{\labar}$, one can derive a relation involving the product of constant $k$ and ratio $\displaystyle\frac{\labar}{\la}$:
\begin{align}
k \frac{\labar}{\la} = \min_{\substack{ \vu \in \Vcb }} \Rcb_{\la}(\vu) \label{eq1:constantklabarla}
\end{align}
with
\begin{align}
\Rcb_{\la}(\vu) = \frac{\labs{\vu}^2_{u, \domla}}{\lnorm{\vu}^2_{u, \omla}}. \label{eq2:Rla}
\end{align}

\begin{remark}
This problem is connected to the description of a 2D or 3D homogeneous domain as an abstract beam for which a semi-group could be defined \cite{Hoc93}. However, we do not know any name for this constant which is strictly positive for star-shaped domains such that $\vu_{\labar} \cdot \und{n} > 0$.
\end{remark}

\begin{proposition}\label{prop4}
Let $(\omla,\omlabar)$ be a pair of homothetic domains such that $\la \in \intervaloc{0}{\labar}$, \ie $\omla \subset \omlabar \subset \Om$. The following key inequality holds:
\begin{align}
\forall \: \vu \in \Vcb, \quad \lnorm{\vu}^2_{u, \omla} \leqslant \left( \frac{\la}{\labar}\right)^k \lnorm{\vu}^2_{u, \omlabar}. \label{eq2:fundrelation2}
\end{align}
\end{proposition}

The proof of \Prop{prop4} is given in \Appendix{} B.

The bounding of part $q_{\labar,1}$ involves quantity $\uu - \uu^h_{ex}$ (resp. $\cont - \cont^h_{ex}$) of decomposition (\ref{eq2:decomposition}) (resp. (\ref{eq3:decomposition})) of the discretization error. Recalling that $\uu - \uu^h_{ex} \in \Vcb$, applying fundamental result (\ref{eq2:fundrelation2}) of \Prop{prop4} to $\uu - \uu^h_{ex}$ leads to:
\begin{align}
\lnorm{\uu - \uu^h_{ex}}^2_{u, \omla} \leqslant \left( \frac{\la}{\labar} \right)^{k} \lnorm{\uu - \uu^h_{ex}}^2_{u, \omlabar}, \label{eq3:fundrelation2}
\end{align}
or, equivalently:
\begin{align}
\lnorm{\cont - \cont^h_{ex}}^2_{\cont, \omla} \leqslant \left( \frac{\la}{\labar} \right)^{k} \lnorm{\cont - \cont^h_{ex}}^2_{\cont, \omlabar}. \label{eq4:fundrelation2}
\end{align}

%: 5.4
\subsection{Final bounding result}\label{5.4}

\begin{proposition}\label{prop5}
The final improved bounding result reads:
\begin{align}\label{eq3:localerrorbounding}
\labs{\Iex - \Ih - \Ihh} \leqslant \frac{1}{2} \left[ e_{\cre} \left[ \tilde{\theta}^2_{\labar} + \tilde{e}^2_{\cre,\setminus \labar} \right]^{1/2} + e_{\cre,\labar} \left[ \tilde{\theta}_{\labar} + \tilde{e}_{\cre,\labar} \right] \right],
\end{align}
where
\begin{align}\label{eq1:theta}
\tilde{\theta}_{\labar} \equiv \tilde{\theta}_{\labar}(\hat{\tilde{\uu}}_h,\hat{\tilde{\cont}}_h) = \labs{\hat{\tilde{\cont}}_h - \K \: \defo(\hat{\tilde{\uu}}_h)}_{\cont, \domlabar} \frac{2 \: \displaystyle k^{1/2}}{k+n+1}
\end{align}
and
\begin{align}
\tilde{e}^2_{\cre,\setminus \labar} = \tilde{e}^2_{\cre} - \tilde{e}^2_{\cre,\labar}
\end{align}
are fully calculable from the calculated approximate solution of adjoint problem.
\end{proposition}

\begin{proof}
Let us first handle the question of bounding of part $q_{\labar,1}$, which is actually the key point of this new technique. Using relation (\ref{eq1:linkintolaintdolabar}) for domain $\omlabar$ itself and applying the Cauchy-Schwarz inequality with respect to scalar product $\lscalprodp{\bullet}{\circ}_{\cont, \domlabar}$, one has:
\begin{align}
q_{\labar,1} & = \int_{\la = 0}^{\labar} {\left(\frac{\la}{\displaystyle\labar}\right)}^n \frac{1}{\displaystyle\labar} %\underbrace{\intdolabar \Tr\big[ (\cont - \cont^h_{ex}) \: \K^{-1} \: (\hat{\tilde{\cont}}_h - \K \: \defo(\hat{\tilde{\uu}}_h))\big] \: \vu_{\labar} \cdot \und{n} \dS}_{\displaystyle\lscalprodp{\cont - \cont^h_{ex}}{\hat{\tilde{\cont}}_h - \K \: \defo(\hat{\tilde{\uu}}_h)}_{\cont, \domlabar}} \dla \nonumber\\
\lscalprodp{\cont - \cont^h_{ex}}{\hat{\tilde{\cont}}_h - \K \: \defo(\hat{\tilde{\uu}}_h)}_{\cont, \domlabar} \dla
\end{align}
and
\begin{align}
\labs{q_{\labar,1}} & \leqslant \int_{\la = 0}^{\labar} \left[ {\left(\frac{\la}{\displaystyle\labar}\right)}^n \frac{1}{\displaystyle\labar} \labs{\cont - \cont^h_{ex}}_{\cont, \domlabar} \labs{\hat{\tilde{\cont}}_h - \K \: \defo(\hat{\tilde{\uu}}_h)}_{\cont, \domlabar} \right] \dla. \label{ineq2:qomlabar1}
\end{align}
In order to derive an upper bound of (\ref{ineq2:qomlabar1}), let us introduce a function $\mu(\la) > 0$:
\begin{align}
\labs{q_{\labar,1}} & \leqslant \int_{\la = 0}^{\labar} \left[ \sqrt{\mu(\la)} {\left(\frac{\la}{\displaystyle\labar}\right)}^n \frac{1}{\displaystyle\labar} \labs{\cont - \cont^h_{ex}}_{\cont, \domlabar} \frac{1}{\sqrt{\mu(\la)}} \labs{\hat{\tilde{\cont}}_h - \K \: \defo(\hat{\tilde{\uu}}_h)}_{\cont, \domlabar} \right] \dla \nonumber \\
& \leqslant \frac{1}{2} \int_{\la = 0}^{\labar} \left[ \mu(\la) {\left(\frac{\la}{\displaystyle\labar}\right)}^n \frac{1}{\displaystyle\labar} \labs{\cont - \cont^h_{ex}}^2_{\cont, \domlabar} \right] \dla + \frac{1}{2} \int_{\la = 0}^{\labar} \left[ \frac{1}{\mu(\la)} {\left(\frac{\la}{\displaystyle\labar}\right)}^n \frac{1}{\displaystyle\labar} \labs{\hat{\tilde{\cont}}_h - \K \: \defo(\hat{\tilde{\uu}}_h)}^2_{\cont, \domlabar} \right] \dla \label{ineq3:qomlabar1}
\end{align}
Subsequently, applying result (\ref{eq2:diff}) to $\cont - \cont^h_{ex}$ in the left-hand side term of (\ref{ineq3:qomlabar1}) leads to:
\begin{align}
\labs{q_{\labar,1}} & \leqslant \frac{1}{2} \int_{\la = 0}^{\labar} \left[ \mu(\la) \frac{\diff}{\dla} \left[ \lnorm{\cont - \cont^h_{ex}}^2_{\cont, \omla} \right] \right] \dla + \frac{1}{2} \int_{\la = 0}^{\labar} \left[ \frac{1}{\mu(\la)} {\left(\frac{\la}{\displaystyle\labar}\right)}^n \frac{1}{\displaystyle\labar} \labs{\hat{\tilde{\cont}}_h - \K \: \defo(\hat{\tilde{\uu}}_h)}^2_{\cont, \domlabar} \right] \dla. \label{ineq4:qomlabar1}
\end{align}
Integrating the first term of the left-hand side of (\ref{ineq4:qomlabar1}) by parts, one obtains:
\begin{align}
\labs{q_{\labar,1}} \leqslant \frac{1}{2} \mu(\labar) \lnorm{\cont - \cont^h_{ex}}^2_{\cont, \omlabar} & - \frac{1}{2} \int_{\la = 0}^{\labar} \left[ \frac{\diff}{\dla} \left[ \mu(\la) \right] \lnorm{\cont - \cont^h_{ex}}^2_{\cont, \omla} \right] \dla \nonumber\\
& + \frac{1}{2} \int_{\la = 0}^{\labar} \left[ \frac{1}{\mu(\la)} {\left(\frac{\la}{\displaystyle\labar}\right)}^n \frac{1}{\displaystyle\labar} \labs{\hat{\tilde{\cont}}_h - \K \: \defo(\hat{\tilde{\uu}}_h)}^2_{\cont, \domlabar} \right] \dla. \label{ineq5:qomlabar1}
\end{align}

Assuming that $\displaystyle\frac{\diff}{\dla} \left[ \mu(\la) \right] \leqslant 0 \quad \forall \: \la \in \intervaloc{0}{\labar}$, using fundamental result (\ref{eq4:fundrelation2}) coming from \Prop{prop4} leads to:
\begin{align}
\labs{q_{\labar,1}} \leqslant \frac{1}{2} \mu(\labar) \lnorm{\cont - \cont^h_{ex}}^2_{\cont, \omlabar} & - \frac{1}{2} \lnorm{\cont - \cont^h_{ex}}^2_{\cont, \omlabar} \int_{\la = 0}^{\labar} \left[ \frac{\diff}{\dla} \left[ \mu(\la) \right] \left( \frac{\la}{\labar} \right)^{k} \right] \dla \nonumber\\
& + \frac{1}{2} \int_{\la = 0}^{\labar} \left[ \frac{1}{\mu(\la)} {\left(\frac{\la}{\displaystyle\labar}\right)}^n \frac{1}{\displaystyle\labar} \labs{\hat{\tilde{\cont}}_h - \K \: \defo(\hat{\tilde{\uu}}_h)}^2_{\cont, \domlabar} \right] \dla. \label{ineq6:qomlabar1}
\end{align}

Integrating the second term of the right-hand side of (\ref{ineq6:qomlabar1}) by parts, one obtains:
\begin{align}
\int_{\la = 0}^{\labar} \left[ \frac{\diff}{\dla} \left[ \mu(\la) \right] \left( \frac{\la}{\labar} \right)^{k} \right] \dla = \mu(\labar) - \int_{\la = 0}^{\labar} \left[ \mu(\la) \frac{k}{\la} \left( \frac{\la}{\labar} \right)^{k} \right] \dla;
\end{align}
thus resulting in the following inequality:
\begin{align}
\labs{q_{\labar,1}} \leqslant \frac{1}{2} \int_{\la = 0}^{\labar} \left[ \mu(\la) \frac{k}{\la} \left( \frac{\la}{\labar} \right)^{k} \lnorm{\cont - \cont^h_{ex}}^2_{\cont, \omlabar} + \frac{1}{\mu(\la)} {\left(\frac{\la}{\displaystyle\labar}\right)}^n \frac{1}{\displaystyle\labar} \labs{\hat{\tilde{\cont}}_h - \K \: \defo(\hat{\tilde{\uu}}_h)}^2_{\cont, \domlabar} \right] \dla. \label{ineq7:qomlabar1}
\end{align}

The optimal function $\mu$ is the one which minimizes the right-hand side of (\ref{ineq7:qomlabar1}):
\begin{align}
\mu(\la) = \argmin_{\substack{ \mu^{\ast}(\la) > 0 \\ \frac{\diff}{\dla} \left[ \mu^{\ast}(\la) \right] \leqslant 0 }} \left\{ \frac{1}{2} \int_{\la = 0}^{\labar} \left[ \mu^{\ast}(\la) \frac{k}{\la} \left( \frac{\la}{\labar} \right)^{k} \lnorm{\cont - \cont^h_{ex}}^2_{\cont, \omlabar} + \frac{1}{\mu^{\ast}(\la)} {\left(\frac{\la}{\displaystyle\labar}\right)}^n \frac{1}{\displaystyle\labar} \labs{\hat{\tilde{\cont}}_h - \K \: \defo(\hat{\tilde{\uu}}_h)}^2_{\cont, \domlabar} \right] \dla \right\}.
\end{align}
It reads:
\begin{align}\label{eq1:functionmu}
\mu(\la) = \frac{\labs{\hat{\tilde{\cont}}_h - \K \: \defo(\hat{\tilde{\uu}}_h)}_{\cont, \domlabar}}{\lnorm{\cont - \cont^h_{ex}}_{\cont, \omlabar}} \frac{1}{\displaystyle k^{1/2}} \left( \frac{\labar}{\la} \right)^{(k-n-1)/2},
\end{align}
under the assumption that $\mu$ is a strictly positive, monotonically decreasing function of $\la$ on the interval $\intervaloc{0}{\labar}$, which implies the following condition:
\begin{align}
k \geqslant n+1. \label{eq1:conditionk}
\end{align}

\begin{remark}
Note that, in the two- and three-dimensional cases of circular and spherical domains, respectively, condition (\ref{eq1:conditionk}) holds (see \Rem{rem1:k}).
\end{remark}

Then, replacing function $\mu(\la)$ by expression (\ref{eq1:functionmu}) into (\ref{ineq7:qomlabar1}), one gets:
\begin{align}
\labs{q_{\labar,1}} % & \leqslant \lnorm{\cont - \cont^h_{ex}}_{\cont, \omlabar} \labs{\hat{\tilde{\cont}}_h - \K \: \defo(\hat{\tilde{\uu}}_h)}_{\cont, \domlabar} k^{1/2} \int_{\la = 0}^{\labar} \left[ \frac{1}{\displaystyle\labar} {\left( \frac{\la}{\labar} \right)}^{(k+n-1)/2} \right] \dla \nonumber\\
& \leqslant \lnorm{\cont - \cont^h_{ex}}_{\cont, \omlabar} \tilde{\theta}_{\labar}, \label{ineq8:qomlabar1}
\end{align}
where $\tilde{\theta}_{\labar}$ is defined by (\ref{eq1:theta}).
Finally, collecting bounding results (\ref{ineq8:qomlabar1}) and (\ref{ineq2:qomlabar2}) for $q_{\labar,1}$ and $q_{\labar,2}$, respectively, $q_{\labar}$ can be bounded as follows:
\begin{align}
\labs{q_{\labar}} & \leqslant \lnorm{\cont - \cont^h_{ex}}_{\cont, \omlabar} \tilde{\theta}_{\labar} + \frac{1}{2} \: e_{\cre,\labar} \: \tilde{e}_{\cre,\labar} \nonumber\\
& \leqslant \left[ \lnorm{\cont - \cont^h_{ex} - \frac{1}{2} \left( \hat{\cont}_h - \K \: \defo(\hat{\uu}_h) \right)}_{\cont, \omlabar} + \frac{1}{2} \: e_{\cre,\labar} \right] \tilde{\theta}_{\labar} + \frac{1}{2} \: e_{\cre,\labar} \: \tilde{e}_{\cre,\labar}. \label{ineq1:qomlabar}
\end{align}

\begin{corollary}\label{cor1}
Let $(\uu^h_{ex}, \cont^h_{ex}) \in \Ucb \times \Scb$ be the exact solution of problem $(\Pexh)$ and $(\uu, \cont) \in \Ucb \times \Scb$ the one of the reference problem. Given an admissible approximate solution $(\hat{\uu}_h, \hat{\cont}_h)$ of the reference problem, the following equality holds:
\begin{equation}\label{eq1:cor1}
\lnorm{\cont - \hat{\cont}_h^m}_{\cont, \omlabar} = \lnorm{\cont - \cont^h_{ex} - \frac{1}{2} \left( \hat{\cont}_h - \K \: \defo(\hat{\uu}_h) \right)}_{\cont, \omlabar}
\end{equation}
\end{corollary}

\begin{proof}
Using result (\ref{eq1:prop3}) of \Prop{prop3}, one obtains:
\begin{align}
\lnorm{\cont - \hat{\cont}_h^m}^2_{\cont, \omlabar} %& = \lnorm{\cont - \cont^h_{ex} + \cont^h_{ex} - \hat{\cont}_h^m}^2_{\cont, \omlabar} \nonumber\\
& = \lnorm{\cont - \cont^h_{ex}}^2_{\cont, \omlabar} + \lnorm{\cont^h_{ex} - \hat{\cont}_h^m}^2_{\cont, \omlabar} + 2 \lscalproda{\cont - \cont^h_{ex}}{\cont^h_{ex} - \hat{\cont}_h^m}_{\cont, \omlabar} \nonumber\\
& = \lnorm{\cont - \cont^h_{ex}}^2_{\cont, \omlabar} + \frac{1}{4} \lnorm{\hat{\cont}_h - \K \: \defo(\hat{\uu}_h)}^2_{\cont, \omlabar} + 2 \lscalproda{\cont - \cont^h_{ex}}{\cont^h_{ex} - \hat{\cont}_h^m}_{\cont, \omlabar} \label{eq1:proofcor1}
\end{align}
with
\begin{align}
\lscalproda{\cont - \cont^h_{ex}}{\cont^h_{ex} - \hat{\cont}_h^m}_{\cont, \omlabar} & = \lscalproda{\cont - \cont^h_{ex}}{\cont^h_{ex} - \K \: \defo(\hat{\uu}_h)}_{\cont, \omlabar} - \lscalproda{\cont - \cont^h_{ex}}{\frac{1}{2} \left( \hat{\cont}_h - \K \: \defo(\hat{\uu}_h) \right)}_{\cont, \omlabar}, \nonumber\\
\end{align}
%and, using the Dirichlet boundary conditions over $\domlabar$ satisfied by $\uu^h_{ex}$ and the strong form of equilibrium equations verified by $\cont$ and $\cont^h_{ex}$,
%\begin{align}
%\lscalproda{\cont - \cont^h_{ex}}{\cont^h_{ex} - \K \: \defo(\hat{\uu}_h)}_{\cont, \omlabar} & = \intolabar \Tr\big[(\cont - \cont^h_{ex}) \: \defo(\uu^h_{ex} - \hat{\uu}_h) \big] \dO \nonumber\\
%& = \intdolabar (\cont - \cont^h_{ex}) \: \und{n} \cdot \left( \uu^h_{ex} - \hat{\uu}_h \right) \dS - \intolabar \und{\diver}(\cont - \cont^h_{ex}) \cdot (\uu^h_{ex} - \hat{\uu}_h) \dO \nonumber\\
%& = 0. \nonumber
%\end{align}
and
\begin{align}
\lscalproda{\cont - \cont^h_{ex}}{\cont^h_{ex} - \K \: \defo(\hat{\uu}_h)}_{\cont, \omlabar} & = \intolabar \Tr\big[(\cont - \cont^h_{ex}) \: \defo(\uu^h_{ex} - \hat{\uu}_h) \big] \dO = 0.
\end{align}

Eventually, we end up the proof of \Cor{cor1} with:
\begin{align}
\lnorm{\cont - \hat{\cont}_h^m}^2_{\cont, \omlabar} & = \lnorm{\cont - \cont^h_{ex}}^2_{\cont, \omlabar} + \lnorm{\frac{1}{2} \left( \hat{\cont}_h - \K \: \defo(\hat{\uu}_h) \right)}^2_{\cont, \omlabar} - 2 \lscalproda{\cont - \cont^h_{ex}}{\frac{1}{2} \left( \hat{\cont}_h - \K \: \defo(\hat{\uu}_h) \right)}_{\cont, \omlabar} \nonumber\\
& = \lnorm{\cont - \cont^h_{ex} - \frac{1}{2} \left( \hat{\cont}_h - \K \: \defo(\hat{\uu}_h) \right)}^2_{\cont, \omlabar}. \label{eq2:proofcor1}
\end{align}
\end{proof}

It follows from result (\ref{eq1:cor1}) of Corollary~\ref{cor1}:
\begin{align}
\labs{q_{\labar}} & \leqslant \left[ \lnorm{\cont - \hat{\cont}_h^m}_{\cont, \omlabar} + \frac{1}{2} e_{\cre,\labar} \right] \tilde{\theta}_{\labar} + \frac{1}{2} \: e_{\cre,\labar} \: \tilde{e}_{\cre,\labar}; \label{ineq2:qomlabar}
\end{align}
it can be rewritten in the following form:
\begin{align}
\labs{q_{\labar}} & \leqslant \lnorm{\cont - \hat{\cont}_h^m}_{\cont, \omlabar} \tilde{\theta}_{\labar} + \frac{1}{2} \: e_{\cre,\labar} \left[ \tilde{\theta}_{\labar} + \tilde{e}_{\cre,\labar} \right]; \label{ineq3:qomlabar}
\end{align}
thereby getting back to the bounding (\ref{eq3:q}) of quantity $q$ by using bounding results (\ref{ineq3:qomlabar}) and (\ref{ineq1:qomminuslabar}) for $q_{\labar}$ and $q_{\setminus \labar}$, respectively, one obtains:
\begin{align}
\labs{q} & \leqslant \lnorm{\cont - \hat{\cont}_h^m}_{\cont, \omlabar} \tilde{\theta}_{\labar} + \frac{1}{2} \: e_{\cre,\labar} \left[ \tilde{\theta}_{\labar} + \tilde{e}_{\cre,\labar} \right] + \lnorm{\cont - \hat{\cont}_h^m}_{\cont, \omminuslabar} \tilde{e}_{\cre,\setminus \labar}. \label{ineq2:q}
\end{align}

Let us now introduce a scalar $\nu > 0$ in order to shrewdly regroup $\lnorm{\cont - \hat{\cont}_h^m}_{\cont, \omlabar}$ and $\lnorm{\cont - \hat{\cont}_h^m}_{\cont, \omminuslabar}$, which are such that:
\begin{align}
\lnorm{\cont - \hat{\cont}_h^m}^2_{\cont, \omlabar} + \lnorm{\cont - \hat{\cont}_h^m}^2_{\cont, \omminuslabar} = \lnorm{\cont - \hat{\cont}_h^m}^2_{\cont, \Om} = \frac{1}{4} \: e^2_{\cre};
\end{align}
one has:
\begin{align}
\lnorm{\cont - \hat{\cont}_h^m}_{\cont, \omlabar} \tilde{\theta}_{\labar} + \lnorm{\cont - \hat{\cont}_h^m}_{\cont, \omminuslabar} \tilde{e}_{\cre,\setminus \labar} %& = \displaystyle\frac{1}{\sqrt{\nu}} \lnorm{\cont - \hat{\cont}_h^m}_{\cont, \omlabar} \tilde{\theta}_{\labar} \sqrt{\nu} + \frac{1}{\sqrt{\nu}} \lnorm{\cont - \hat{\cont}_h^m}_{\cont, \omminuslabar} \tilde{e}_{\cre,\setminus \labar} \sqrt{\nu} \nonumber\\
%& \leqslant \frac{1}{2} \underbrace{\left[ \displaystyle\frac{1}{\nu} \lnorm{\cont - \hat{\cont}_h^m}^2_{\cont, \omlabar} + \nu \: \tilde{\theta}^2_{\labar} + \frac{1}{\nu} \lnorm{\cont - \hat{\cont}_h^m}^2_{\cont, \omminuslabar} + \nu \: \tilde{e}^2_{\cre,\setminus \labar} \right]}_{\displaystyle\frac{1}{\nu} \lnorm{\cont - \hat{\cont}_h^m}^2_{\cont, \Om} +\nu\left( \tilde{\theta}^2_{\labar} + \tilde{e}^2_{\cre,\setminus \labar} \right)}. \label{ineq1:determinationnu}
& \leqslant \frac{1}{2} \left[ \displaystyle\frac{1}{\nu} \lnorm{\cont - \hat{\cont}_h^m}^2_{\cont, \Om} +\nu\left( \tilde{\theta}^2_{\labar} + \tilde{e}^2_{\cre,\setminus \labar} \right) \right]. \label{ineq1:determinationnu}
\end{align}
The optimal scalar $\nu$ is the one which minimizes the right-hand side of (\ref{ineq1:determinationnu}):
\begin{align}
\nu = \argmin_{\substack{ \nu^{\ast} > 0 }} \left\{ \frac{1}{\nu^{\ast}} \lnorm{\cont - \hat{\cont}_h^m}^2_{\cont, \Om} +\nu^{\ast}\left( \tilde{\theta}^2_{\labar} + \tilde{e}^2_{\cre,\setminus \labar} \right) \right\}.
\end{align}
It reads:
\begin{align}\label{eq1:scalarnu}
\nu = \frac{\lnorm{\cont - \hat{\cont}_h^m}_{\cont, \Om}}{\sqrt{\tilde{\theta}^2_{\labar} + \tilde{e}^2_{\cre,\setminus \labar}}} = \frac{1}{2} \frac{e_{\cre}}{\sqrt{\tilde{\theta}^2_{\labar} + \tilde{e}^2_{\cre,\setminus \labar}}},
\end{align}
which is a strictly positive scalar.

Then, changing scalar $\nu$ by expression (\ref{eq1:scalarnu}) into inequation (\ref{ineq1:determinationnu}), one gets:
\begin{align}
\labs{q} & \leqslant \frac{1}{2} \: e_{\cre} \left[ \tilde{\theta}^2_{\labar} + \tilde{e}^2_{\cre,\setminus \labar} \right]^{1/2} + \frac{1}{2} \: e_{\cre,\labar} \left[ \tilde{\theta}_{\labar} + \tilde{e}_{\cre,\labar} \right]. \label{ineq3:q}
\end{align}
which proves result (\ref{eq3:localerrorbounding}) of \Prop{prop5}.
\end{proof}

Thus, this improved technique provides the following guaranteed lower and upper bounds $(\zetainf,\zetasup)$ of $\Iex$:
\begin{align}\label{eq2:improvedlocalerrorbounding}
\zetainf \leqslant \Iex \leqslant \zetasup,
\end{align}
with
\begin{align}
\zetainf = \Ih + \Ihh - \frac{1}{2} \labs{ e_{\cre} \left[ \tilde{\theta}^2_{\labar} + \tilde{e}^2_{\cre,\setminus \labar} \right]^{1/2} + e_{\cre,\labar} \left[ \tilde{\theta}_{\labar} + \tilde{e}_{\cre,\labar} \right] }; & \label{eq2:improvedlocalerrorboundszetainf} \\
\zetasup = \Ih + \Ihh + \frac{1}{2} \labs{ e_{\cre} \left[ \tilde{\theta}^2_{\labar} + \tilde{e}^2_{\cre,\setminus \labar} \right]^{1/2} + e_{\cre,\labar} \left[ \tilde{\theta}_{\labar} + \tilde{e}_{\cre,\labar} \right] }. & \label{eq2:improvedlocalerrorboundszetasup}
\end{align}

\begin{remark}
These bounds involve only one parameter $\labar$, while bounds $(\chiinf,\chisup)$ defined by (\ref{eq1:improvedlocalerrorboundschiinf}) and (\ref{eq1:improvedlocalerrorboundschisup}) in \Sect{4.3} depend on two parameters $(\la,\labar)$. In practice, one can determine an optimum value $\labar_{\opt}$ for parameter $\labar$ that minimizes upper bound (\ref{ineq3:q}) of $q$ in order to obtain very precise bounds $(\zetainf,\zetasup)$ of $\Iex$. Besides, practically subdomain $\omlabar$ should recover the zone where the solution of adjoint problem has stiff gradients.
\end{remark}

\begin{remark}\label{rem1:k}
The constant $k$ could be easily computed, $\Rcb_{\labar}(\vu)$ being a Rayleigh quotient associated with a symmetric eigenproblem. In the two-dimensional (three-dimensional, respectively) case of a circle\footnotemark[1] (sphere\footnotemark[1], respectively) $\omlabar$, it has been shown numerically that the minimum of $\Rcb_{\labar}$ is reached for $\vu = \OM$, where $O$ is the homothetic center and $M \in \domlabar$; furthermore, for a circular and spherical geometric shapes, it has been shown in \cite{Hoc92} (see \Appendix{} D); it follows that $\defo(\vu) = \Id$. Let us note that in the two-dimensional case of a cracked circle and of a double square\footnotemark[2], as well as in the three-dimensional case of a double unit parallelepiped\footnotemark[2], the same minimum eigenfunction for $\Rcb_{\labar}$ has been obtained numerically.
\footnotetext[1]{circle (sphere, respectively) $\omlabar$ corresponds to a circular (spherical, respectively) domain of radius $\labar$.}
\footnotetext[2]{double square (parallelepiped, respectively) $\omlabar$ corresponds to a squared (parallelepiped, respectively) domain of side length $2 \labar$.}

For all aforementioned two- (three-, respectively) dimensional cases, constant $k$ is equal to $2$ ($3$, respectively).
%\begin{itemize}
%\item In two dimensions, one has:
%\begin{enumerate}
%\item[$\circ$] in the case of a unit circle:
%\begin{align}
%k_{\text{unit circle}} = \Rcb_{\labar}(\vu) & = \frac{4 (\mu + \lambda) \times 2 \pi \times \labar}{4 (\mu + \lambda) \times \pi \times \labar^2} \nonumber\\
%& = 2. \label{eq1:circulark}
%\end{align}
%\item[$\circ$] in the case of a unit cracked circle and of a double unit square, one can show that constant $k$ is equal to $k_{\text{unit circle}}$.
%\end{enumerate}
%\item In three dimensions, one gets:
%\begin{enumerate}
%\item[$\circ$] in the case of a unit sphere:
%\begin{align}
%k_{\text{unit sphere}} = \Rcb_{\labar}(\vu) & = \frac{(6 \mu + 9 \lambda) \times 4 \pi \times \labar^2}{(6 \mu + 9 \lambda) \times \displaystyle\frac{4}{3} \pi \times \labar^3} \nonumber\\
%& = 3. \label{eq1:sphericalk}
%\end{align}
%\item[$\circ$] in the case of a double unit parallelepiped, one can show that constant $k$ is equal to $k_{\text{unit sphere}}$.
%\end{enumerate}
%\end{itemize}
Note that, in these particular shape domains, constant $k$ does not depend on any material parameter.
\end{remark}

\begin{remark}
In the present work, we restrict to only one subdomain, but the results could be easily extended to several ones. Multiple subdomains could be associated to different zones where the local CRE is relatively high either for adjoint or reference problem.
\end{remark}

%: 6. Numerical results
\section{Numerical results}\label{6}

All numerical experiments have been performed assuming that the material remains isotropic, homogeneous, linear and elastic with Young's modulus $E = 1$ and Poisson's ratio $\nu = 0.3$. Furthermore, the two-dimensional examples are assumed to satisfy the plane-stress approximation. The balance technique used to derive a statically admissible stress field is the element equilibration technique (EET) combined with a $p$-refinement technique consisting of a $p+k$ discretization, $p$ being the FE interpolation degree and $k$ an additional degree equal to $3$ (see principles of the different techniques for constructing admissible stress fields in \cite{Ple11} for more information).

Performances of the proposed bounding techniques are illustrated through a two-dimensional cracked structure, already considered in \cite{Par06,Lad10bis,Ple12}.

The results obtained for classical bounding technique as well as first and second improved variants are presented in terms of the normalized bounds $(\xibarinf,\xibarsup)$, $(\chibarinf,\chibarsup)$, $(\zetabarinf,\zetabarsup)$, respectively, defined by:
\begin{align}
& \xibarinf = \frac{\xiinf}{\Iex} \quad \text{and} \quad \xibarsup = \frac{\xisup}{\Iex}; \\
& \chibarinf = \frac{\chiinf}{\Iex} \quad \text{and} \quad \chibarsup = \frac{\chisup}{\Iex}; \\
& \zetabarinf = \frac{\zetainf}{\Iex} \quad \text{and} \quad \zetabarsup = \frac{\zetasup}{\Iex}.
\end{align}

%: 6.1.
\subsection{Presentation of the 2D cracked structure}\label{6.1}

Let us consider the two-dimensional structure shown in \Fig{fig1:structure_fissuree_2D_geometry_fine_mesh}, which presents two round cavities. A homogeneous Dirichlet boundary condition is imposed to the bigger circular hole, whereas a unit internal constant pressure $p_0$ is applied to the smaller one. Furthermore, the top-left edge is subjected to a unit normal traction force density $\und{t} = +\und{n}$. Besides, a single edge crack emanates from the bottom of the smaller cavity. The two lips of this crack as well as the remaining sides are traction-free boundaries. The FE mesh consists of $7 \, 751$ linear triangular elements and $4 \, 122$ nodes (\ie $8 \, 244$ d.o.f.), see \Fig{fig1:structure_fissuree_2D_geometry_fine_mesh}. It has been adaptively refined in the vicinity of the crack tip. The reference mesh used to compute an ``overkill'' solution and to define a ``quasi-exact" value, denoted $\Iex$ for convenience, of the quantity of interest is built up by dividing each element into $256$ elements; thereby, it is made of $1 \, 984 \, 256$ linear triangular elements and $996 \, 080$ nodes (\ie $1 \, 992 \, 160$ d.o.f.).

\begin{figure}
\centering\includegraphics[scale = 0.4]{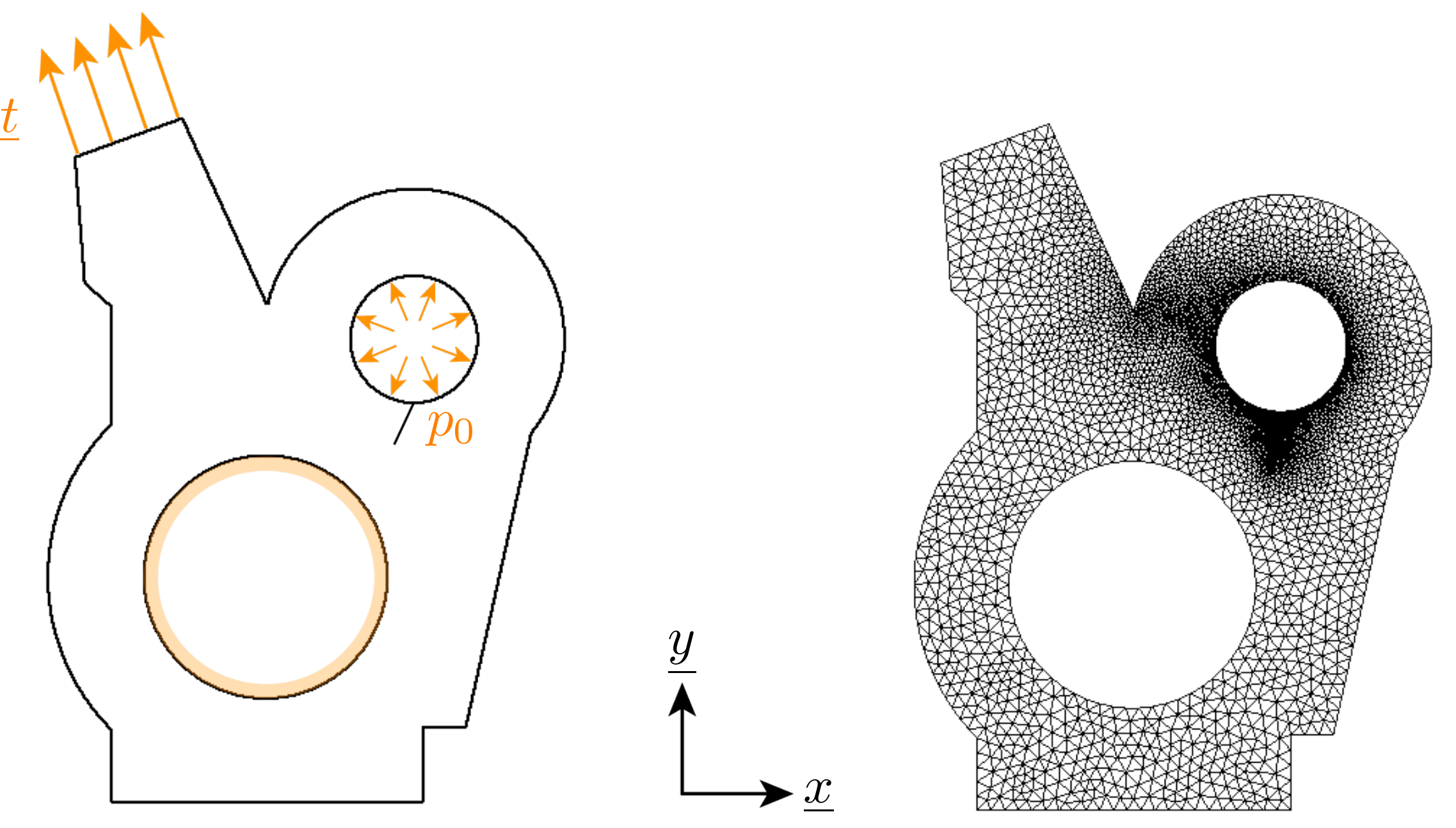}
\caption{Cracked structure model problem (left) and associated finite element mesh (right).}\label{fig1:structure_fissuree_2D_geometry_fine_mesh}
\end{figure}

The quantities of interest being considered in this work are:
\begin{itemize}
\item the average value in a local zone $\om \subset \Om$ of the component $(.)_{xx}$ of the stress field $\cont$:
\begin{equation}\label{eq2:interestquantity}
I_1 = \bar{\sigma}_{xx} = \frac{1}{\labs{\om}} \into \sigma_{xx} \dO,
\end{equation}
where extraction domain $\om$ corresponds to a finite element $E$ of FE mesh $\Mc_h$ illustrated in \Fig{fig1:structure_fissuree_fine_Triangle_EF_mean_sigma_xx_elem_4886_and_pointwise_dep_x_node_661_interest_zone_orange_blue} and $\labs{\om}$ represents its measure;
\item the pointwise value of the component $(.)_x$ of the displacement field $\uu$ at a point $P$:
\begin{equation}\label{eq3:interestquantity}
I_2 = u_x(P),
\end{equation}
where point $P$ coincides with a node of FE mesh $\Mc_h$ illustrated in \Fig{fig1:structure_fissuree_fine_Triangle_EF_mean_sigma_xx_elem_4886_and_pointwise_dep_x_node_661_interest_zone_orange_blue}, the corresponding extraction domain $\om$ being pointwise;
\item the stress intensity factor $\KI$ involved in the finite-energy analytical asymptotic expression of the stress field $\cont$ in the vicinity of the crack tip \cite{Gal06,Pan10} and classically used in crack propagation criteria for linear elastic fracture mechanics (LEFM) problems:
\begin{equation}\label{eq4:interestquantity}
I_3 = \KI.
\end{equation}
\end{itemize}

\begin{figure}
\centering\includegraphics[scale = 0.4]{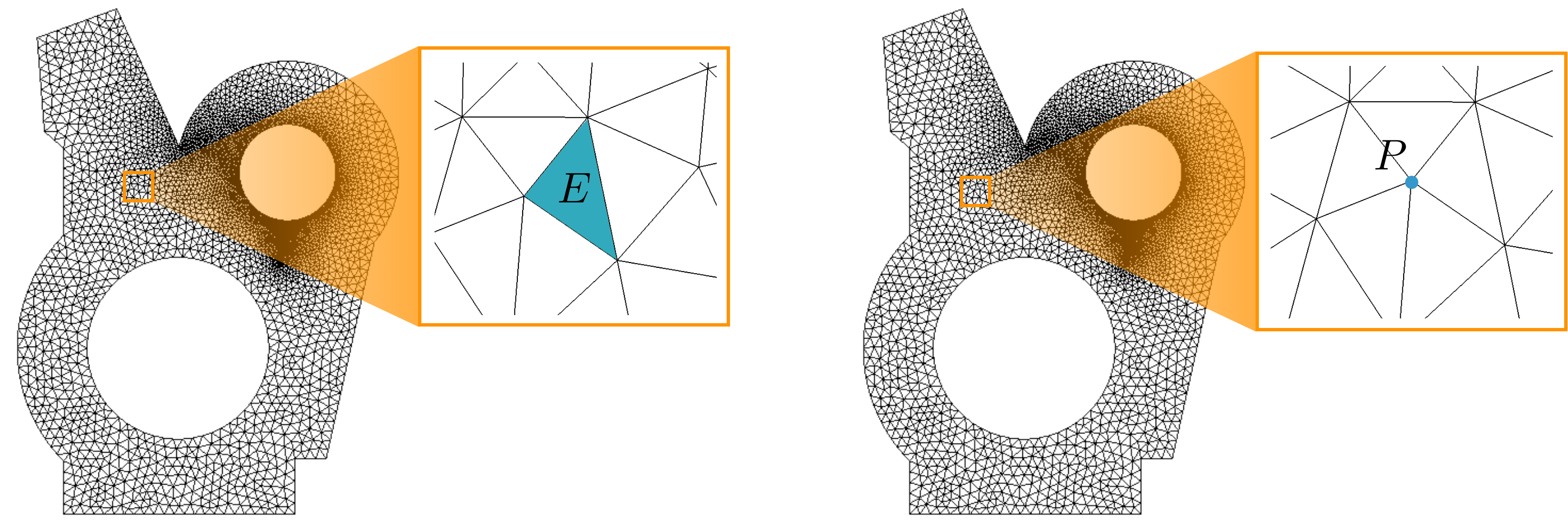}
\caption{Position of element $E$ (left) and point $P$ (right) in the FE mesh defining the zones of interest for local quantities $I_1$ and $I_2$, respectively.}\label{fig1:structure_fissuree_fine_Triangle_EF_mean_sigma_xx_elem_4886_and_pointwise_dep_x_node_661_interest_zone_orange_blue}
\end{figure}

All the local quantities being considered are linear functions of displacement field $\uu$ associated to reference problem. Exact values $\Iex$ and approximate values $\Ih$ obtained for the three considered quantities of interest are reported in \Tab{table1:quantities_interest}.

\begin{table}
\caption{Calculated values of the quantities of interest}
\centering
\tabsize
\begin{tabular}{l c c}
\toprule
Quantity of interest & Exact value \ $\Iex$ & Approximate value \ $\Ih$\\
\midrule
$I_1$ & $0.0347969$ & $0.0303747$ \\
$I_2$ & $-16.2939$ & $-16.1596$ \\
$I_3$ & $2.8974$ & $2.86161$ \\
\bottomrule
\end{tabular}
\label{table1:quantities_interest}
\end{table}

In the case of pointwise quantity of interest $I_2$, for which the loading of the adjoint problem is a pointwise force defined using a Dirac-type function applied to a node of FE mesh $\Mc_h$, a natural regularization would consist in treating the pointwise force as a nodal force $\tilde{\und{f}}_{\Sigma} = \delta(\und{x}_P) \tilde{\und{F}}_{\Sigma}$, where $\tilde{\und{F}}_{\Sigma} = \displaystyle\begin{bmatrix}
1 \\
0
\end{bmatrix}$ and $\und{x}_P$ denotes the position of point $P$. Let us note that this approach can be employed only in the case of pointwise values of the displacement field related to the position of a node $P$ of the FE mesh and leads to a coarse approximate solution of the adjoint problem. Another classical way to handle truly pointwise quantities of interest is to have recourse to the use of mollification, also called mollifying process, replacing the initial quantity of interest by a local weighted average value \cite{Pru99}. A last approach, henceforth known as handbook techniques \cite{Cha08,Cha09,Lad10}, consists in introducing a local enrichment of the solution of the adjoint problem particularly well-suited to pointwise error estimation without performing any regularization of the pointwise quantity of interest being considered. The singularities involved in the loading of the adjoint problem are captured explicitly in a non-intrusive way introducing adapted enrichment functions, also called handbook functions and denoted $(\tilde{\uu}^{hand},\tilde{\cont}^{hand})$, through the PUM (satisfied by the linear FE shape functions) allowing to ensure displacement compatibility. These functions represent local (quasi-)exact solutions of the adjoint problem. In the context of pointwise error estimation, they correspond to the well-known and possibly infinite-energy Green functions describing the singular solutions of the adjoint problem over an infinite (or semi-infinite) domain. Several examples of such functions and a detailed description of handbook techniques can be found in \cite{Cha08,Cha09,Lad10}. The Green functions can be determined analytically using approaches based on strain nuclei combined with the image method \cite{Min36,Min50,Cou53,Gra75,Vij87}; for quantity of interest $I_2$, handbook functions are given in \Fig{fig1:cercle_pointwise_force_x_dep_sigma_handbook}.

\begin{figure}
\centering\includegraphics[scale = 0.4]{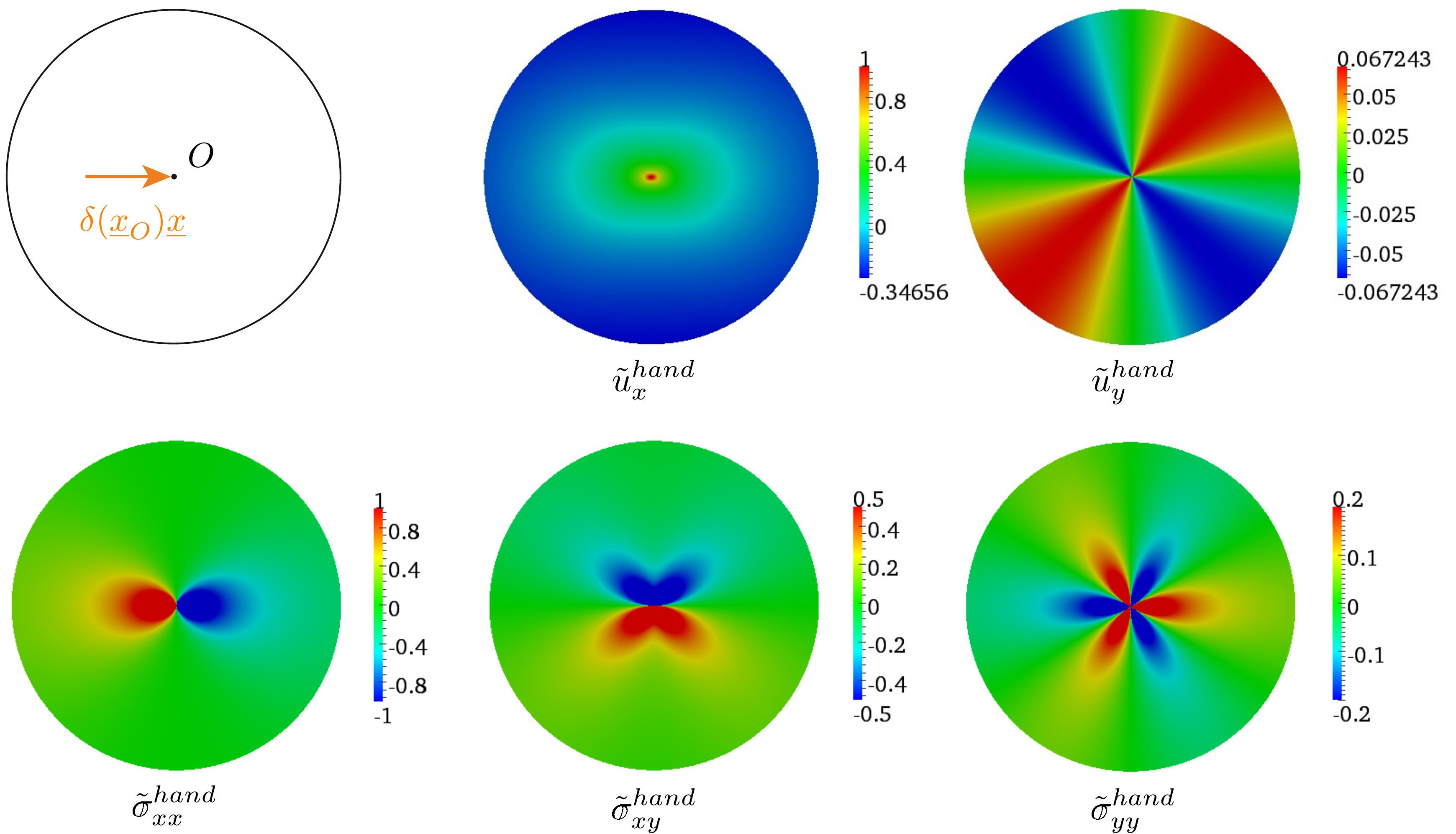}
\caption{Spatial distribution of handbook pair $(\tilde{\uu}^{hand},\tilde{\cont}^{hand})$ associated to a pointwise force loading $\delta(\und{x}_O) \und{x}$ applied to point $O$ over an infinite domain, $\und{x}_O$ being the position of point $O$.}\label{fig1:cercle_pointwise_force_x_dep_sigma_handbook}
\end{figure}

Enrichment is applied via the PUM to a region $\Om^{PUM}$ of domain $\Om$, which is decomposed into two non-overlapping subregions: a zone $\Om_1^{PUM}$ containing the zone of interest $\om$ over which the quantity of interest is defined and a complementary zone $\Om_2^{PUM}$ surrounding $\Om_1^{PUM}$. In the case of quantity of interest $I_2$, an enrichment by the PUM involving only one or two layers of nodes is sufficient to capture the local high gradients of the exact solution $\tilde{u}$ of the adjoint problem; the definition of subregions $\Om_1^{PUM}$ (containing $6$ (resp. $24$) elements and $7$ (resp. $19$) enriched nodes) and $\Om_2^{PUM}$ (containing $18$ (resp. $30$) elements) in region $\Om^{PUM}$ is given in \Fig{fig1:structure_fissuree_fine_Triangle_EF_pointwise_dep_x_node_661_enriched_zone_PUM_12} by considering a single layer (resp. two layers) of nodes involved in the enrichment.

\begin{figure}
\centering\includegraphics[scale = 0.4]{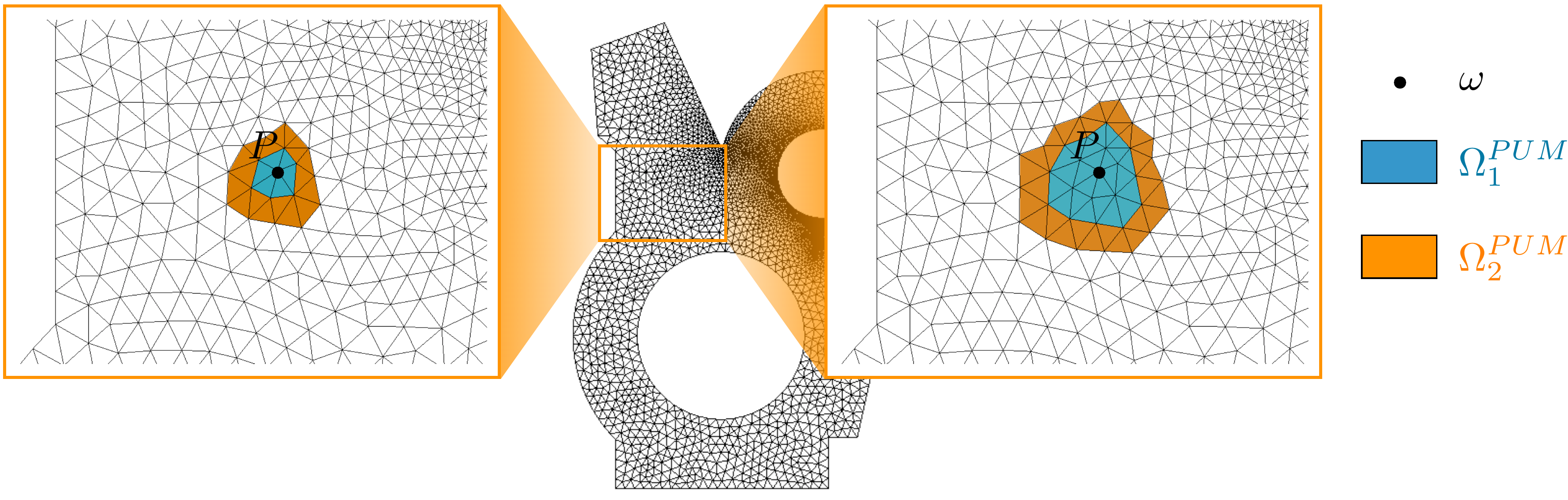}
\caption{Definition of the zone of interest $\om$, the subregions $\Om_1^{PUM}$ and $\Om_2^{PUM}$ in the enriched region $\Om^{PUM}$ for quantity of interest $I_2$; enrichment with the PUM is applied to one layer (left) or two layers (right) of nodes.}\label{fig1:structure_fissuree_fine_Triangle_EF_pointwise_dep_x_node_661_enriched_zone_PUM_12}
\end{figure}

Finally, the global solution $(\tilde{\uu},\tilde{\cont})$ of the adjoint problem is composed of an enrichment part $(\tilde{\uu}^{hand}_{PUM},\tilde{\cont}^{hand}_{PUM})$ introduced a priori either analytically or numerically and a residual part $(\tilde{\uu}^{res},\tilde{\cont}^{res})$ computed a posteriori numerically by using the FEM. The new adjoint problem consists in finding the residual pair $(\tilde{\uu}^{res},\tilde{\cont}^{res})$. The corresponding force vector involved in the equilibrium equations comes down to a traction load $- \tilde{\cont}^{hand} \: \und{n}_{12}$ over $\partial \Om_1^{PUM}$, where $\und{n}_{12}$ is the outgoing normal vector to $\Om_1^{PUM}$, and a prestress $- \tilde{\cont}^{hand}_{PUM}$ over $\Om_2^{PUM}$ by using properties of the handbook functions. Let us note that the loading of the adjoint problem becomes finite and smooth and that a correct and satisfactory approximate solution can be merely obtained by using the same spatial discretization as the one employed for the reference problem. Besides, observing that enrichment part $(\tilde{\uu}^{hand}_{PUM},\tilde{\cont}^{hand}_{PUM})$ satisfies the constitutive relation over $\Om$, the CRE $e_{\cre}(\hat{\tilde{\uu}}_h, \hat{\tilde{\cont}}_h)$ of initial approximate adjoint problem is exactly equal to the CRE $e_{\cre}(\hat{\tilde{\uu}}^{res}_h, \hat{\tilde{\cont}}^{res}_h)$ of residual approximate adjoint problem.

In the case of quantity of interest $I_3$, open-mode stress intensity factor $\KI$ (\ie associated with the Mode-I loading) is defined from an extension of the contour integral method proposed in \cite{Ste76}. The interested reader can refer to \cite{Ste76,Bab84,Gal06,Pan09,Pan10} for a detailed description of this technique. The calculation of stress intensity factor $\KI$ comes down to evaluating a surface integral over an arbitrary crown $\Om_c$ surrounding the crack tip:
\begin{align}\label{eq1:KI}
\KI = \intOc \Tr\big[\left( \K \: \defo(\phi \: \vuI) - \phi \: \contI \right) \: \defo(\uu)\big] \dO - \intOc \left( \contI \: \nabla \phi \right) \cdot \uu \dO,
\end{align}
where $(\vuI,\contI)$ corresponds to a singular (infinite-energy) analytical solution of the elastic problem defined near the crack tip and $\phi$ is an arbitrary continuously derivable scalar function that is equal to $0$ (resp. $1$) along external boundary $\Gamma_e$ (resp. internal boundary $\Gamma_i$) of crown $\Om_c$. One usually considers a circular annulus $\Om_c$ centered around the crack tip and a linear continuous function $\phi$ \cite{Gal06,Pan09,Pan10}, but other polynomial functions $\phi$ could be considered \cite{Bab84}. Here, internal part $\Gamma_i$ and external part $\Gamma_e$ of boundary $\partial \Om_c$ are assumed to be circles centered at the crack tip with radii $R_i = 6$ and $R_e = 8$, respectively; the definition of crown $\Om_c$ is given in \Fig{fig1:structure_fissuree_fine_Triangle_EF_SIF_pos_crack_tip_109_105_R1_6_R2_8_crown}.

\begin{figure}
\centering\includegraphics[scale = 0.4]{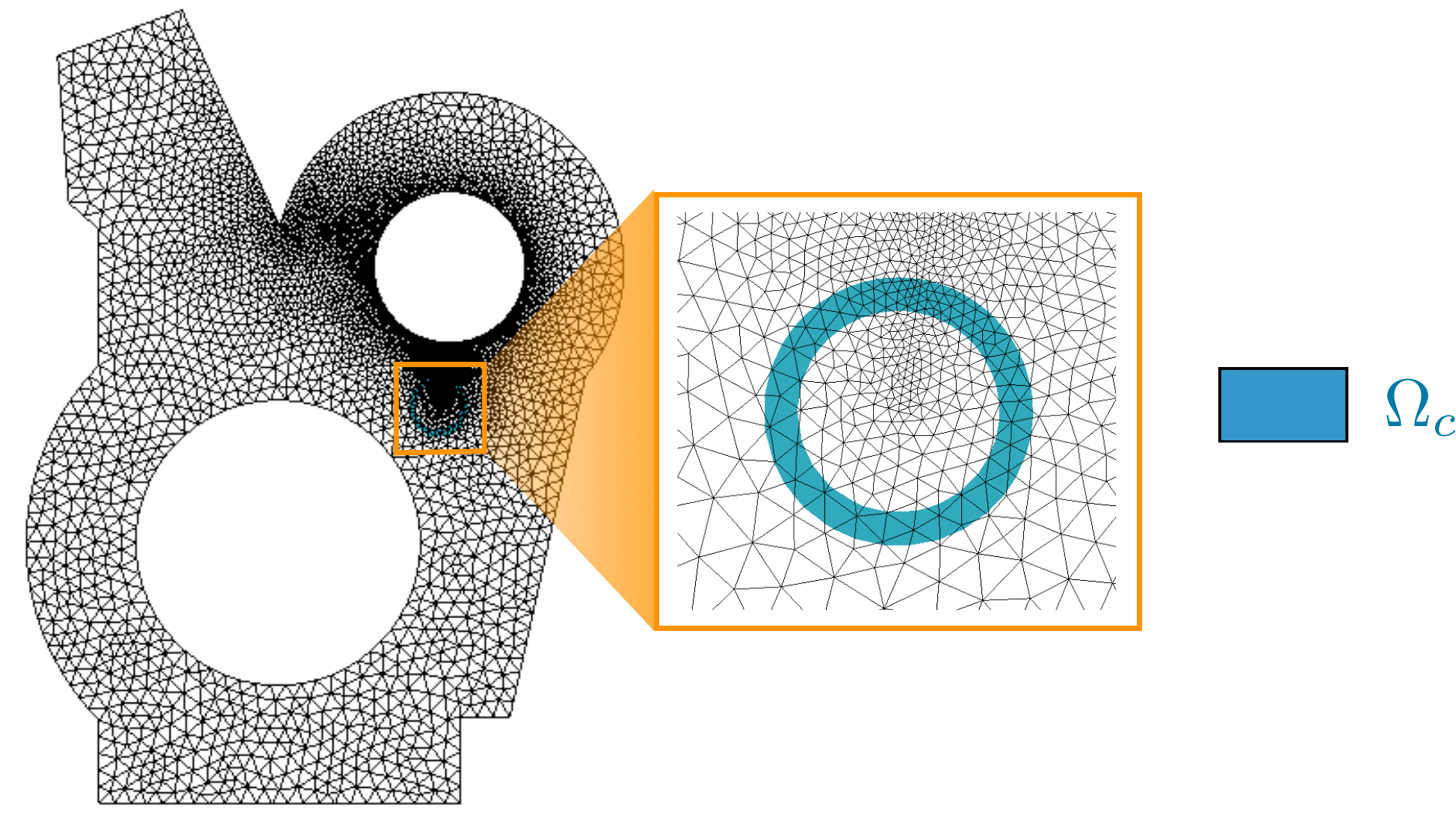}
\caption{Definition of the crown $\Om_c$ for quantity of interest $I_3$.}\label{fig1:structure_fissuree_fine_Triangle_EF_SIF_pos_crack_tip_109_105_R1_6_R2_8_crown}
\end{figure}

Function $\phi$ is chosen to be linear with radius $r$ and is defined as:
\begin{equation}
\begin{cases}
\phi(r) = 1 & \text{if} \ 0 \leqslant r \leqslant R_i\\
\phi(r) = \displaystyle\frac{R_e - r}{R_e - R_i} & \text{if} \ R_i \leqslant r \leqslant R_e\\
\phi(r) = 0 & \text{if} \ r \geqslant R_e.
\end{cases}
\end{equation}
Let us mention that taking a polynomial function $\phi$ with a degree at most compatible with that of the FE analysis is appropriate, since it can be correctly represented in FE framework. The effectiveness and accuracy of the classical bounding technique to derive lower and upper bounds for open- and shear-mode stress intensity factors with a standard FEM have been demonstrated on two-dimensional benchmark problems in \cite{Gal06}.

The loading of the adjoint problems involves the following extractors:
\begin{itemize}
\item a uniform prestress field $\tilde{\cont}_{\Sigma} = \K \: \tilde{\defo}_{\Sigma}$ over element $E$, where $\tilde{\defo}_{\Sigma} = \displaystyle\frac{1}{\labs{E}} 
\begin{bmatrix}
1 & 0 \\
0 & 0
\end{bmatrix}$ and $\labs{E}$ denotes the measure of element $E$, for quantity of interest $I_1$;
\item a traction load $\tilde{\und{F}}_{\Sigma} = - \tilde{\cont}^{hand} \: \und{n}_{12}$ over $\partial \Om_1^{PUM}$ and a prestress field $\tilde{\cont}_{\Sigma} = - \tilde{\cont}^{hand}_{PUM}$ over $\Om_2^{PUM}$ for quantity of interest $I_2$;
\item a prestress field $\tilde{\cont}_{\Sigma} = \K \: \defo(\phi \: \vuI) - \phi \: \contI$ and a body force field $\tilde{\und{f}}_{\Sigma} = - \contI \: \nabla \phi$ over crown $\Om_c$ for quantity of interest $I_3$.
\end{itemize}

In this work, circular shape domains are considered for quantities of interest $I_1$ and $I_2$, while cracked circular shape domains are used for quantity of interest $I_3$, but other geometric shapes could have been investigated, the main technical aspect being the calculation of the constants $h$ and $k$ involved in the first and second improved techniques, respectively. Values of constants $h$ and $k$ have been calculated analytically and computed numerically for different geometric shapes by considering an isotropic, homogeneous, linear and elastic material with Poisson's ratio $\nu = 0.3$. Results are given in \Tab{table1:values_constants}.

\begin{table}
\caption{Values of constants $h$ and $k$ involved in both optimized improvements for different geometric shapes}
\centering
\tabsize
\begin{tabular}{l c c c}
\toprule
Geometric shape & \multicolumn{2}{c}{Constant \ $h$} & Constant \ $k$\\
\midrule
Dim 2 & plane stress assumption & plane strain assumption & \\
\midrule
circle & $0.76923$ & $0.7$ & $2$ \\
cracked circle with $\theta = 0$ & $0.79780$ & $0.72122$ & $2$\\
cracked circle with $\theta = \pi / 6$ & $0.80039$ & $0.72315$ & $2$ \\
cracked circle with $\theta = \pi / 3$ & $0.80351$ & $0.72546$ & $2$ \\
cracked circle with $\theta = \pi / 2$ & $0.80732$ & $0.72829$ & $2$ \\
square & $0.85897$ & $0.76667$ & $2$ \\
\midrule
Dim 3 \\
\midrule
sphere & \multicolumn{2}{c}{$0.53846$} & $3$ \\
parallelepiped & \multicolumn{2}{c}{$0.64103$} & $3$ \\
\bottomrule
\end{tabular}
\label{table1:values_constants}
\end{table}

%: 6.2
\subsection{Average value of a field over a local zone: case of quantity of interest $I_1$}\label{6.2}

Maps of the local contributions to the global error estimates for both reference and adjoint problems are displayed in \Fig{fig1:distribution_estimators_structure_fissuree_2D_mean_sigma_xx_elem_4886}. The adjoint mesh has been slightly refined around the zone of interest $\om$ (corresponding to element $E$), and contains $8 \, 973$ linear triangular elements and $4 \, 733$ nodes (\ie $9 \, 466$ d.o.f.). The main contributions to the error estimate associated to reference problem are by a majority located near the crack tip, while that associated to adjoint problem are concentrated around the zone of interest $\om$. Therefore, the error estimates for both reference and adjoint problems are localized in disjoint regions.

\begin{figure}
\centering\includegraphics[scale = 0.4]{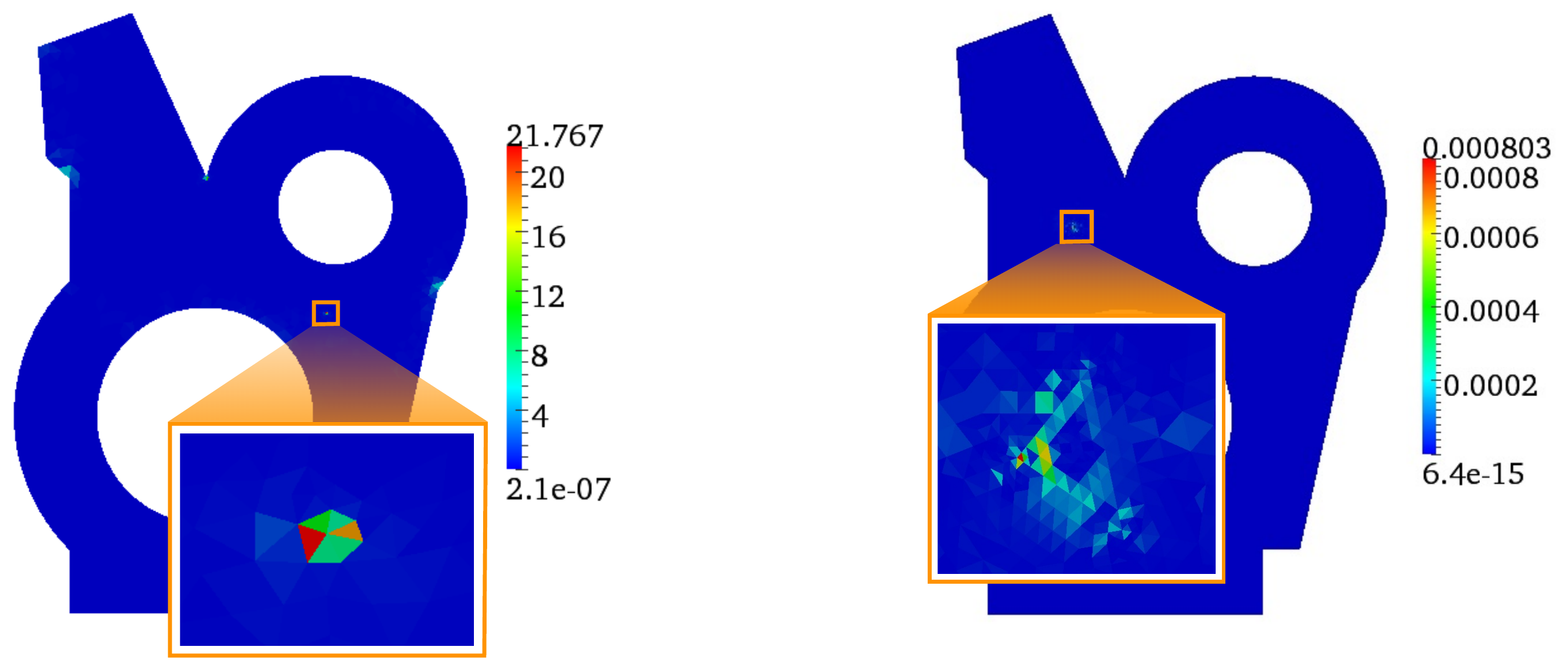}
\caption{Spatial distribution of local contributions to the error estimates associated to reference (left) and adjoint (right) problems related to local quantity $I_1$. Zoom boxes represent the estimated errors in the vicinity of the crack tip for reference problem (left) and around the zone of interest for adjoint problem (right).}\label{fig1:distribution_estimators_structure_fissuree_2D_mean_sigma_xx_elem_4886}
\end{figure}

The values of parameters $\la$ and $\labar$ involved in the first improved technique are set to $2 r$ and $14 r$, respectively, where $r$ corresponds to the radius of the circle circumscribed by element $E$. The value of parameter $\labar_{\opt}$ involved in the second improved technique is set to $9 r$, which enables to achieve the sharpest bounds for quantity of interest $I_1$. The corresponding subdomains $\omla$, $\omlabar$ and $\omlabaropt$ are illustrated in \Fig{fig1:structure_fissuree_fine_Triangle_EF_mean_sigma_xx_elem_4886_homothetic_domains}.

\begin{figure}
\centering\includegraphics[scale = 0.4]{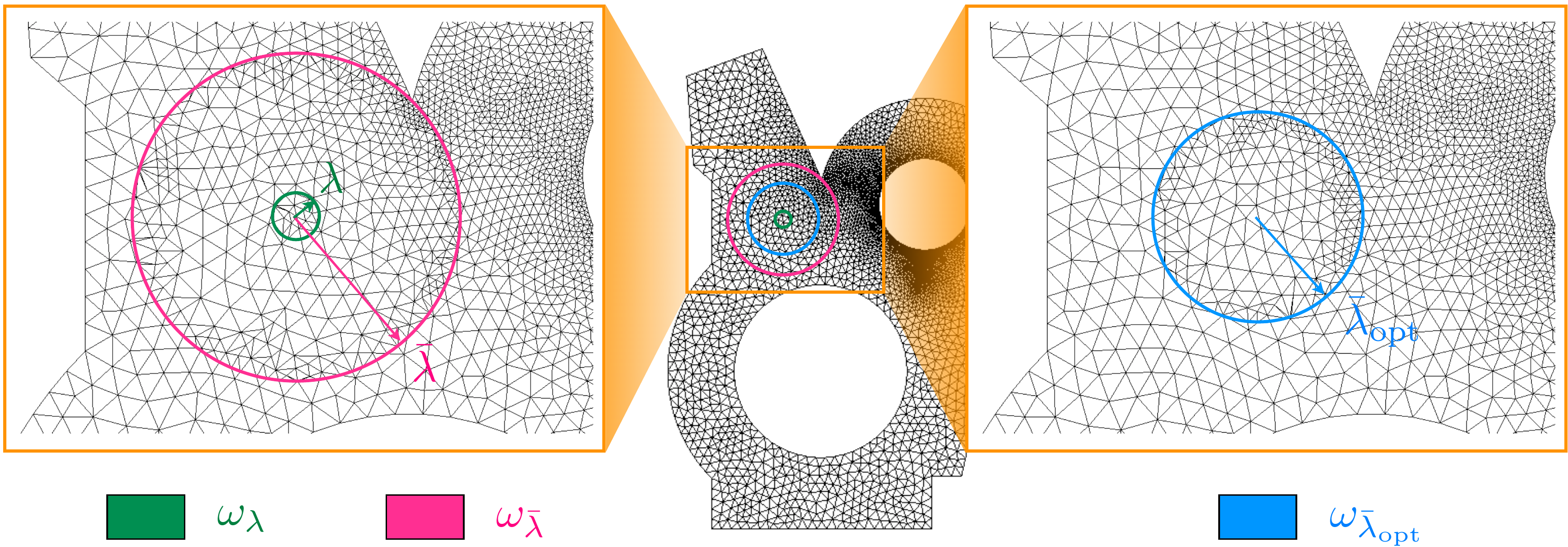}
\caption{Definition of subdomains $\omla$, $\omlabar$ involved in the first improved technique and $\omlabaropt$ involved in the second one for quantity of interest $I_1$.}\label{fig1:structure_fissuree_fine_Triangle_EF_mean_sigma_xx_elem_4886_homothetic_domains}
\end{figure}

\Fig{fig1:normalized_bounds_structure_fissuree_2D_fine_elem_4886_sigma_xx} shows the evolutions of the normalized bounds on $\Iex$ for quantity of interest $I_1$ as functions of the number of elements $\tilde{N}_e$ of adjoint problem for the classical bounding technique as well as the two improved ones. The adjoint mesh has been locally refined near the zone of interest $\om$, since the loading and the contributions to the global error estimate of the adjoint problem are highly localized in this region. One can see a slight improvement in the bounds obtained with the first improved technique compared to the classical one. As regards the second improved technique, a very clear improvement is observed allowing to achieve sharp local error bounds without refining too much the adjoint problem, thus keeping an affordable computing time. Similar results can be obtained for the average values in element $E$ of the other components $(.)_{xy}$ and $(.)_{yy}$ of stress field $\cont$, as well as for the ones of the different components of strain field $\defo$.

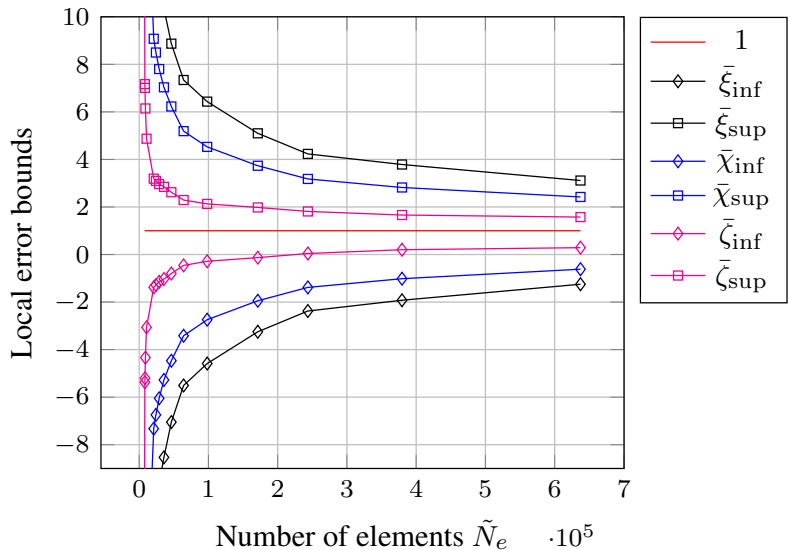
\begin{figure}
\centering
\begin{tikzpicture}[scale=1.2]%[baseline]
\pgfplotsset{
xlabel near ticks,
ylabel near ticks,
label style={font=\small},
tick label style={font=\footnotesize},
legend style={font=\small},
try min ticks=7
}
\begin{axis}[
	width=0.4\textwidth,
	scale only axis,
	ymin=-9,ymax=10,
	%axis y line*=left,
	%axis on top,
	%xlabel=Nombre d'\'el\'ements $\tilde{N}_e$,
	xlabel=Number of elements $\tilde{N}_e$,
	%ylabel=Bornes d'erreur locale,
	ylabel=Local error bounds,
	grid=major,
	legend pos=outer north east,
	legend entries={
	$1$,
	$\xibarinf$, $\xibarsup$,
	$\chibarinf$, $\chibarsup$,
	$\zetabarinf$, $\zetabarsup$,
	%$(\Ih + \Ihh - \sum_E \frac{e_{\cre,E} \: \tilde{e}_{\cre,E}}{2})/\Iex$, $(\Ih + \Ihh + \sum_E \frac{e_{\cre,E} \: \tilde{e}_{\cre,E}}{2})/\Iex$
	}
]
\addplot+[sharp plot,red,solid,mark=none,mark options={red,scale=0.7}] table[x=nb_elem,y expr=\thisrow{I_ex}/\thisrow{I_ex}] {structure_fissuree_2D_fine_elem_4886_sigma_xx.txt};
\addplot+[sharp plot,black,solid,mark=diamond,mark options={black,scale=1.0}] table[x=nb_elem,y expr=(\thisrow{I_h}+\thisrow{I_hh_w}-(\thisrow{theta_EET_EF}*\thisrow{theta_EET_adjoint})/2)/\thisrow{I_ex}] {structure_fissuree_2D_fine_elem_4886_sigma_xx.txt};
\addplot+[sharp plot,black,solid,mark=square,mark options={black,scale=0.7}] table[x=nb_elem,y expr=(\thisrow{I_h}+\thisrow{I_hh_w}+(\thisrow{theta_EET_EF}*\thisrow{theta_EET_adjoint})/2)/\thisrow{I_ex}] {structure_fissuree_2D_fine_elem_4886_sigma_xx.txt};
\addplot+[sharp plot,blue,solid,mark=diamond,mark options={blue,scale=1.0}] table[x=nb_elem,y expr=(\thisrow{I_h}+\thisrow{I_hh_w}+\thisrow{I_hhh}-\thisrow{xi_la_min_2_max_14})/\thisrow{I_ex}] {structure_fissuree_2D_fine_elem_4886_sigma_xx.txt};
\addplot+[sharp plot,blue,solid,mark=square,mark options={blue,scale=0.7}] table[x=nb_elem,y expr=(\thisrow{I_h}+\thisrow{I_hh_w}+\thisrow{I_hhh}+\thisrow{xi_la_min_2_max_14})/\thisrow{I_ex}] {structure_fissuree_2D_fine_elem_4886_sigma_xx.txt};
\addplot+[sharp plot,magenta,solid,mark=diamond,mark options={magenta,scale=1.0}] table[x=nb_elem,y expr=(\thisrow{I_h}+\thisrow{I_hh_w}-\thisrow{zeta_la_9})/\thisrow{I_ex}] {structure_fissuree_2D_fine_elem_4886_sigma_xx.txt};
\addplot+[sharp plot,magenta,solid,mark=square,mark options={magenta,scale=0.7}] table[x=nb_elem,y expr=(\thisrow{I_h}+\thisrow{I_hh_w}+\thisrow{zeta_la_9})/\thisrow{I_ex}] {structure_fissuree_2D_fine_elem_4886_sigma_xx.txt};
%\addplot+[sharp plot,teal,solid,mark=diamond,mark options={teal,scale=1.0}] table[x=nb_elem,y expr=(\thisrow{I_h}+\thisrow{I_hh_w}-\thisrow{sum})/\thisrow{I_ex}] {structure_fissuree_2D_fine_elem_4886_sigma_xx.txt};
%\addplot+[sharp plot,teal,solid,mark=square,mark options={teal,scale=0.7}] table[x=nb_elem,y expr=(\thisrow{I_h}+\thisrow{I_hh_w}+\thisrow{sum})/\thisrow{I_ex}] {structure_fissuree_2D_fine_elem_4886_sigma_xx.txt};
\end{axis}
\end{tikzpicture}
\caption{Evolutions of the lower and upper normalized bounds of $\Iex$ for local quantity $I_1$, obtained using the classical bounding technique as well as first and second improvements, with respect to the number of elements $\tilde{N}_e$ associated to the discretization of the adjoint problem.}\label{fig1:normalized_bounds_structure_fissuree_2D_fine_elem_4886_sigma_xx}
\end{figure}

\Fig{fig1:normalized_interest_quantities_structure_fissuree_2D_fine_elem_4886_sigma_xx} represents the evolutions of the normalized exact value $\Iex/\Iex$ of local quantity $I_1$, its normalized approximate value $\Ih/\Iex$ obtained through the FEM and its new normalized approximate value $(\Ih + \Ihh)/\Iex$ as functions of the number of elements $\tilde{N}_e$ of adjoint problem. One can see that $\Ih + \Ihh$ corresponds to an approximation of better quality with respect to $\Iex$ compared to $\Ih$.

\begin{figure}
\centering
\begin{tikzpicture}[scale=1.2]%[baseline]
\pgfplotsset{
xlabel near ticks,
ylabel near ticks,
label style={font=\small},
tick label style={font=\footnotesize},
legend style={font=\small},
try min ticks=7
}
\begin{axis}[
	width=0.4\textwidth,
	scale only axis,
	%ymin=-9,ymax=10,
	%axis y line*=left,
	%axis on top,
	%xlabel=Nombre d'\'el\'ements $\tilde{N}_e$,
	xlabel=Number of elements $\tilde{N}_e$,
	%ylabel=Bornes d'erreur locale,
	ylabel=Local error bounds,
	grid=major,
	legend pos=outer north east,
	legend entries={
	$1$,
	$\Ih/\Iex$,
	$(\Ih + \Ihh)/\Iex$,
	}
]
\addplot+[sharp plot,red,solid,mark=none,mark options={red,scale=0.7}] table[x=nb_elem,y expr=\thisrow{I_ex}/\thisrow{I_ex}] {structure_fissuree_2D_fine_elem_4886_sigma_xx.txt};
\addplot+[sharp plot,green!60!black,solid,mark=none,mark options={green!60!black,scale=0.7}] table[x=nb_elem,y expr=\thisrow{I_h}/\thisrow{I_ex}] {structure_fissuree_2D_fine_elem_4886_sigma_xx.txt};
\addplot+[sharp plot,orange,solid,mark=o,mark options={orange,scale=0.7}] table[x=nb_elem,y expr=(\thisrow{I_h}+\thisrow{I_hh_w})/\thisrow{I_ex}] {structure_fissuree_2D_fine_elem_4886_sigma_xx.txt};
\end{axis}
\end{tikzpicture}
\caption{Evolutions of the normalized exact value of local quantity $I_1$, its normalized approximate value $\Ih/\Iex$ and its new normalized approximate value $(\Ih + \Ihh)/\Iex$ with respect to the number of elements $\tilde{N}_e$ associated to the discretization of the adjoint problem.}\label{fig1:normalized_interest_quantities_structure_fissuree_2D_fine_elem_4886_sigma_xx}
\end{figure}
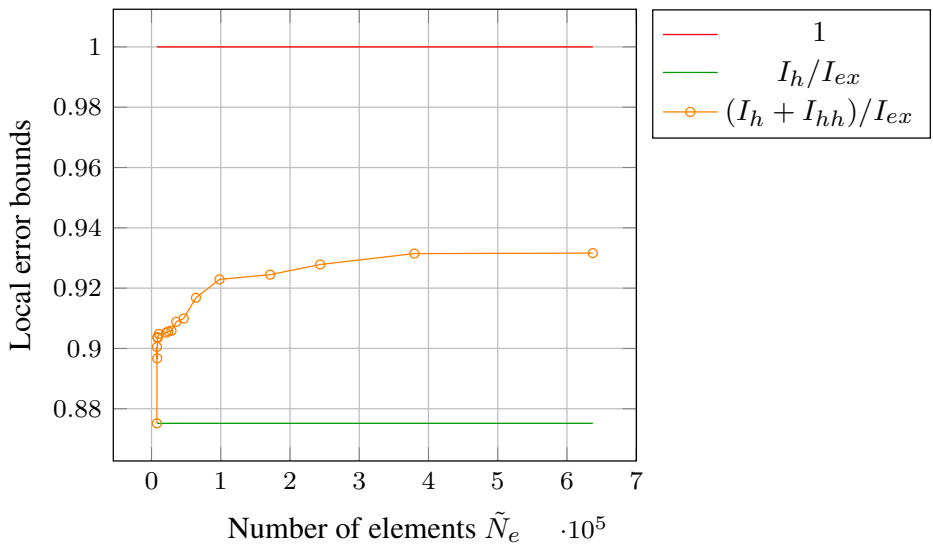

%: 6.3
\subsection{Pointwise value of a field: case of quantity of interest $I_2$}\label{6.3}
In order to properly solve the adjoint problem and to capture the singularities of its solution, specific pre-calculated handbook functions are introduced locally in the neighborhood of the loading of the adjoint problem, \ie in the vicinity of point $P$. Let us recall that handbook techniques enable to achieve accurate local error bounds by using exactly the same spatial mesh for both adjoint and reference problems. The graph represented in \Fig{fig1:normalized_bounds_structure_fissuree_2D_fine_node_661_dep_x_handbook} shows the evolution of the normalized bounds $(\xibarsup,\xibarinf)$ of $\Iex$, obtained using the classical bounding technique coupled with handbook techniques, with respect to the number of enriched nodes. The corresponding number of layers of enriched nodes ranges from $1$ to $7$. By using the same spatial discretization for both residual adjoint and reference problems, meaningful and effective bounds can be obtained with an enrichment via the PUM involving only few nodes. In the following, we restrict to an enrichment applied only to one or two layers of nodes without performing any refinement of the residual adjoint problem.

\begin{figure}
\centering
\begin{tikzpicture}[scale=1.2]%[baseline]
\pgfplotsset{
xlabel near ticks,
ylabel near ticks,
label style={font=\small},
tick label style={font=\footnotesize},
legend style={font=\small},
try min ticks=7
}
\begin{axis}[
	width=0.4\textwidth,
	scale only axis,
	%ymin=-9,ymax=10,
	%axis y line*=left,
	%axis on top,
	%xlabel=Nombre de n\oe uds enrichis,
	xlabel=Number of enriched nodes,
	%ylabel=Bornes d'erreur locale,
	ylabel=Local error bounds,
	grid=major,
	legend pos=outer north east,
	legend entries={
	$1$,
	$\xibarinf$, $\xibarsup$, %$(\Ih + \Ihh - \frac{e_{\cre} \: \tilde{e}_{\cre}}{2})/\Iex$, $(\Ih + \Ihh + \frac{e_{\cre} \: \tilde{e}_{\cre}}{2})/\Iex$,
	%$(\Ih + \Ihh - \sum_E \frac{e_{\cre,E} \: \tilde{e}_{\cre,E}}{2})/\Iex$, $(\Ih + \Ihh + \sum_E \frac{e_{\cre,E} \: \tilde{e}_{\cre,E}}{2})/\Iex$
	}
]
\addplot+[sharp plot,red,solid,mark=none,mark options={red,scale=0.7}] table[x=nb_node,y expr=\thisrow{I_ex}/\thisrow{I_ex}] {structure_fissuree_2D_fine_node_661_dep_x_handbook.txt};
\addplot+[sharp plot,black,solid,mark=diamond,mark options={black,scale=1.0}] table[x=nb_node,y expr=(\thisrow{I_h}+\thisrow{I_hh_w}+(\thisrow{theta_EET_EF}*\thisrow{theta_EET_adjoint})/2)/\thisrow{I_ex}] {structure_fissuree_2D_fine_node_661_dep_x_handbook.txt};
\addplot+[sharp plot,black,solid,mark=square,mark options={black,scale=0.7}] table[x=nb_node,y expr=(\thisrow{I_h}+\thisrow{I_hh_w}-(\thisrow{theta_EET_EF}*\thisrow{theta_EET_adjoint})/2)/\thisrow{I_ex}] {structure_fissuree_2D_fine_node_661_dep_x_handbook.txt};
%\addplot+[sharp plot,teal,solid,mark=diamond,mark options={teal,scale=1.0}] table[x=nb_node,y expr=(\thisrow{I_h}+\thisrow{I_hh_w}+\thisrow{sum})/\thisrow{I_ex}] {structure_fissuree_2D_fine_node_661_dep_x_handbook.txt};
%\addplot+[sharp plot,teal,solid,mark=square,mark options={teal,scale=0.7}] table[x=nb_node,y expr=(\thisrow{I_h}+\thisrow{I_hh_w}-\thisrow{sum})/\thisrow{I_ex}] {structure_fissuree_2D_fine_node_661_dep_x_handbook.txt};
\end{axis}
\end{tikzpicture}
\caption{Evolutions of the lower and upper normalized bounds $(\xibarinf,\xibarsup)$ of $\Iex$ for local quantity $I_2$, obtained using the classical bounding technique combined with handbook techniques, with respect to the number of enriched nodes.}\label{fig1:normalized_bounds_structure_fissuree_2D_fine_node_661_dep_x_handbook}
\end{figure}
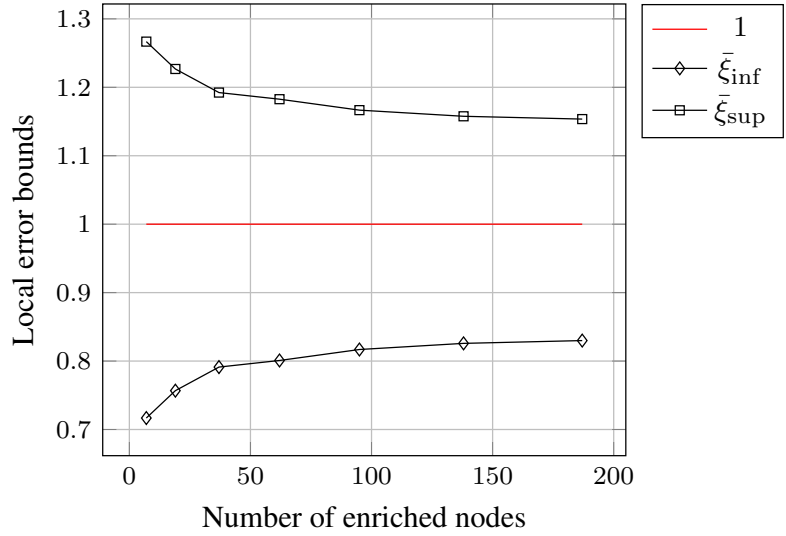

Maps of the local contributions to the global error estimates for (residual) adjoint problem with both enrichments are displayed in \Fig{fig1:distribution_estimators_structure_fissuree_2D_dep_x_node_661_enrichment_nb_layers_1_2}. The main contributions to the error estimate associated to residual adjoint problem are concentrated in zone $\Om_2^{PUM}$, which constitutes the support of the loading of residual adjoint problem. Similarly to the previous case presented in \Sect{6.2}, the error estimates for both reference and residual adjoint problems are located in distant regions.

\begin{figure}
\centering\includegraphics[scale = 0.4]{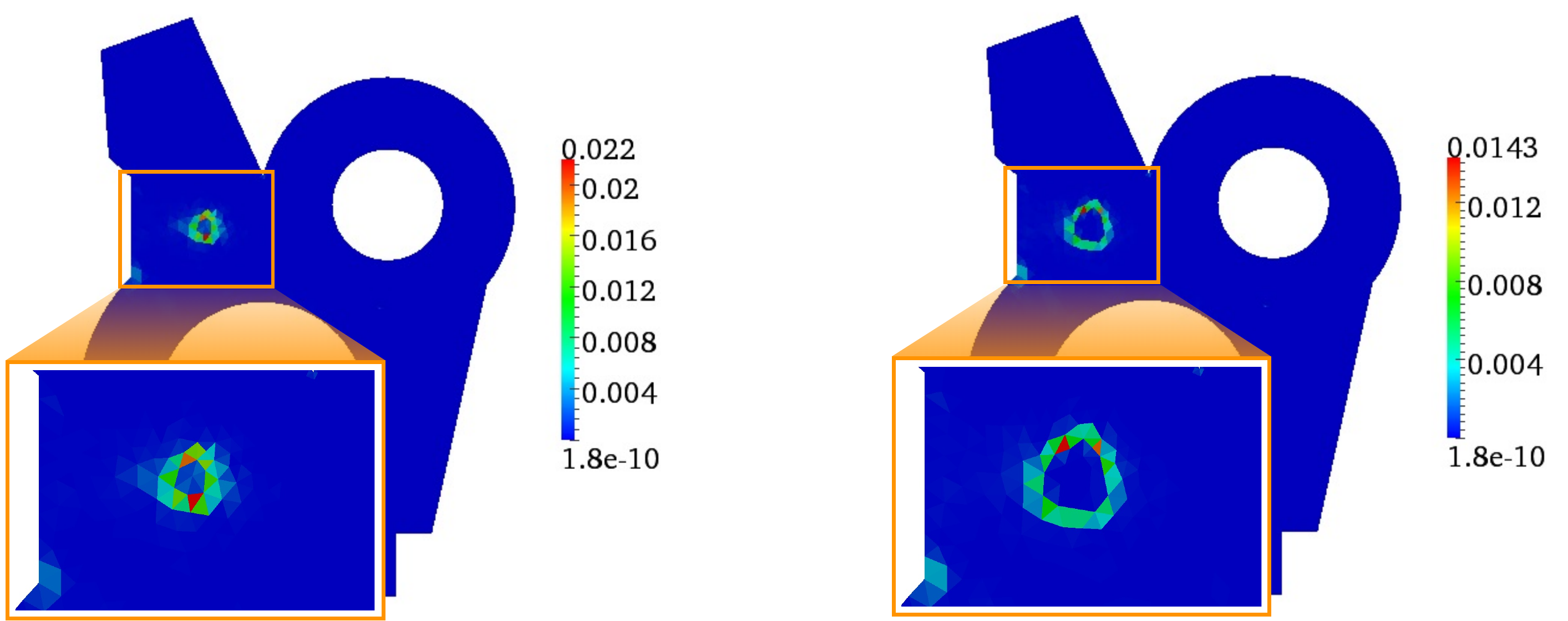}
\caption{Spatial distribution of local contributions to the error estimates associated to (residual) adjoint problem related to local quantity $I_2$, with one (left) and two (right) layers of nodes involved in the enrichment. Zoom boxes represent the estimated errors around enriched zone $\Om^{PUM}$.}\label{fig1:distribution_estimators_structure_fissuree_2D_dep_x_node_661_enrichment_nb_layers_1_2}
\end{figure}

For the enrichment involving a single layer of nodes, the values of parameters $\la$ and $\labar$ involved in the first improved technique are set to $1.7 r$ and $7 r$, respectively, where $r$ corresponds to the distance between point $P$ (defining the pointwise zone of interest) and its farthest neighbor node. The value of parameter $\labar_{\opt}$ involved in the second improved technique is set to $4.4 r$, which enables to achieve the sharpest bounds for quantity of interest $I_2$. The corresponding subdomains $\omla$, $\omlabar$ and $\omlabaropt$ are illustrated in \Fig{fig1:structure_fissuree_fine_Triangle_EF_pointwise_dep_x_node_661_enrichment_nb_layers_1_homothetic_domains}.

\begin{figure}
\centering\includegraphics[scale = 0.35]{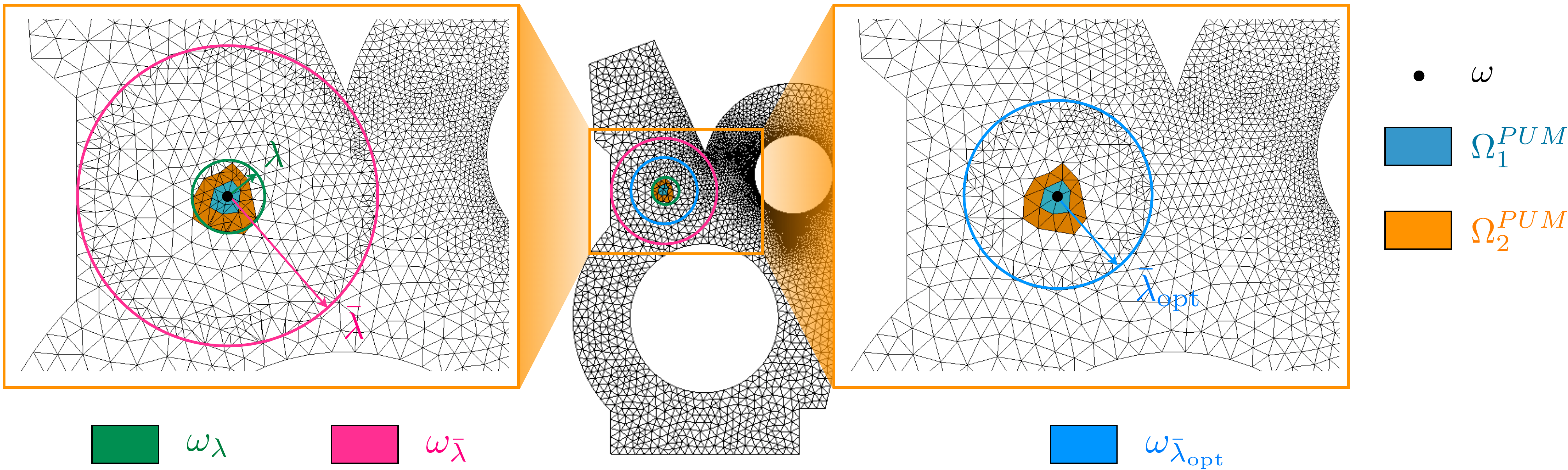}
\caption{Definition of subdomains $\omla$, $\omlabar$ involved in the first improved technique and $\omlabaropt$ involved in the second one for quantity of interest $I_2$ with a single layer of nodes involved in the enrichment.}\label{fig1:structure_fissuree_fine_Triangle_EF_pointwise_dep_x_node_661_enrichment_nb_layers_1_homothetic_domains}
\end{figure}

For the enrichment involving two layers of nodes, the values of parameters $\la$ and $\labar$ involved in the first improved technique are set to $2.5 r$ and $7 r$, while the value of parameter $\labar_{\opt}$ involved in the second improved technique is set to $4.4 r$, which enables to achieve the sharpest bounds for quantity of interest $I_2$. The corresponding subdomains $\omla$, $\omlabar$ and $\omlabaropt$ are illustrated in \Fig{fig1:structure_fissuree_fine_Triangle_EF_pointwise_dep_x_node_661_enrichment_nb_layers_2_homothetic_domains}.

\begin{figure}
\centering\includegraphics[scale = 0.35]{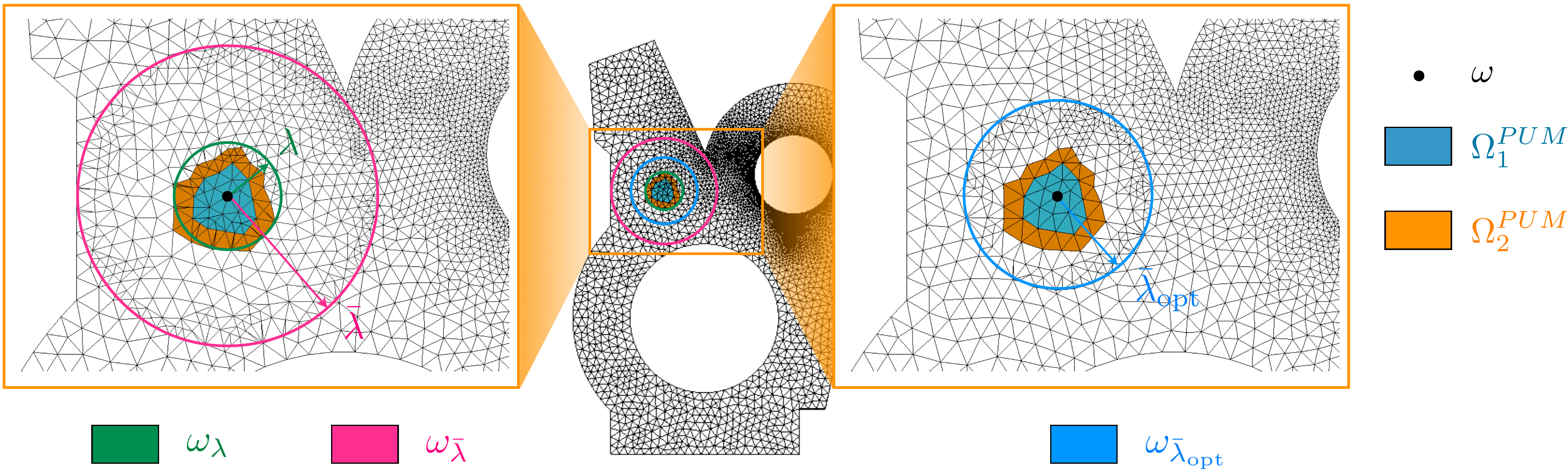}
\caption{Definition of subdomains $\omla$, $\omlabar$ involved in the first improved technique and $\omlabaropt$ involved in the second one for quantity of interest $I_2$ with two layers of nodes involved in the enrichment.}\label{fig1:structure_fissuree_fine_Triangle_EF_pointwise_dep_x_node_661_enrichment_nb_layers_2_homothetic_domains}
\end{figure}

The normalized bounds obtained using classical and improved bounding techniques combined with handbook techniques for an enrichment by the PUM involving either one or two layers of nodes are summarized in \Tab{table1:normalized_bounds_I_2}.
\begin{table}
\caption{Normalized bounds obtained using classical and optimized bounding techniques for quantity of interest $I_2$}
\centering
\tabsize
\begin{tabular}{l c c c c c c}
\toprule
Number of layers & \multicolumn{2}{c}{Classical bounds} & \multicolumn{2}{c}{Improved bounds 1} & \multicolumn{2}{c}{Improved bounds 2} \\
of enriched nodes & $\xibarinf$ & $\xibarsup$ & $\chibarinf$ & $\chibarsup$ & $\zetabarinf$ & $\zetabarsup$ \\
\midrule
$1$ & $0.7168$ & $1.2668$ & $0.7291$ & $1.2545$ & $0.8512$ & $1.1323$ \\
\midrule
$2$ & $0.7567$ & $1.2268$ & $0.7387$ & $1.2451$ & $0.8543$ & $1.1292$ \\
\bottomrule
\end{tabular}
\label{table1:normalized_bounds_I_2}
\end{table}

The analysis of the results reveals that bounds $(\zetabarinf,\zetabarsup)$ obtained using the second improved technique are more accurate than the ones $(\xibarinf,\xibarsup)$ obtained using the classical technique as well as the ones $(\chibarinf,\chibarsup)$ obtained using the first improved technique for both considered enrichments. It is worth noticing that bounds $(\chibarinf,\chibarsup)$ obtained using the first improved technique are coarser than the ones $(\xibarinf,\xibarsup)$ obtained using the classical technique for the second enrichment. Indeed, performances of the first improved technique strongly depend on ratio $\displaystyle\frac{\la}{\labar}$; the larger the enriched zone is, the higher ratio $\displaystyle\frac{\la}{\labar}$ is. 

Similar results can be obtained for the value at point $P$ of the other component $(.)_y$ of displacement field $\uu$.

%: 6.4
\subsection{Extracted value of a field: case of quantity of interest $I_3$}\label{6.4}
Map of the local contributions to the global error estimate for adjoint problem is depicted in \Fig{fig1:distribution_estimators_structure_fissuree_2D_SIF_I_pos_crack_tip_109_105_R1_6_R2_8}. The adjoint mesh density has been slightly increased toward the crown $\Om_c$; the adjoint mesh is made of $10 \, 699$ linear triangular elements and $5 \, 597$ nodes (\ie $11 \, 194$ d.o.f.). The highest contributions to the error estimate associated to adjoint problem are localized along the boundary of crown $\Om_c$. Therefore, the error estimates for both reference and adjoint problems are localized in close regions, contrary to previous cases presented in \Sects{6.2} and \ref{6.3}.

\begin{figure}
\centering\includegraphics[scale = 0.4]{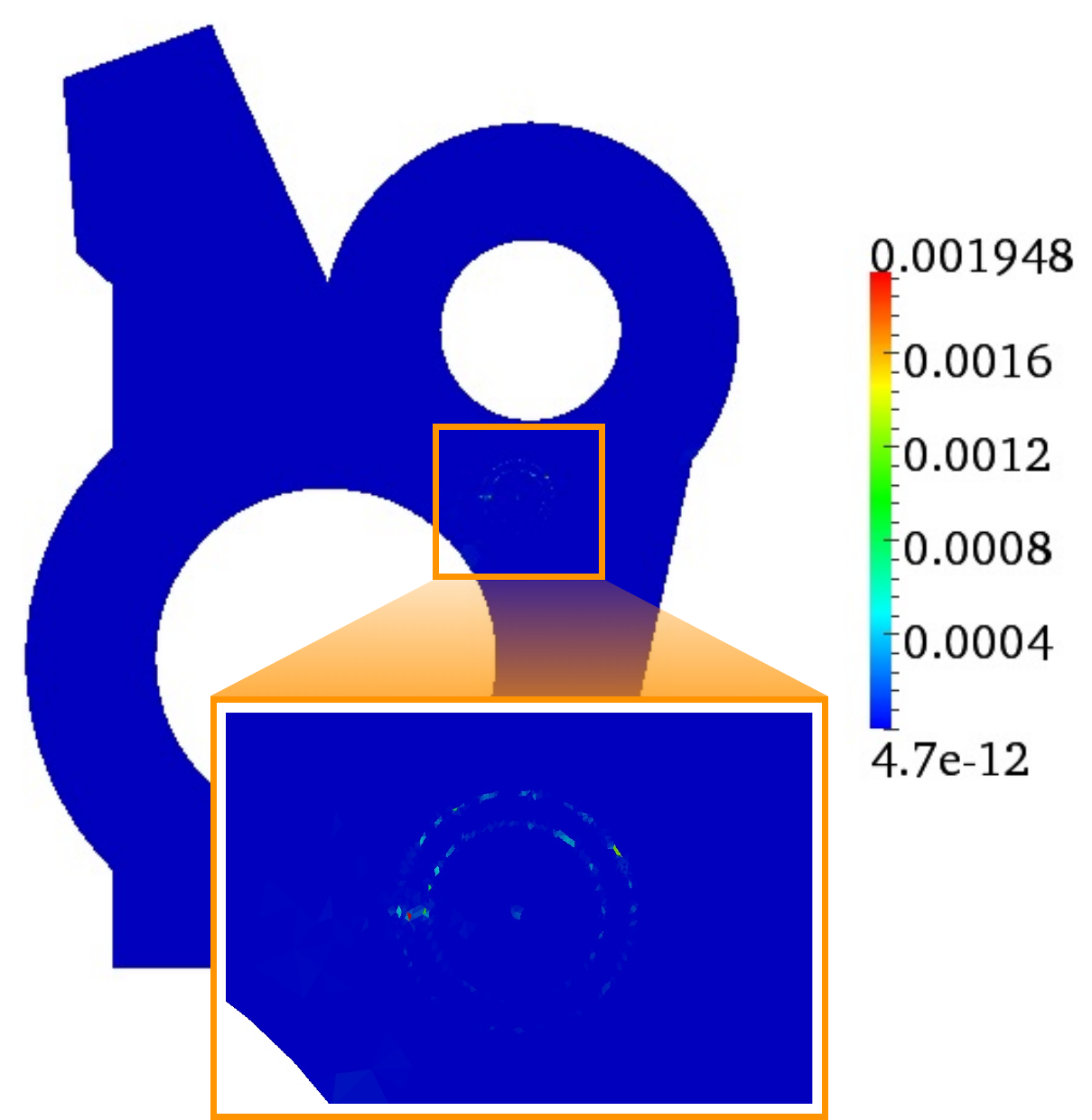}
\caption{Spatial distribution of local contributions to the error estimate associated to adjoint problem related to local quantity $I_3$. Zoom box represents the estimated error around crown $\Om_c$.}\label{fig1:distribution_estimators_structure_fissuree_2D_SIF_I_pos_crack_tip_109_105_R1_6_R2_8}
\end{figure}

The values of parameters $\la$ and $\labar$ involved in the first improved technique are set to $1.2 R_e$ and $2.1 R_e$, respectively, where $R_e$ corresponds to the radius of external circle $\Gamma_e$. The value of parameter $\labar_{\opt}$ involved in the second improved technique is set to $1.6 R_e$, which enables to achieve the sharpest bounds for quantity of interest $I_3$. The corresponding subdomains $\omla$, $\omlabar$ and $\omlabaropt$ are illustrated in \Fig{fig1:structure_fissuree_fine_Triangle_EF_SIF_I_pos_crack_tip_109_105_R1_6_R2_8_homothetic_domains}.

\begin{figure}
\centering\includegraphics[scale = 0.37]{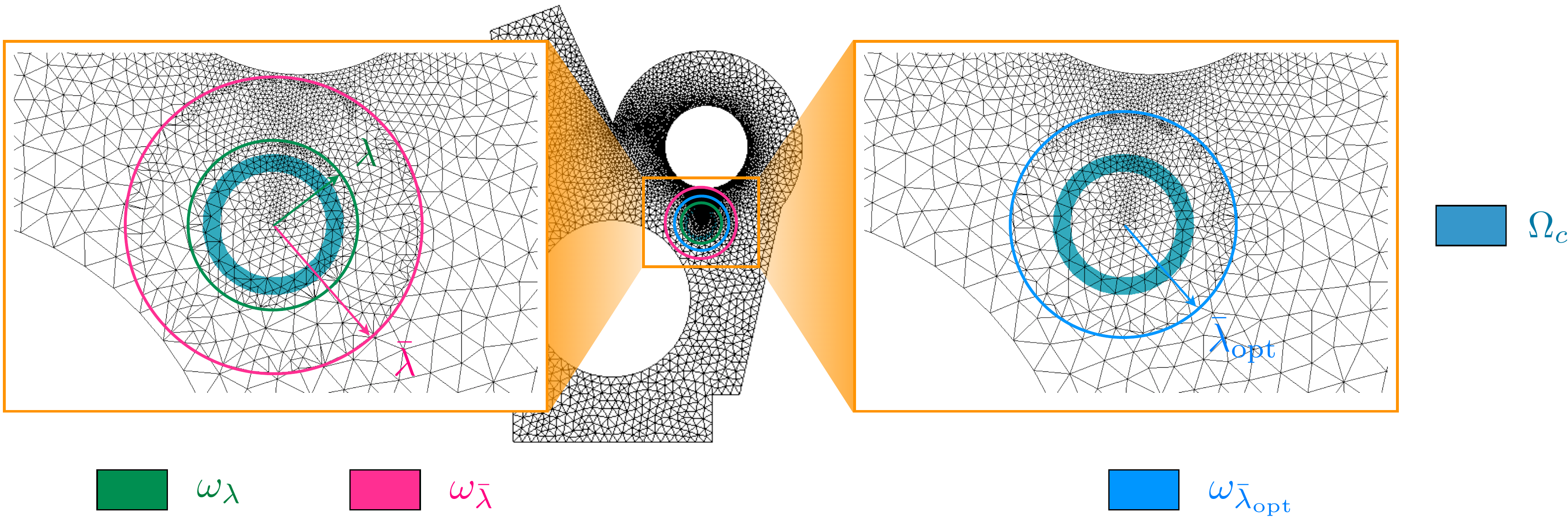}
\caption{Definition of subdomains $\omla$, $\omlabar$ involved in the first improved technique and $\omlabaropt$ involved in the second one for quantity of interest $I_3$.}\label{fig1:structure_fissuree_fine_Triangle_EF_SIF_I_pos_crack_tip_109_105_R1_6_R2_8_homothetic_domains}
\end{figure}

\Fig{fig1:normalized_bounds_structure_fissuree_2D_fine_SIF_I_pos_crack_tip_109_105_R1_6_R2_8} shows the evolutions of the normalized bounds on $\Iex$ for quantity of interest $I_3$ as functions of the number of elements $\tilde{N}_e$ of adjoint problem for the classical bounding technique as well as the two improved ones. A local refinement of adjoint mesh has been performed adaptively near the zone of interest $\om$, especially along internal and external boundaries of crown $\Om_c$, since the contributions to the global error estimate of the adjoint problem are mainly located in this region. One can see a moderate decrease in the bounds obtained with the first improved technique compared to the classical one., while the second improved technique and the classical one give similar results as regards the accuracy of the bounds. Consequently, the improvements presented in this work may lead to useless bounds in cases where the major part of estimated error related to both reference and adjoint problems are localized in close regions. Similar results can be obtained for the shear-mode stress intensity factor $\KII$.

\begin{figure}
\centering
\begin{tikzpicture}[scale=1.2]%[baseline]
\pgfplotsset{
xlabel near ticks,
ylabel near ticks,
label style={font=\small},
tick label style={font=\footnotesize},
legend style={font=\small},
try min ticks=7
}
\begin{axis}[
	width=0.4\textwidth,
	scale only axis,
	ymin=0,ymax=2,
	%axis y line*=left,
	%axis on top,
	%xlabel=Nombre d'\'el\'ements $\tilde{N}_e$,
	xlabel=Number of elements $\tilde{N}_e$,
	%ylabel=Bornes d'erreur locale,
	ylabel=Local error bounds,
	grid=major,
	legend pos=outer north east,
	legend entries={
	$1$,
	$\xibarinf$, $\xibarsup$,
	$\chibarinf$, $\chibarsup$,
	$\zetabarinf$, $\zetabarsup$,
	%$(\Ih + \Ihh - \sum_E \frac{e_{\cre,E} \: \tilde{e}_{\cre,E}}{2})/\Iex$, $(\Ih + \Ihh + \sum_E \frac{e_{\cre,E} \: \tilde{e}_{\cre,E}}{2})/\Iex$
	}
]
\addplot+[sharp plot,red,solid,mark=none,mark options={red,scale=0.7}] table[x=nb_elem,y expr=\thisrow{I_ex}/\thisrow{I_ex}] {structure_fissuree_2D_fine_SIF_I_pos_crack_tip_109_105_R1_6_R2_8.txt};
\addplot+[sharp plot,black,solid,mark=diamond,mark options={black,scale=1.0}] table[x=nb_elem,y expr=(\thisrow{I_h}+\thisrow{I_hh_w}-(\thisrow{theta_EET_EF}*\thisrow{theta_EET_adjoint})/2)/\thisrow{I_ex}] {structure_fissuree_2D_fine_SIF_I_pos_crack_tip_109_105_R1_6_R2_8.txt};
\addplot+[sharp plot,black,solid,mark=square,mark options={black,scale=0.7}] table[x=nb_elem,y expr=(\thisrow{I_h}+\thisrow{I_hh_w}+(\thisrow{theta_EET_EF}*\thisrow{theta_EET_adjoint})/2)/\thisrow{I_ex}] {structure_fissuree_2D_fine_SIF_I_pos_crack_tip_109_105_R1_6_R2_8.txt};
\addplot+[sharp plot,blue,solid,mark=diamond,mark options={blue,scale=1.0}] table[x=nb_elem,y expr=(\thisrow{I_h}+\thisrow{I_hh_w}+\thisrow{I_hhh}-\thisrow{xi_la_min_1.2_max_2.1})/\thisrow{I_ex}] {structure_fissuree_2D_fine_SIF_I_pos_crack_tip_109_105_R1_6_R2_8.txt};
\addplot+[sharp plot,blue,solid,mark=square,mark options={blue,scale=0.7}] table[x=nb_elem,y expr=(\thisrow{I_h}+\thisrow{I_hh_w}+\thisrow{I_hhh}+\thisrow{xi_la_min_1.2_max_2.1})/\thisrow{I_ex}] {structure_fissuree_2D_fine_SIF_I_pos_crack_tip_109_105_R1_6_R2_8.txt};
\addplot+[sharp plot,magenta,solid,mark=diamond,mark options={magenta,scale=1.0}] table[x=nb_elem,y expr=(\thisrow{I_h}+\thisrow{I_hh_w}-\thisrow{zeta_la_1.6})/\thisrow{I_ex}] {structure_fissuree_2D_fine_SIF_I_pos_crack_tip_109_105_R1_6_R2_8.txt};
\addplot+[sharp plot,magenta,solid,mark=square,mark options={magenta,scale=0.7}] table[x=nb_elem,y expr=(\thisrow{I_h}+\thisrow{I_hh_w}+\thisrow{zeta_la_1.6})/\thisrow{I_ex}] {structure_fissuree_2D_fine_SIF_I_pos_crack_tip_109_105_R1_6_R2_8.txt};
%\addplot+[sharp plot,teal,solid,mark=diamond,mark options={teal,scale=1.0}] table[x=nb_elem,y expr=(\thisrow{I_h}+\thisrow{I_hh_w}-\thisrow{sum})/\thisrow{I_ex}] {structure_fissuree_2D_fine_SIF_I_pos_crack_tip_109_105_R1_6_R2_8.txt};
%\addplot+[sharp plot,teal,solid,mark=square,mark options={teal,scale=0.7}] table[x=nb_elem,y expr=(\thisrow{I_h}+\thisrow{I_hh_w}+\thisrow{sum})/\thisrow{I_ex}] {structure_fissuree_2D_fine_SIF_I_pos_crack_tip_109_105_R1_6_R2_8.txt};
\end{axis}
\end{tikzpicture}
\caption{Evolutions of the lower and upper normalized bounds of $\Iex$ for local quantity $I_3$, obtained using the classical bounding technique as well as first and second improvements, with respect to the number of elements $\tilde{N}_e$ associated to the discretization of the adjoint problem.}\label{fig1:normalized_bounds_structure_fissuree_2D_fine_SIF_I_pos_crack_tip_109_105_R1_6_R2_8}
\end{figure}
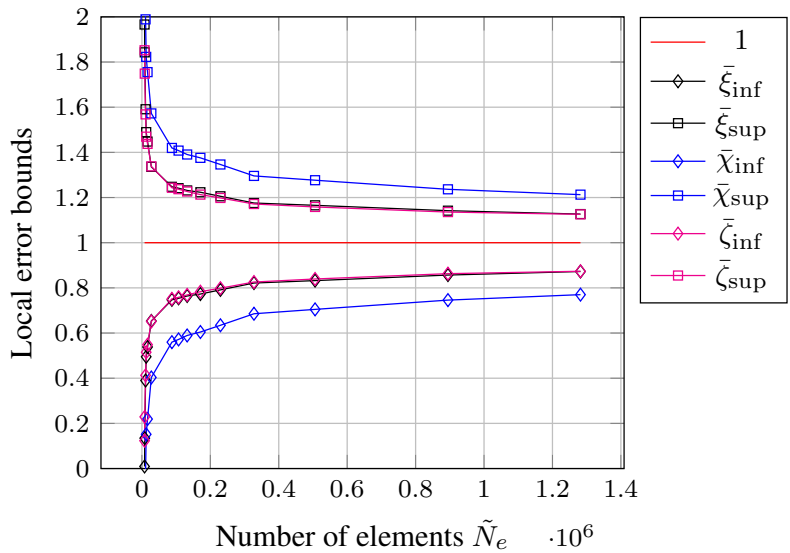

\Fig{fig1:normalized_interest_quantities_structure_fissuree_2D_fine_SIF_I_pos_crack_tip_109_105_R1_6_R2_8} represents the evolutions of the normalized exact value $\Iex/\Iex$ of local quantity $I_3$, its normalized approximate value $\Ih/\Iex$ obtained through the FEM and its new normalized approximate value $(\Ih + \Ihh)/\Iex$ as functions of the number of elements $\tilde{N}_e$ of adjoint problem. One can observe once again that $\Ih + \Ihh$ corresponds to an approximation of better quality with respect to $\Iex$ compared to $\Ih$.

\begin{figure}
\centering
\begin{tikzpicture}[scale=1.2]%[baseline]
\pgfplotsset{
xlabel near ticks,
ylabel near ticks,
label style={font=\small},
tick label style={font=\footnotesize},
legend style={font=\small},
%try min ticks=7
}
\begin{axis}[
	width=0.4\textwidth,
	scale only axis,
	yticklabel style={/pgf/number format/precision=3},
	%axis y line*=left,
	%axis on top,
	%xlabel=Nombre d'\'el\'ements $\tilde{N}_e$,
	xlabel=Number of elements $\tilde{N}_e$,
	%ylabel=Bornes d'erreur locale,
	ylabel=Local error bounds,
	grid=major,
	legend pos=outer north east,
	legend entries={
	$1$,
	$\Ih/\Iex$,
	$(\Ih + \Ihh)/\Iex$,
	}
]
\addplot+[sharp plot,red,solid,mark=none,mark options={red,scale=0.7}] table[x=nb_elem,y expr=\thisrow{I_ex}/\thisrow{I_ex}] {structure_fissuree_2D_fine_SIF_I_pos_crack_tip_109_105_R1_6_R2_8.txt};
\addplot+[sharp plot,green!60!black,solid,mark=none,mark options={green!60!black,scale=0.7}] table[x=nb_elem,y expr=\thisrow{I_h}/\thisrow{I_ex}] {structure_fissuree_2D_fine_SIF_I_pos_crack_tip_109_105_R1_6_R2_8.txt};
\addplot+[sharp plot,orange,solid,mark=o,mark options={orange,scale=0.7}] table[x=nb_elem,y expr=(\thisrow{I_h}+\thisrow{I_hh_w})/\thisrow{I_ex}] {structure_fissuree_2D_fine_SIF_I_pos_crack_tip_109_105_R1_6_R2_8.txt};
\end{axis}
\end{tikzpicture}
\caption{Evolutions of the normalized exact value of local quantity $I_3$, its normalized approximate value $\Ih/\Iex$ and its new normalized approximate value $(\Ih + \Ihh)/\Iex$ with respect to the number of elements $\tilde{N}_e$ associated to the discretization of the adjoint problem.}\label{fig1:normalized_interest_quantities_structure_fissuree_2D_fine_SIF_I_pos_crack_tip_109_105_R1_6_R2_8}
\end{figure}
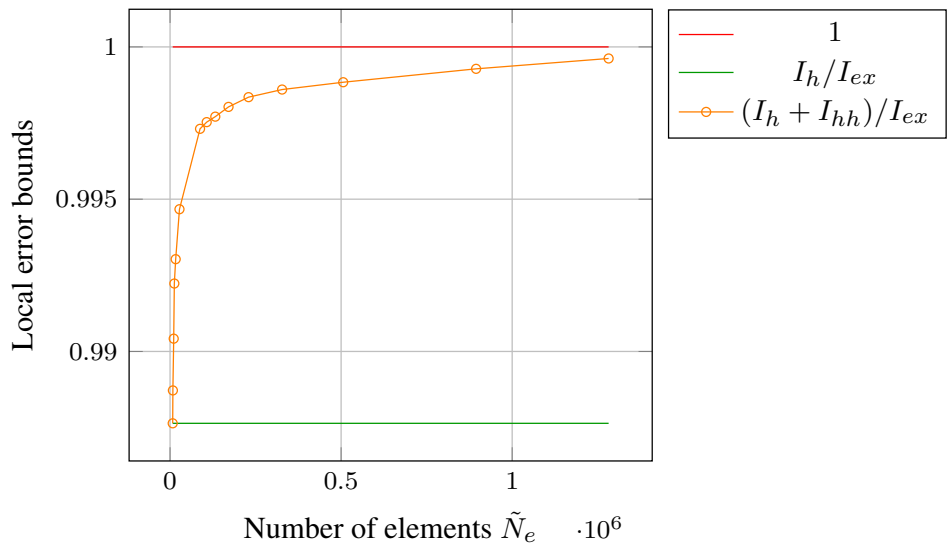

%: 7. Conclusion and prospects
\section{Conclusion and prospects}\label{7}

In this paper, we introduced two new approaches related to the general framework of robust goal-oriented error estimation dealing with extraction techniques. These techniques are based on mathematical tools which are not classical in model verification. Several numerical experiments clearly demonstrate the efficiency of these methods to produce strict and relevant bounds on the errors in linear local quantities of interest compared to the classical bounding technique, especially when the discretization error related to the reference error is not concentrated in the local zone of interest. Nevertheless, the second proposed technique seems to achieve sharper local error estimates than the first one.

Finally, such powerful methods may open up opportunities and help widen the field of robust goal-oriented error estimation methods. Both techniques could be easily extended to other quantities of interest but are restricted to linear problems, \ie cases where Saint-Venant's principle is well established. The extension to broader classes of mechanical problems such as time-dependent non linear problems is thus an open question. Besides, coupling these new improved methods with handbook techniques offers several potentially fruitful directions for future research. This crucial issue is still under active consideration.

%: Appendix A.
\section*{Appendix A. Proof of the main technical result of \Prop{prop1}}\label{appA}

First, let us recall that the CRE in subdomain $\omla$, included in $\Om$, defined by (\ref{eq1:ecrela}) can be recast in the following form:
\begin{align}
e^2_{\cre,\la} & = \lnorm{\cont - \hat{\cont}_h}^2_{\cont, \omla} + \lnorm{\uu - \hat{\uu}_h}^2_{u, \omla} - 2 \intola \Tr\big[(\cont - \hat{\cont}_h) \: \defo(\uu - \hat{\uu}_h) \big] \dO \label{eq1:RDCla} \\
 & = \lnorm{\cont - \hat{\cont}_h}^2_{\cont, \omla} + \lnorm{\uu - \hat{\uu}_h}^2_{u, \omla} - 2 \intdola (\cont - \hat{\cont}_h) \: \und{n} \cdot (\uu - \hat{\uu}_h) \dS. \label{eq2:RDCla}
\end{align}

Let us now consider the following decomposition of the discretization error $\und{e}_h = \uu - \hat{\uu}_h$ on subdomain $\omla$:
\begin{align}\label{eq1:decomposition}
\uu - \hat{\uu}_h = \uu_1 + \uu_2,
\end{align}
where the local error $\uu_1$ and the pollution error $\uu_2$ are solutions of local problems $(\Pone)$ and $(\Ptwo)$, respectively, defined on $\omla$ and introduced in \Sect{4.2}.

Such a decomposition verifies
\begin{align}
\lnorm{\uu - \hat{\uu}_h}^2_{u, \omla} = \lnorm{\uu_1}^2_{u, \omla} + \lnorm{\uu_2}^2_{u, \omla}
\end{align}

and
\begin{align}
\intdola (\cont - \hat{\cont}_h) \: \und{n} \cdot (\uu - \hat{\uu}_h) \dS & = \intdola (\cont - \hat{\cont}_h) \: \und{n} \cdot \uu_2 \dS \\
 & = \intdola \Tr\big[(\cont - \hat{\cont}_h) \: (\uu_2 \otimes \und{n}) \big] \dS \\
 & = \lscalproda{\cont - \hat{\cont}_h}{\K \: {(\uu_2 \otimes \und{n})}_{\sym}}_{\cont, \domla} \\
 & \leqslant \lnorm{\cont - \hat{\cont}_h}_{\cont, \domla} \lnorm{\K \: {(\uu_2 \otimes \und{n})}_{\sym}}_{\cont, \domla},
\end{align}
where ${(\bullet)}_{\sym}$ represents the symmetric part of matrix $\bullet$.

Then, relation (\ref{eq2:RDCla}) becomes:
\begin{align}
e^2_{\cre,\la} \geqslant \lnorm{\cont - \hat{\cont}_h}^2_{\cont, \omla} + \lnorm{\uu_1}^2_{u, \omla} + \lnorm{\uu_2}^2_{u, \omla} - 2 \lnorm{\cont - \hat{\cont}_h}_{\cont, \domla} \lnorm{\K \: {(\uu_2 \otimes \und{n})}_{\sym}}_{\cont, \domla}. \label{eq3:RDCla}
\end{align}

By observing that $\uu_2 \in \Vcb$, where space $\Vcb$ has been defined in (\ref{eq1:V}), let us now introduce the Steklov constant $h \la$ defined in (\ref{eq1:constanthla}), which leads to:
\begin{align}
\lnorm{\K \: {(\uu_2 \otimes \und{n})}_{\sym}}_{\cont, \domla} \leqslant \sqrt{h \la} \lnorm{\uu_2}_{u, \omla}. \label{eq2:constanthla}
\end{align}

As a result, the following inequality can be deduced from (\ref{eq3:RDCla}) and (\ref{eq2:constanthla}):
\begin{align}
e^2_{\cre,\la} \geqslant & \lnorm{\cont - \hat{\cont}_h}^2_{\cont, \omla} + \lnorm{\uu_1}^2_{u, \omla} + \displaystyle\left[ \lnorm{\uu_2}_{u, \omla} - \sqrt{h \la} \lnorm{\cont - \hat{\cont}_h}_{\cont, \domla} \right]^2 - h \la \lnorm{\cont - \hat{\cont}_h}^2_{\cont, \domla} \label{eq4:RDCla}
\end{align}

Besides, using equality (\ref{eq1:diffintola}) for quantity $\Tr\big[(\cont - \hat{\cont}_h) \: \K^{-1} \: (\cont - \hat{\cont}_h)\big]$, one obtains:
\begin{align}
\frac{\diff}{\dla} \left[ \lnorm{\cont - \hat{\cont}_h}^2_{\cont, \omla} \right] & = \lnorm{\cont - \hat{\cont}_h}^2_{\cont, \domla}. \label{eq2:diffintola}
\end{align}
Incorporating (\ref{eq2:diffintola}) into (\ref{eq4:RDCla}) and rearranging the terms, inequality (\ref{eq4:RDCla}) can be rewritten as follows:
\begin{align}
\al_{\la} \geqslant \lnorm{\cont - \hat{\cont}_h}^2_{\cont, \omla} - h \la \frac{\diff}{\dla} \left[ \lnorm{\cont - \hat{\cont}_h}^2_{\cont, \omla} \right] \label{eq5:RDCla}
\end{align}
where $\displaystyle \al_{\la} \equiv \al_{\la}(\hat{\uu}_h, \hat{\cont}_h, \uu_1, \uu_2) = e^2_{\cre,\la} - \lnorm{\uu_1}^2_{u, \omla} - \left[ \lnorm{\uu_2}_{u, \omla} - \sqrt{h \la} \lnorm{\cont - \hat{\cont}_h}_{\cont, \domla} \right]^2$.
Let us note that first order ordinary differential inequation (\ref{eq5:RDCla}) verified by $\lnorm{\cont - \hat{\cont}_h}^2_{\cont, \omla}$ can be reformulated as:
\begin{align}
\frac{\diff}{\dla} \left[ \lnorm{\cont - \hat{\cont}_h}^2_{\cont, \omla} f(\la)\right] \frac{1}{f(\la)} \geqslant - \frac{1}{h \la} \al_{\la}, \label{eq6:RDCla}
\end{align}
where $\displaystyle f \colon \la \mapsto \la^{-1/h}$ is a strictly positive function for any $\la > 0$.
Finally, a worthwhile relation between $\lnorm{\cont - \hat{\cont}_h}^2_{\cont, \omla}$ and $\lnorm{\cont - \hat{\cont}_h}^2_{\cont, \omlabar}$ can be derived integrating inequality (\ref{eq6:RDCla}) over $\intervalcc{\la}{\labar}$. It reads:
\begin{align}
\lnorm{\cont - \hat{\cont}_h}^2_{\cont, \omla} \leqslant \left( \frac{\la}{\labar} \right)^{1/h} \lnorm{\cont - \hat{\cont}_h}^2_{\cont, \omlabar} + \beta_{\la, \labar}, \label{eq1:fundrelation}
\end{align}
where $\displaystyle\beta_{\la, \labar} \equiv \beta_{\la, \labar}(\hat{\uu}_h, \hat{\cont}_h, \uu_1, \uu_2) = \int_{\la' = \la}^{\labar} \left[ \left( \frac{\la'}{\la} \right)^{-1/h} \frac{1}{h \la'} \al_{\la'} \right] \dla'$ is a function involving displacement solutions $\uu_1$ and $\uu_2$ of problems $(\Pone)$ and $(\Ptwo)$, respectively. In practice, inequality $\displaystyle \al_{\la'} \leqslant e^2_{\cre,\la'} \quad \forall \: \la' \in \intervalcc{\la}{\labar}$ allows to work around the extra resolutions of problems $(\Pone)$ and $(\Ptwo)$ and leads to the following fundamental result of \Prop{prop1}:
\begin{align}
\lnorm{\cont - \hat{\cont}_h}^2_{\cont, \omla} \leqslant \left( \frac{\la}{\labar} \right)^{1/h} \lnorm{\cont - \hat{\cont}_h}^2_{\cont, \omlabar} + \gamma_{\la, \labar}, \label{eq2:fundrelation}
\end{align}
where $\displaystyle\gamma_{\la, \labar}$ is completely defined by (\ref{eq1:gamma}) in terms of computable known quantities. For practical purposes, one-dimensional numerical integration methods, such as the simple trapezoidal rule or the classical Gauss-Legendre quadrature rule, can be employed to get an accurate approximation of $\gamma_{\la, \labar}$.

\begin{remark}
Other expressions for $\gamma_{\la, \labar}$ can be derived directly by integration by parts; it readily reads:
\begin{align}
\displaystyle\gamma_{\la, \labar} = e^2_{\cre,\la} - \left( \frac{\la}{\labar} \right)^{1/h} e^2_{\cre,\labar} + \int_{\la' = \la}^{\labar} \left[ \left( \frac{\la'}{\la} \right)^{-1/h} \frac{\diff}{\dla'} \left[ e^2_{\cre,\la'} \right] \right] \dla'. \label{eq2:gamma}
\end{align}
However, implementation of this latter expression is more cumbersome compared to the former one, since it requires the integration of the first derivative of function $e^2_{\cre,\la'}$ instead of function $e^2_{\cre,\la'}$ itself. Nevertheless, applying relation (\ref{eq2:diff}) to quantity $\hat{\cont}_h - \K \: \defo(\hat{\uu}_h)$ for any homothetic domain $\omlaprime$ such that $\la' \in \intervalcc{\la}{\labar}$, one obtains:
\begin{equation}
\displaystyle\frac{\diff}{\dla'} \left[ e^2_{\cre,\la'} \right] = \frac{\diff}{\dla'} \left[ \lnorm{\hat{\cont}_h - \K \: \defo(\hat{\uu}_h)}^2_{\cont, \omlaprime} \right] = {\left(\frac{\la'}{\displaystyle\labar}\right)}^n \frac{1}{\displaystyle\labar} \labs{\hat{\cont}_h - \K \: \defo(\hat{\uu}_h)}^2_{\cont, \domlabar};
\end{equation}
consequently, using relation (\ref{eq1:linkintolaintdolabar}) for both domains $\omla$ and $\omlabar$, one has:
\begin{align}
\displaystyle\gamma_{\la, \labar} & = e^2_{\cre,\la} - \left( \frac{\la}{\labar} \right)^{1/h} e^2_{\cre,\labar} + \int_{\la' = \la}^{\labar} \left[ \left( \frac{\la'}{\la} \right)^{-1/h} {\left(\frac{\la'}{\displaystyle\labar}\right)}^n \frac{1}{\displaystyle\labar} \labs{\hat{\cont}_h - \K \: \defo(\hat{\uu}_h)}^2_{\cont, \domlabar} \right] \dla' \nonumber\\
& = e^2_{\cre,\la} - \left( \frac{\la}{\labar} \right)^{1/h} e^2_{\cre,\labar} + \int_{\la' = \la}^{\labar} \labs{\left( \frac{\la'}{\la} \right)^{-1/{2h}} \left(\hat{\cont}_h - \K \: \defo(\hat{\uu}_h) \right)}^2_{\cont, \domlabar} {\left(\frac{\la'}{\displaystyle\labar}\right)}^n \frac{1}{\displaystyle\labar} \dla' \nonumber\\
& = e^2_{\cre,\la} - \left( \frac{\la}{\labar} \right)^{1/h} e^2_{\cre,\labar} + e^2_{\wcre,\labar \setminus \la}, \label{eq3:gamma}
\end{align}
where the last term of the right-hand side:
\begin{align}
e^2_{\wcre,\labar \setminus \la} \equiv e^2_{\wcre,\labar \setminus \la}(\hat{\uu}_h, \hat{\cont}_h) & = \lnorm{\left( \frac{\la'}{\la} \right)^{-1/{2h}} \left(\hat{\cont}_h - \K \: \defo(\hat{\uu}_h) \right)}^2_{\cont, \omlabarminusla} \\
& = \lscalproda{\left( \frac{\la'}{\la} \right)^{-1/h} \left(\hat{\cont}_h - \K \: \defo(\hat{\uu}_h) \right)}{\hat{\cont}_h - \K \: \defo(\hat{\uu}_h)}_{\cont, \omlabarminusla}
\end{align}
can be viewed as a weighted constitutive relation error in $\omlabarminusla$. Even though this last relation (\ref{eq3:gamma}) does not call for one-dimensional numerical integration methods contrary to (\ref{eq1:gamma}) and (\ref{eq2:gamma}), it requires the numerical evaluation of a definite integral of a rational function over $\omlabarminusla$.

Eventually, in order to perform an accurate calculation of function $\gamma_{\la, \labar}$, expression (\ref{eq1:gamma}) computed by means of a basic trapezoidal integration method with a large number of integration points is preferred among the different aforementioned expressions.
\end{remark}

%: Appendix B.
\section*{Appendix B. Proof of the main technical result of \Prop{prop4}}\label{appB}

Let us consider $\vu \in \Vcb$. One can deduce the following inequality from (\ref{eq1:constantklabarla}):
\begin{align}
\forall \: \la \in \intervalcc{0}{\labar}, \quad \frac{k}{\la} \lnorm{\vu}^2_{u, \omla} \leqslant \frac{1}{\labar} \labs{\vu}^2_{u, \domla}. \label{ineq1:v}
\end{align}
Then, using (\ref{eq1:linkintdolaintdolabar}), one has: $\displaystyle\labs{\vu}^2_{u, \domla} = {\left(\frac{\la}{\displaystyle\labar}\right)}^n \labs{\vu}^2_{u, \domlabar}$ and (\ref{ineq1:v}) becomes:
\begin{align}
\frac{k}{\la} \lnorm{\vu}^2_{u, \omla} \leqslant {\left(\frac{\la}{\displaystyle\labar}\right)}^n \frac{1}{\displaystyle\labar} \labs{\vu}^2_{u, \domlabar}. \label{ineq2:v}
\end{align}
According to result (\ref{eq1:diff}), one has: $\displaystyle{\left(\frac{\la}{\displaystyle\labar}\right)}^n \frac{1}{\displaystyle\labar} \labs{\vu}^2_{u, \domlabar} = \frac{\diff}{\dla} \left[\lnorm{\vu}^2_{u, \omla}\right]$; consequently, the following first order ordinary differential inequation satisfied by $\lnorm{\vu}^2_{u, \omla}$ holds:
\begin{align}
\frac{k}{\la} \lnorm{\vu}^2_{u, \omla} \leqslant \frac{\diff}{\dla} \left[\lnorm{\vu}^2_{u, \omla}\right], \label{ineq3:v}
\end{align}
which can be recast as:
\begin{align}
\frac{\diff}{\dla} \left[ \lnorm{\vu}^2_{u, \omla} g(\la)\right] \frac{1}{g(\la)} \geqslant 0, \label{ineq4:v}
\end{align}
where $\displaystyle g \colon \la \mapsto \la^{-k}$ is a strictly positive function for any $\la > 0$; as a consequence, one gets:
\begin{align}
\frac{\diff}{\dla} \left[ \lnorm{\vu}^2_{u, \omla} g(\la)\right] \geqslant 0, \label{ineq5:v}
\end{align}

Finally, a fundamental result connecting $\lnorm{\vu}^2_{u, \omla}$ and $\lnorm{\vu}^2_{u, \omlabar}$ can be derived integrating inequality (\ref{ineq5:v}) over $\intervalcc{\la}{\labar}$. It reads:
\begin{align}
\lnorm{\vu}^2_{u, \omla} \leqslant \left( \frac{\la}{\labar} \right)^{k} \lnorm{\vu}^2_{u, \omlabar}, \label{eq1:fundrelation2}
\end{align}
which completes the proof of \Prop{prop4}.

%: Appendix C.
\section*{Appendix C. Analytical values of constant $h$ involved in the first improved technique}\label{appC}

The values of constant $h$, expressed in terms of Lam\'e's c\oe fficients $(\la, \mu)$ in column two and in terms of Poisson's ratio $\nu$ in the last columns (under both plane stress and plane strain assumptions for two-dimensional geometric shapes), are given in \Tab{table1:values_constant_h}.
\begin{table}
\caption{Analytical values of constant $h$ for different geometric shapes}
\centering
\tabsize
\begin{tabular}{l c c c}
\toprule
Geometric shape & \multicolumn{3}{c}{Constant \ $h$}\\
\midrule
Dim 2 & & plane stress assumption & plane strain assumption \\
\midrule
unit circle & $\displaystyle\frac{2 \mu + \la}{2 (\mu + \lambda)}$ & $\displaystyle\frac{1}{1 + \nu}$ & $1- \nu$ \\
unit cracked circle\footnotemark[1] & $\displaystyle\frac{2 \mu + \la}{2 (\mu + \lambda)} + \frac{1}{3 (2 \pi - \theta)} \frac{\mu}{\mu + \lambda}$ & $\displaystyle\frac{1}{1 + \nu} + \frac{1}{3 (2 \pi - \theta)} \frac{1 - \nu}{1 + \nu}$ & $1- \nu + \displaystyle\frac{1}{3 (2 \pi - \theta)} (1 - 2 \nu)$ \\
double unit square & $\displaystyle\frac{7 \mu + 3 \la}{6 (\mu + \lambda)}$ & $\displaystyle\frac{7 - \nu}{6 (1 + \nu)}$ & $\displaystyle\frac{7 - 8 \nu}{6}$ \\
\midrule
Dim 3 \\
\midrule
unit sphere & $\displaystyle\frac{2 \mu + \lambda}{2 \mu + 3 \lambda}$ & \multicolumn{2}{c}{$\displaystyle\frac{1 - \nu}{1 + \nu}$} \\
double unit parallelepiped & $\displaystyle\frac{8 \mu + 3 \la}{6 \mu + 9 \lambda}$ & \multicolumn{2}{c}{$\displaystyle\frac{4 - 5 \nu}{3 (1 + \nu)}$} \\
\bottomrule
\end{tabular}
\label{table1:values_constant_h}
\end{table}
\footnotetext[1]{$\theta$ denotes the angle between the two lips of the crack.}

%: Appendix D.
\section*{Appendix D. Analytical computation of constant $k$ involved in the second improved technique for circular and spherical shape domains}\label{appD}

Let us consider the space $\Vcb$ of functions satisfying homogeneous equilibrium equations expressed in displacements, also known as Lam\'e-Navier equations. Trefftz (or T-) functions are homogeneous solutions of the governing differential equations inside the domain (corresponding to Lam\'e-Navier equations in our case). Trefftz-type approaches, which consist in using a set of linearly independent solutions of a differential equation, were initially introduced by Trefftz \cite{Tre26}. T-functions were firstly employed as basis functions in Trefftz methods \cite{Kit95}. Nowadays, they are classically used as interpolation functions to define Trefftz-type finite elements in (hybrid) FEM, Boundary Element Method (BEM) also called Boundary Integral Equation Method (BIEM); they can be found in the form of polynomials, Legendre, harmonic, exponential, Bessel, Hankel, Kelvin (also called singular Kupradze), Boussinesq functions, depending on the governing equations \cite{Zie95,Kom04,Kom05}. Some special purpose T-functions enable to satisfy not only the governing equations, but also take into account special boundary conditions a priori. The interested reader can refer to \cite{Zie85,Zie95} for a complete description of basic sets of T-functions associated to Laplace, Helmholtz and biharmonic equations in 2D and 3D problems, both for bounded and unbounded domains. Combinations of those classical T-functions allow the derivation of T-functions associated to a broad class of problems, such as 2D and 3D elasticity or Mindlin-type plates. 

Let us consider a circular (resp. spherical) shape domain $\omlabar$ of radius $\labar$ in 2D (resp. 3D). The set of T-functions associated to Lam\'e-Navier equations constitutes a basis of functions belonging to space $\Vcb$ on $\omlabar$ \cite{Hoc92}. Thus, any homogeneous solution $\vu \in \Vcb$ can be defined by linear combination of a set of Trefftz functions. Suitable Trefftz functions corresponding to solutions of Lam\'e-Navier equations can be found in cartesian coordinates system in \cite{Sla00} and in polar coordinates system in \cite{Sla02}. T-functions associated to Lam\'e-Navier equations can be expressed in cartesian coordinate system $(x,y)$ in 2D (resp. $(x,y,z)$ in 3D) as polynomials \cite{Sla00}. Transforming cartesian coordinates into polar coordinates $(r,\theta)$ in 2D (resp. spherical coordinates $(r,\theta,\phi)$ in 3D), displacement T-functions $\und{T}$ can be formulated as regular harmonic polynomials of the form:
\begin{align}\label{eq1:Trefftz_functions}
\begin{cases}
T_x = r^i f_2(\theta), \quad T_y = r^j g_2(\theta), \quad i,j = 0,1,2,\ldots & \text{in 2D}; \\
T_x = r^i f_3(\theta,\phi), \quad T_y = r^j g_3(\theta,\phi), \quad T_z = r^l h_3(\theta,\phi), \quad i,j,l = 0,1,2,\ldots & \text{in 3D},
\end{cases}
\end{align}
where functions $f_2$ and $g_2$ (resp. $f_3$, $g_3$ and $h_3$) depend on material parameters $\nu$, $E$ as well as polar coordinate $\theta$ in 2D (resp. spherical coordinates $\theta,\phi$ in 3D) only. 
Let us consider T-function $\Tu^{(m)}$ of maximum degree $m$ in polar (resp. spherical) coordinate $r$, or, equivalently, in cartesian coordinates $(x,y)$ (resp. $(x,y,z)$) in 2D (resp. 3D), with $m \geqslant 1$ (\ie discarding rigid body motions). Corresponding stress T-function $\K \: \defo(\Tu^{(m)})$ can be derived in a polynomial form of degree $m-1$; it follows that energy $e^{(m)} \equiv \Tr\big[\defo(\Tu^{(m)}) \: \K \: \defo(\Tu^{(m)})\big]$ can be put in the following polynomial form:
\begin{align}
\begin{cases}
e^{(m)}(r,\theta) = r^{2(m-1)} \Psi_2(\theta) & \text{in 2D};\\
e^{(m)}(r,\theta,\phi) = r^{2(m-1)} \Psi_3(\theta,\phi) & \text{in 3D},
\end{cases}
\end{align}
where function $\Psi_2$ (resp. $\Psi_3$) depends on material parameters $\nu$, $E$ as well as polar coordinate $\theta$ in 2D (resp. spherical coordinates $\theta,\phi$ in 3D) only. Thus, after some straightforward computations, corresponding Rayleigh quotient $\Rcb_{\labar}(\Tu^{(m)})$ reads as:
\begin{align}
\Rcb_{\labar}(\Tu^{(m)}) = \frac{\labs{\Tu^{(m)}}^2_{u, \domlabar}}{\lnorm{\Tu^{(m)}}^2_{u, \omlabar}} = \begin{cases}
2m & \text{in 2D}; \\
2m+1 & \text{in 3D}.
\end{cases}
\end{align}

Therefore, any displacement T-function $\Tu$ associated to non-vanishing strain satisfies $\displaystyle\Rcb_{\labar}(\Tu) \geqslant 2$ in 2D and $\displaystyle\Rcb_{\labar}(\Tu) \geqslant 3$ in 3D. Besides, let us note that T-functions are also orthogonal with respect to both inner products $\lscalprodp{\bullet}{\circ}_{u, \domlabar}$ and $\lscalproda{\bullet}{\circ}_{u, \omlabar}$ related to Rayleigh quotient $\Rcb_{\labar}$.

Finally, given that any function $\vu \in \Vcb$ (\ie any homogeneous solution of Lam\'e-Navier equations) can be approximated as a linear combination of T-functions, one can show that:
\begin{align}
\min_{\substack{ \vu \in \Vcb }} \Rcb_{\labar}(\vu) = \begin{cases}
2 & \text{for the two-dimensional case of a circular shape domain;} \\
3 & \text{for the three-dimensional case of a spherical shape domain.}
\end{cases}
\end{align}

\bibliographystyle{unsrt}

\begin{thebibliography}{10}

\bibitem{Bab84}
I.~Babu\v{s}ka and A.~Miller.
\newblock {The post-processing approach in the finite element method - Part 2:
  The calculation of stress intensity factors}.
\newblock {\em International Journal for Numerical Methods in Engineering},
  20(6):1111--1129, 1984.

\bibitem{Par97}
M.~Paraschivoiu, J.~Peraire, and A.~T. Patera.
\newblock {A posteriori finite element bounds for linear-functional outputs of
  elliptic partial differential equations}.
\newblock {\em Computer Methods in Applied Mechanics and Engineering},
  150(1-4):289--312, 1997.

\bibitem{Ran97}
R.~Rannacher and F.-T. Suttmeier.
\newblock {A feed-back approach to error control in finite element methods:
  application to linear elasticity}.
\newblock {\em Computational Mechanics}, 19:434--446, 1997.
\newblock 10.1007/s004660050191.

\bibitem{Cir98}
F.~Cirak and E.~Ramm.
\newblock {A posteriori error estimation and adaptivity for linear elasticity
  using the reciprocal theorem}.
\newblock {\em Computer Methods in Applied Mechanics and Engineering},
  156:351--362, 1998.

\bibitem{Pru99}
S.~Prudhomme and J.~T. Oden.
\newblock {On goal-oriented error estimation for elliptic problems: application
  to the control of pointwise errors}.
\newblock {\em Computer Methods in Applied Mechanics and Engineering},
  176(1-4):313--331, 1999.

\bibitem{Lad99a}
P.~Ladev\`eze, P.~Rougeot, P.~Blanchard, and J.~P. Moreau.
\newblock {Local error estimators for finite element linear analysis}.
\newblock {\em Computer Methods in Applied Mechanics and Engineering},
  176(1-4):231--246, 1999.

\bibitem{Str00}
T.~Strouboulis, I.~Babu\v{s}ka, D.~K. Datta, K.~Copps, and S.~K. Gangaraj.
\newblock {A posteriori estimation and adaptive control of the error in the
  quantity of interest. Part I: A posteriori estimation of the error in the von
  Mises stress and the stress intensity factor}.
\newblock {\em Computer Methods in Applied Mechanics and Engineering},
  181(1-3):261--294, 2000.

\bibitem{Bec01}
R.~Becker and R.~Rannacher.
\newblock {An optimal control approach to a posteriori error estimation in
  finite element methods}.
\newblock {\em Acta Numerica, A. Isereles (ed.), Cambridge University Press},
  10:1--120, 2001.

\bibitem{Par06bis}
N.~Par\'es, J.~Bonet, A.~Huerta, and J.~Peraire.
\newblock {The computation of bounds for linear-functional outputs of weak
  solutions to the two-dimensional elasticity equations}.
\newblock {\em Computer Methods in Applied Mechanics and Engineering},
  195(4-6):406--429, 2006.

\bibitem{Gal06}
L.~Gallimard and J.~Panetier.
\newblock {Error estimation of stress intensity factors for mixed-mode cracks}.
\newblock {\em International Journal for Numerical Methods in Engineering},
  68(3):299--316, 2006.

\bibitem{Cha07}
L.~Chamoin and P.~Ladev\`eze.
\newblock {Bounds on history-dependent or independent local quantities in
  viscoelasticity problems solved by approximate methods}.
\newblock {\em International Journal for Numerical Methods in Engineering},
  71(12):1387--1411, 2007.

\bibitem{Cha08}
L.~Chamoin and P.~Ladev\`eze.
\newblock {A non-intrusive method for the calculation of strict and efficient
  bounds of calculated outputs of interest in linear viscoelasticity problems}.
\newblock {\em Computer Methods in Applied Mechanics and Engineering},
  197(9-12):994--1014, 2008.

\bibitem{Lad08}
P.~Ladev\`eze.
\newblock {Strict upper error bounds on computed outputs of interest in
  computational structural mechanics}.
\newblock {\em Computational Mechanics}, 42(2):271--286, 2008.

\bibitem{Lad09}
P.~Ladev\`eze and J.~Waeytens.
\newblock {Model verification in dynamics through strict upper error bounds}.
\newblock {\em Computer Methods in Applied Mechanics and Engineering},
  198(21-26):1775--1784, 2009.

\bibitem{Pan10}
J.~Panetier, P.~Ladev\`eze, and L.~Chamoin.
\newblock {Strict and effective bounds in goal-oriented error estimation
  applied to fracture mechanics problems solved with XFEM}.
\newblock {\em International Journal for Numerical Methods in Engineering},
  81(6):671--700, 2010.

\bibitem{Lad10}
P.~Ladev\`eze and L.~Chamoin.
\newblock {Calculation of strict error bounds for finite element approximations
  of non-linear pointwise quantities of interest}.
\newblock {\em International Journal for Numerical Methods in Engineering},
  84(13):1638--1664, 2010.

\bibitem{Cha09}
L.~Chamoin and P.~Ladev\`eze.
\newblock {Strict and practical bounds through a non-intrusive and
  goal-oriented error estimation method for linear viscoelasticity problems}.
\newblock {\em Finite Elements in Analysis and Design}, 45(4):251--262, 2009.

\bibitem{Lad04}
P.~Ladev\`eze and J.~P. Pelle.
\newblock {\em {Mastering Calculations in Linear and Nonlinear Mechanics}}.
\newblock Springer, New York, 2004.

\bibitem{Par06}
N.~Par\'es, P.~D\'iez, and A.~Huerta.
\newblock {Subdomain-based flux-free a posteriori error estimators}.
\newblock {\em Computer Methods in Applied Mechanics and Engineering},
  195(4-6):297--323, 2006.

\bibitem{Moi09}
J.~P. Moitinho~de Almeida and E.~A.~W. Maunder.
\newblock {Recovery of equilibrium on star patches using a partition of unity
  technique}.
\newblock {\em International Journal for Numerical Methods in Engineering},
  79(12):1493--1516, 2009.

\bibitem{Cot09}
R.~Cottereau, P.~D\'iez, and A.~Huerta.
\newblock {Strict error bounds for linear solid mechanics problems using a
  subdomain-based flux-free method}.
\newblock {\em Computational Mechanics}, 44(4):533--547, 2009.

\bibitem{Gal09}
L.~Gallimard.
\newblock {A constitutive relation error estimator based on traction-free
  recovery of the equilibrated stress}.
\newblock {\em International Journal for Numerical Methods in Engineering},
  78(4):460--482, 2009.

\bibitem{Lad10bis}
P.~Ladev\`eze, L.~Chamoin, and E.~Florentin.
\newblock {A new non-intrusive technique for the construction of admissible
  stress fields in model verification}.
\newblock {\em Computer Methods in Applied Mechanics and Engineering},
  199:766--777, 2010.

\bibitem{Ple11}
F.~Pled, L.~Chamoin, and P.~Ladev\`eze.
\newblock {On the techniques for constructing admissible stress fields in model
  verification: Performances on engineering examples}.
\newblock {\em International Journal for Numerical Methods in Engineering},
  88(5):409--441, 2011.

\bibitem{Ple12}
F.~Pled, L.~Chamoin, and P.~Ladev\`eze.
\newblock {An enhanced method with local energy minimization for the robust a
  posteriori construction of equilibrated stress fields in finite element
  analyses}.
\newblock {\em Computational Mechanics}, 49(3):357--378, 2012.

\bibitem{Str00bis}
T.~Strouboulis, I.~Babu\v{s}ka, and K.~Copps.
\newblock {The design and analysis of the Generalized Finite Element Method}.
\newblock {\em Computer Methods in Applied Mechanics and Engineering},
  181(1-3):43--69, 2000.

\bibitem{Ohn01}
S.~Ohnimus, E.~Stein, and E.~Walhorn.
\newblock {Local error estimates of FEM for displacements and stresses in
  linear elasticity by solving local Neumann problems}.
\newblock {\em International Journal for Numerical Methods in Engineering},
  52(7):727--746, 2001.

\bibitem{Lad92}
P.~Ladev\`eze, P.~Marin, J.~P. Pelle, and J.~L. Gastine.
\newblock {Accuracy and optimal meshes in finite element computation for nearly
  incompressible materials}.
\newblock {\em Computer Methods in Applied Mechanics and Engineering},
  94(3):303--315, 1992.

\bibitem{Pra47}
W.~Prager and J.~L. Synge.
\newblock {Approximations in elasticity based on the concept of functions
  spaces}.
\newblock {\em Quarterly of Applied Mathematics}, 5:261--269, 1947.

\bibitem{Ste02}
M.~W. Steklov.
\newblock {Sur les probl\`emes fondamentaux de la physique math\'ematique}.
\newblock {\em Annales Scientifiques de l'\'Ecole Normale Sup\'erieure},
  19:455--490, 1902.

\bibitem{Lad78}
J.~Ladev\`eze and P.~Ladev\`eze.
\newblock {Upper bounds of the Poincar\'e constant in elasticity (in french)}.
\newblock {\em Rendiconti Accademia Nazionale dei Lincei, Serie VIII},
  64:548--556, 1978.

\bibitem{Hoc93}
C.~Hochard, P.~Ladev\`eze, and L.~Proslier.
\newblock {A simplified analysis of elastic structures}.
\newblock {\em European Journal of Mechanics Solids}, 12:509--535, 1993.

\bibitem{Hoc92}
C.~Hochard and L.~Proslier.
\newblock {A simplified analysis of plate structures using Trefftz-functions}.
\newblock {\em International Journal for Numerical Methods in Engineering},
  34(1):179--195, 1992.

\bibitem{Min36}
R.~D. Mindlin.
\newblock {Force at a point in the interior of a semi-infinite solid}.
\newblock {\em Journal of Physics}, 7(5):195--202, 1936.

\bibitem{Min50}
R.~D. Mindlin and D.~H. Cheng.
\newblock {Nuclei of strain in the semi-infinite solid}.
\newblock {\em Journal of Applied Physics}, 21(9):926--930, 1950.

\bibitem{Cou53}
R.~Courant and D.~Hilbert.
\newblock {\em {Methods of Mathematical Physics}}.
\newblock Interscience publishers, New York, 1953.

\bibitem{Gra75}
K.F. Graff.
\newblock {\em {Wave Motion in Elastic Solids}}.
\newblock Dover, New York, 1975.

\bibitem{Vij87}
S.~Vijayakumar and D.~E. Cormack.
\newblock {Green's functions for the biharmonic equation: bonded elastic
  media}.
\newblock {\em SIAM Journal on Applied Mathematics}, 47(5):982--997, 1987.

\bibitem{Ste76}
M.~Stern, E.~B. Becker, and R.~S. Dunham.
\newblock {A contour integral computation of mixed-mode stress intensity
  factors}.
\newblock {\em International Journal of Fracture}, 12:359--368, 1976.
\newblock 10.1007/BF00032831.

\bibitem{Pan09}
J.~Panetier, P.~Ladev\`eze, and F.~Louf.
\newblock {Strict bounds for computed stress intensity factors}.
\newblock {\em Computers \& Structures}, 87(15-16):1015--1021, 2009.

\bibitem{Tre26}
E.~Trefftz.
\newblock {Ein Gegenst{\"u}ck zum Ritzschen Verfahren}.
\newblock {\em Proceedings of the 2nd International Congress of Applied
  Mechanics}, pages 131--137, 1926.

\bibitem{Kit95}
Eisuke Kita and Norio Kamiya.
\newblock {Trefftz method: an overview}.
\newblock {\em Advances in Engineering Software}, 24(1-3):3--12, 1995.

\bibitem{Zie95}
A.~P Zieli\'nski.
\newblock {On trial functions applied in the generalized Trefftz method}.
\newblock {\em Advances in Engineering Software}, 24(1-3):147--155, 1995.

\bibitem{Kom04}
V.~Kompi{\v s}, M.~Toma, M.~{\v Z}mind{\'a}k, and M.~Handrik.
\newblock {Use of Trefftz functions in non-linear BEM/FEM}.
\newblock {\em Computers \& Structures}, 82(27):2351--2360, 2004.

\bibitem{Kom05}
V.~Kompi{\v s}, M.~{\v S}tiavnick{\'y}, M.~Kompi{\v s}, and M.~{\v
  Z}mind{\'a}k.
\newblock {Trefftz interpolation based multi-domain boundary point method}.
\newblock {\em Engineering Analysis with Boundary Elements}, 29(4):391--396,
  2005.

\bibitem{Zie85}
A.~P Zieli\'nski and O.~C. Zienkiewicz.
\newblock {Generalized finite element analysis with T-complete boundary
  solution functions}.
\newblock {\em International Journal for Numerical Methods in Engineering},
  21(3):509--528, 1985.

\bibitem{Sla00}
J.~Sladek, V.~Sladek, V.~Kompi{\v s}, and R.~Van~Keer.
\newblock {Application of multi-region Trefftz method in elasticity}.
\newblock {\em CMES- Computer Modeling in Engineering \& Sciences}, 1(4):1--8,
  2000.

\bibitem{Sla02}
J.~Sladek and V.~Sladek.
\newblock {A Trefftz function approximation in local boundary integral
  equations}.
\newblock {\em Computational Mechanics}, 28(3):212--219, 2002.

\end{thebibliography}

\end{document}